\documentclass[oneside, 12pt]{amsart}
\usepackage{amscd,amssymb,enumerate, amsmath,mathrsfs, amsfonts}
\usepackage{url}
\usepackage{tikz}
\usetikzlibrary{decorations.pathreplacing}
\usepackage{hyperref}

\usepackage[utf8]{inputenc}

\newcommand\mylabel[1]{\label{#1}}

\newtheorem{theorem}{Theorem}[section]

\newtheorem{lemma}[theorem]{Lemma}
\newtheorem{proposition}[theorem]{Proposition}

\theoremstyle{definition}

\newtheorem{example}[theorem]{Example}

\newtheorem{remark}[theorem]{Remark}
\newtheorem{emp}[theorem]{}
\newtheorem*{acknowledgement}{Acknowledgement}

\theoremstyle{remark}

\DeclareFontFamily{U}{wncy}{}
\DeclareFontShape{U}{wncy}{m}{n}{<->wncyr10}{}
\DeclareSymbolFont{mcy}{U}{wncy}{m}{n}
\DeclareMathSymbol{\Sh}{\mathord}{mcy}{"58}

\newcommand{\ideala}    {\mathfrak{a}}
\newcommand{\idealb}    {\mathfrak{b}}

\newcommand  {\maxid}   {\mathfrak{m}}
\newcommand  {\primid}  {\mathfrak{p}}

\newcommand{\YYY}{Z}
\newcommand{\SB}{B}
\newcommand{\SC}{C}
\newcommand{\PeskinA}{A_0}
\newcommand{\R}{B}

\newcommand{\ZZ}	{\mathbb{Z}}
\newcommand{\QQ}	{\mathbb{Q}}
\newcommand{\RR}	{\mathbb{R}}

\newcommand{\FF}	{\mathbb{F}}

\newcommand{\PP}	{\mathbb{P}}

\newcommand{\GG}	{\mathbb{G}}

\newcommand  {\shF}     {\mathscr{F}}

\newcommand  {\calO}     {\mathcal{O}}

\newcommand  {\Bl}     {\operatorname{Bl}}

\newcommand  {\Cl}      {\operatorname{Cl}}

\newcommand  {\depth}   {\operatorname{depth}}

\newcommand  {\Frac}    {\operatorname{Frac}}

\newcommand  {\Hom}     {\operatorname{Hom}}

\newcommand  {\lra}     {\longrightarrow}

\newcommand  {\Mat}     {\operatorname{Mat}}

\renewcommand{\O}       {\mathscr{O}}

\newcommand  {\Proj}    {\operatorname{Proj}}

\newcommand  {\quadand} {\quad\text{and}\quad}

\newcommand  {\ra}      {\rightarrow}

\newcommand  {\rank}    {\operatorname{rank}}
\newcommand  {\red}     {{\operatorname{red}}}

\newcommand  {\Spec}    {\operatorname{Spec}}

\newcommand  {\lcm}  	{\operatorname{lcm}}

\newcommand{\ttt}{\mu}

\newcommand{\AAA}{\widehat{A_\maxid}}

\def\mydate{\number\day\space\ifcase\month \or January\or February\or March\or 
April\or May\or June\or July\or
August\or September\or October\or November\or December\fi \space\number\year}

\DeclareFontFamily{U}{wncy}{}
\DeclareFontShape{U}{wncy}{m}{n}{<->wncyr10}{}
\DeclareSymbolFont{mcy}{U}{wncy}{m}{n}
\DeclareMathSymbol{\Sh}{\mathord}{mcy}{"58}

\if false
\numberwithin{equation}{section}

\makeatletter
    \let\c@equation\c@thm
    \let\c@figure\c@thm
   \let\c@table\c@thm
\makeatother
\fi

\numberwithin{equation}{section}

\begin{document}

\title[Discriminant groups of   quotient singularities]
      {Discriminant groups of wild cyclic quotient singularities} 

\author[Dino Lorenzini]{Dino Lorenzini}
\address{Department of Mathematics, University of Georgia, Athens, GA 30602, USA}
\curraddr{}
\email{lorenzin@uga.edu}

\author[Stefan Schr\"oer]{Stefan Schr\"oer}
\address{Mathematisches Institut, Heinrich-Heine-Universit\"at,
40204 D\"usseldorf, Germany}
\curraddr{}
\email{schroeer@math.uni-duesseldorf.de}


\dedicatory{21 January 2021} 

\begin{abstract}
Let $p$ be prime. We describe explicitly the resolution of singularities of several families
of wild ${\mathbb Z}/p{\mathbb Z}$-quotient singularities in dimension two, 
including   families that generalize the quotient singularities of type $E_6$, $E_7$, and $E_8$ from $p=2$ to arbitrary characteristics.
We prove that for odd primes, any power of $p$ can appear as the determinant of the intersection matrix of a wild ${\mathbb Z}/p{\mathbb Z}$-quotient singularity.
We also provide evidence towards the conjecture that in this situation one may choose the wild action to be ramified precisely at the origin.
\end{abstract}

\maketitle
\tableofcontents


\section*{Introduction}
\mylabel{Introduction}

The goal of this paper is to study \emph{wild quotient singularities} which arise from 
actions of $G:=\ZZ/p\ZZ$ on the formal power series ring $A=k[[u,v]]$ when $k$ is an algebraically closed field
of characteristic $p>0$.  Here the term ``wild'' refers to the fact that the order of the group $G$
is not coprime to the characteristic exponent of the ground field $k$.
The resulting quotient singularity is the  ring of invariants $A^G$
or, more precisely, the closed point of $\Spec(A^G)$.

Let $X \to \Spec(A^G)$ be a resolution of the singularity. Let $C_i$, $i=1,\dots, r$, denote the irreducible components
of the exceptional divisor, and form the   \emph{intersection matrix}
$$
N:=((C_i\cdot C_j)_X)_{1\leq i,j\leq r}\in \Mat_r(\ZZ).
$$
This matrix is negative-definite. The \emph{discriminant group} $\Phi_N:=\ZZ^r/N\ZZ^r$
attached to $N$ is a finite group of order $|\text{det}(N)|$, independent of the resolution. The group $\Phi_N$  appears 
as a natural quotient of the class group $\Cl(A^G)$
(\ref{remark.classgroup}).
Attached to the resolution is its \emph{dual graph} $\Gamma_N$,  with vertices $v_1,\dots, v_r$, where $v_i$ and $v_j$ are linked by 
$(C_i\cdot C_j)_X$ distinct edges when $i \neq j$.
Our ultimate, long term,  goal is to characterize the  intersection matrices $N$, discriminant groups $\Phi_N$,
and dual graphs $\Gamma_N$, that can arise from such wild quotient singularities.

The {\it fixed point scheme} of the action of $G$ on $\Spec A$ is defined by the ideal $I:=(\sigma(a)-a \mid a \in A, \sigma \in G)$.
We say that the action is {\it ramified precisely at the origin} if the radical of $I$ is the maximal ideal $(u,v)$; 
in this case, the closed point of $\Spec(A^G)$ is singular. Otherwise, we say that the action is {\it ramified in codimension $1$}. 
When $I$ is principal,  $A^G$ is regular (\cite[Theorem 2]{K-L}), and when $A^G$ is regular, $I$ is conjectured to be principal \cite[Conjecture 9]{K-L}.

It is known that when the exceptional divisor has smooth components with normal crossings, all components $C_i$ are smooth projective lines
and   the dual graph $\Gamma_N$ is a tree \cite[Theorem 2.8]{Lorenzini 2013}.
It is also known that the discriminant group $\Phi_N$ is an elementary abelian $p$-group \cite[Theorem 2.6]{Lorenzini 2013}, so that in particular
we may write 
$$
|\Phi_N|=|\text{det}(N)|=p^s
$$
for some integer $s\geq 0$.
In this article, we consider which exponents $s\geq 0$ can arise in this way.
By studying diagonal actions on products of curves, the first author \cite[Theorem 3.15]{Lorenzini 2018}
produced wild quotient singularities with  $|\Phi_N|=p^s$ for all exponents $s\geq 2$ with $s\not\equiv 1$ modulo $p$.
Mitsui \cite{Mit} later explicitly resolved all wild quotient singularities arising from product of curves, and showed that the previous list is the complete list of exponents arising from product of curves.
The {\it missing exponents} are then $s=0$, as well as all $s$ with $s\equiv 1 \mod p$.

\begin{emp} \label{conjecture}
\emph{We conjecture that for $p$ odd, all exponents $s \geq 0$ arise in this way from wild $\ZZ/p\ZZ$-quotient singularities associated with an action that is ramified precisely at the origin}.
\end{emp}

In this article, we prove this conjecture for $s=0$ and $s=1$ by explicitly resolving certain wild quotient singularities of independent interest.
We also exhibit singularities as in the conjecture that are likely to produce a group $\Phi_N$ with $|\Phi_N|=p^s$ for all other missing values $s>1$ (\ref{p.odd.conjecture}). When the condition that the action be ramified precisely at the origin is relaxed, we can prove the following result.

\medskip
\noindent
{\bf Theorem (see \ref{thm.exhibit}}).
{\it For $p$ odd, all missing exponents $s \geq 0$ arise from wild $\ZZ/p\ZZ$-quotient singularities associated with an action that is ramified in codimension $1$.}
\medskip

Let $c,d,e \geq 2$ be integers. Recall that the equation $x^c+y^d+z^e=0$ is said to define a {\it Brieskorn surface singularity}.
The missing exponents $s$ are exhibited to arise from wild quotient singularities
with the help of well-chosen Brieskorn singularities, as in our next theorem.

\medskip
\noindent
{\bf Theorem (see \ref{thm.BrieskornResolution} and \ref{Brieskorn quotient sing}).}
{\it Let $B:=k[[x,y,z]]/(z^p+x^c+y^d)$. Assume that $p$ does not divide $cd$. Let $g:=\gcd(c,d)$.  Any resolution of $\Spec B$ has an intersection matrix whose associated discriminant group 
has order $p^{g-1}$ and is killed by $p$. When $c=pm+1$ and $d=pn+1$ for some $m,n \geq 1$, then $\Spec B$ is a wild $\ZZ/p\ZZ$-quotient singularity.
}
\medskip

The resolutions of the Brieskorn singularities in the previous theorem are found in Theorem \ref{thm.BrieskornResolution}, and coincide with 
the known resolutions in characteristic $0$ (\cite[Theorem, page 232]{H-J}, \cite{O-W}).  
The theorem is valid when $p=2$, but in this case, the order $p^{g-1}$ is always an even power of $2$, and thus provides no examples of missing odd exponents. The theorem shows that when $p=2$ and $\gcd(p,cd)=1$, all  singularities $z^p+x^c+y^d=0$ are wild $\ZZ/p\ZZ$-quotient singularities.
It would be of interest to determine whether this fails to be the case when $p>2$. 

Let now $C_n$ denote the $n$-th Catalan number, and let $p \geq 3$.
To produce singularities associated with an action that is ramified precisely at the origin and which have  $|\Phi_N|=p$, we expand on the work of Peskin \cite{Peskin 1983} and consider the ring $B_\ttt:=k[[x,y,z]]/(h)$, where $\mu \in k[y]$ and  
\begin{equation*}
h:=z^p +   2y^{p+1} - x^2 + \sum_{n=2}^{(p+1)/2} (-1)^nC_{n-1} (\ttt y)^{2p-2n}z^n.
\end{equation*}
When $\ttt=1$, this equation defines a wild quotient singularity that can be regarded as an analogue of the $E_6^1$-singularity (notation as in Artin's classification in \cite{Artin 1977}). We compute explicitly its resolution in our next theorem.
When $p=3$, the graph $\Gamma_N$ below reduces to the Dynkin diagram $E_6$. When drawing a dual graph, we adopt in this article the usual convention that 
a vertex is adorned with the associated self-intersection number, unless this self-intersection number is $-2$, in which case it is suppressed.

\medskip
\noindent
{\bf Theorem (see \ref{resolution peskin}}).
{\it Let $p$ be an odd prime. Let $B_\ttt$ be as above. Then $\Spec B_\ttt$ has a resolution independent of $\ttt$, with dual graph $\Gamma_N$ of the following form
$$
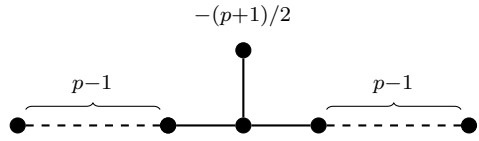
\begin{tikzpicture}
[node distance=1cm, font=\small] 
\tikzstyle{vertex}=[circle, draw, fill, inner sep=0mm, minimum size=1.1ex]
\node[vertex]	(v1)  	at (0,0) 	[label=below:{}] 		{};
\node[]	(dummy1)		[right of=v1, label=above:{}]	{};
\node[vertex]	(v2)			[right of=dummy1, label=below:{}]	{};
\node[vertex]	(v2)			[right of=dummy1, label=below:{}]	{};
\draw [decorate,decoration={brace, raise=6pt}] (v1)--(v2) node [black, midway,xshift=-0cm,yshift=0.5cm] {$^{p-1}$};
\node[vertex]	(v3)			[right of=v2, label=below:{}]	{};
\node[vertex]	(v)			[above of=v3,  label=above:{ $^{-(p+1)/2}$}]	{};
\node[vertex]	(v4)			[right of=v3, label=below:{}]	{};
\node[]	(dummy2)		[right of=v4, label=above:{}]	{};
\node[vertex]	(v5)			[right of=dummy2, label=below:{}]	{};
\draw [decorate,decoration={brace, raise=6pt}] (v4)--(v5) node [black, midway,xshift=-0cm,yshift=0.5cm] {$^{p-1}$};
\draw [thick, dashed] (v1)--(v2);
\draw [thick] (v2)--(v3)--(v4);
\draw [thick] (v)--(v3);
\draw[thick, dashed] (v4)--(v5);
\end{tikzpicture}
$$
The associated discriminant group $\Phi_N$ has order $p$.
}

\medskip
\begin{emp} \label{moderately ramified}
To treat the case where $\Phi_N$ is the trivial group in Conjecture \ref{conjecture}, we use a family of hypersurface singularities introduced in  \cite{LS1} and 
which is of independent interest. 
Fix a system of parameters $a,b$ in $k[[x,y]]$. Let $\mu \in k[[x,y]]$, and consider the equation
\begin{equation}
\label{special normal form}
z^p - (\mu ab)^{p-1}z - a^py + b^px=0,
\end{equation} 
and the associated ring 
$$
\R_{\mu}=\R:=k[[x,y,z]]/(z^p - (\mu ab)^{p-1}z - a^py + b^px).
$$

\smallskip
(a) {\it Assume that $\mu$ is a unit in $ k[[x,y]]$.}
It is shown in \cite{LS1}, 7.1, that $\R$ is isomorphic to the ring of invariants $A^G$ of an explicit wild action of ${\mathbb Z}/p{\mathbb Z}$  on $A=k[[u,v]]$ ramified precisely at the origin.
More precisely, after identifying $A$ with the ring 
$$k[[x,y]][u,v]/(u^p-(\mu a)^{p-1}u -x, v^p-(\mu b)^{p-1}v -y),$$
the action is determined by the automorphism $\sigma$ with $\sigma(u)= u+\mu a$ and $\sigma(v)=v+\mu b$.
The morphism $\Spec A \to \Spec A^G$ is ramified only at the maximal ideal ${\maxid}$, and we find that the \'etale fundamental group 
$\pi_1^{\rm loc}(A^G)$ 
of the punctured
spectrum $U:=\Spec A^G \setminus \{\maxid\}$ is isomorphic to ${\mathbb Z}/p{\mathbb Z}$. Such actions are called {\it moderately ramified}  
in \cite{LS1}, and we refer the reader to \cite{LS1} for further information on these actions. 
 
(b) {\it Assume that $\mu \neq 0$, and that it is coprime to both $a$ and $b$.} Then $\R$ is again isomorphic 
to the ring of invariants $A^G$ for the action on $A:=k[[u,v]]$ described above. 
However, in this case the morphism $\Spec A \to \Spec A^G$ is ramified in codimension $1$ and the 
group $\pi_1^{\rm loc}(A^G)$ is trivial.

In this article, we restrict our attention to the case
where $a=y^n$ and $b=x^m$. The case $\mu=0$ is then also of interest:

(c) {\it Assume that $\mu=0$, with $a=y^n $ and $b=x^m$.} The resulting hypersurface is a Brieskorn singularity of type $z^p-y^{pn+1}+x^{pm+1}$. 

\smallskip
In the specialized case where $a=y^n $ and $b=x^m$, preliminary computations with Magma \cite{Magma} and Singular \cite{Singular}
show that the resolution of singularities in all three cases above has the same combinatorial type, independent of $\mu$. 
We prove that this is indeed the case in two instances in this article, when  $a=y $ and $b=x$ in Theorem \ref{a=y, b=x}, and when $a=y^2 $ and $b=x$ in 
Theorem \ref{antidiagonal E8}.
In the latter case,  Artin \cite{Artin 1977} (see also \cite{Peskin 1980}) shows when $p=2$
that the values $\mu=0$, $\mu=1$, and $\mu =y$, produce the rational double points
$E_8^0$, $E_8^2$, and $E_8^1$, respectively.  These singularities are not isomorphic but  have the same resolution graph, the Dynkin diagram $E_8$. 
Our generalization  of these singularities to any odd prime $p$ has a resolution with the following dual graph.
\end{emp}

\noindent
{\bf Theorem (see \ref{antidiagonal E8}).} {\it Let $p$ be an odd prime. Let $B_\ttt$ be as in {\rm \ref{moderately ramified}}. Assume that $a=y^2 $ and $b=x$. 
Then $\Spec \R_{\mu}$ has a resolution of singularities independent of $\ttt$, 
with dual graph $\Gamma_N$

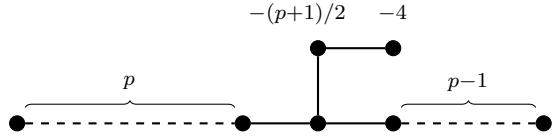
\begin{equation*}
\begin{gathered}
\begin{tikzpicture}
[node distance=1cm, font=\small]
\tikzstyle{vertex}=[circle, draw, fill,  inner sep=0mm, minimum size=1.1ex]
\node[vertex]	(E1)  	at (0,0) 	[label=below:{}]               {};
\node[]	        (E11)  	         	[right of=E1,label=above:{}] {};
\node[]	        (E2)  	         	[right of=E11,label=above:{}]               {};
\node[vertex]	(E3)  	         	[right of=E2,label=below:{}]               {};

\node[vertex]	(D)			[right of=E3, label=right:{ }]    {};

\node[vertex]	(E4)			[above of=D, label=above:{$^{ -(p+1)/2  \quad \ }$}]    {};
\node[vertex]	(E5)			[right of=E4, label=above:{$^{-4}$}]    {};

\node[vertex]	(E6)			[right of=D, label=below:{}]    {};
\node[]	        (E7)			[right  of=E6, label=above:{}]    {};
\node[vertex]	(E8)			[right  of=E7, label=below:{}]    {};
\draw [decorate,decoration={brace, raise=6pt}] (E6)--(E8) node [black, midway,xshift=-0cm,yshift=0.5cm] {$^{p-1}$};

\draw [decorate,decoration={brace, raise=6pt}] (E1)--(E3) node [black, midway,xshift=-0cm,yshift=0.5cm] {$^p$};
 
\draw [thick,dashed] (E1)--(E3);
\draw [thick,dashed] (E6)--(E8);

\draw [thick] (E3)--(D)--(E4)--(E5);
\draw [thick] (D)--(E6);

\end{tikzpicture}
\end{gathered}
\end{equation*}
The associated discriminant group $\Phi_N$ is trivial.}

\begin{emp} \label{p.odd.conjecture}
Let $p$ be odd. Recall that when $\mu=1$, the associated quotient singularity $\Spec B_{\mu=1}$ is induced by an action that is ramified precisely at the origin. 
It is likely that by varying the exponents $m$ and $n$ in $a=y^n$ and $b=x^m$,
one will obtain examples of resolutions of $\Spec B_{\mu=1}$ with associated discriminant group $\Phi_N$ 
of order $p^s$ for any power $s$ with $s \not \equiv -1 \mod p$. In particular, we exhibit in \ref{alls} the appropriate exponents $m$ and $n$
that would cover all remaining open cases in our conjecture \ref{conjecture} (that is, all values of $s$ with $s \equiv 1 \mod p$). 
\end{emp}

Peskin's singularity with $\mu=1$ introduced above, and all the singularities considered in \cite{Lorenzini 2018} or \cite{Mit}, are also induced by an action that is ramified precisely at the origin. When $p=2$, none of the known explicit 
resolutions for  examples in these classes of singularities
 produce an associated discriminant group $\Phi_N$ with order $2^s$ and $s$ {\it odd}. This lack of examples might 
 indicate that there is a serious obstruction to exhibiting such examples. It is natural to wonder whether such examples in fact do not exist for actions ramified precisely at the origin.
 
Let $p=2$. The Dynkin diagram $E_7$, with discriminant group $\Phi_{E_7}$ of order $2$, might be the most ubiquitous  graph 
with  discriminant group of order $2^s$ with $s$ odd. Many other such examples are exhibited in \ref{ex.star}. Artin \cite{Artin 1977} showed that there exists a wild ${\mathbb Z}/2{\mathbb Z}$-action 
on $A=k[[u,v]]$, ramified in codimension $1$, such that $\Spec A^{{\mathbb Z}/2{\mathbb Z}}$ 
is a rational double point of type $E_7$. 
He also showed that any such surface  singularity 
must have a trivial local fundamental group. In other words, there cannot exist a wild ${\mathbb Z}/2{\mathbb Z}$-action 
on $A=k[[u,v]]$, ramified precisely at the origin, such that $\Spec A^{{\mathbb Z}/2{\mathbb Z}}$ has a resolution of combinatorial type $E_7$.

Inspired by Artin's considerations, we define in Section \ref{analogues E7}
some explicit wild ${\mathbb Z}/p{\mathbb Z}$-actions 
on $A=k[[u,v]]$ ramified in codimension $1$. When $p=2$, we exhibit for each $s$ {\it odd} an explicit example conjectured to have discriminant group of order $2^s$.  In Section \ref{The Ap-1 singularity},  for any prime $p$,
we exhibit a wild ${\mathbb Z}/p{\mathbb Z}$-action  
on $A=k[[u,v]]$ ramified in codimension $1$  which results in an $A_{p-1}$-singularity. 

\medskip
\noindent
{\bf Theorem (see \ref{Ap-1}).} {\it Let $k$ be a field of characteristic $p>0$. Let $A:=k[[u,v]]$. Then there exists an automorphism $\sigma: A \to A$
of order $p$ such that $\Spec A^{\langle \sigma \rangle}$ 
is a rational double point of type $A_{p-1}$, which has discriminant group $\Phi_{A_{p-1}}$ of order $p$. 
Any such automorphism induces a morphism $\Spec A \to \Spec A^{\langle \sigma \rangle}$ that must be ramified in codimension $1$.
}
\medskip

It is natural to wonder whether the same result holds for any {\it Hirzebruch--Jung chain}  whose discriminant group has order $p$ (definition recalled in \ref{chain}).
The last statement in the above theorem follows from a result of Ito and Schr\"oer \cite{IS}, which states that 
if the action is ramified precisely at the origin, then the resolution of the resulting quotient singularity has a dual graph $\Gamma_N$ 
which must have a vertex of valency at least $3$. 

Artin shows in \cite{Artin 1975} that in characteristic $p=2$,  all wild quotient singularities $A^G$ with $\Spec A \to \Spec A^G$ ramified precisely at the origin
can be described by an equation of the form \eqref{special normal form} with $\mu=1$. 
In particular, all such singularities are complete intersection.
We show in Proposition \ref{pro.p=2completeintersection} that when $p=2$, any wild quotient singularity $A^G$ is a  complete intersection, even when 
$\Spec A \to \Spec A^G$ ramifies in codimension $1$. 
When $A^G $ is a complete intersection, it is then also Gorenstein, with an intersection matrix 
which is {\it numerically Gorenstein}. The purely linear algebraic definition of numerically Gorenstein 
is recalled in \ref{NumericallyGor}, and it is natural to wonder whether this condition imposes a new restriction on intersection matrices associated with ${\mathbb Z}/2{\mathbb Z}$-quotient singularities. The answer to this question is negative, and we show 
in Proposition \ref{thm.p=2NumGor} that any intersection matrix $N$ such that $\Phi_N$ is killed by $2$ is always numerically Gorenstein.

\if false
\medskip
\noindent
{\bf Theorem (see \ref{thm.p=2NumGor}).} {\it Let $p=2$.  The intersection matrix associated with a resolution of a wild quotient singularity $\Spec A^G$
 is always numerically Gorenstein.
}
\medskip
\fi

The paper is organized as follows.
Section \ref{intersection matrices} contains several useful facts concerning the linear algebra of 
intersection matrices $N$, in particular formulas for the order of $\Phi_N$
when the dual graph $\Gamma_N$ is star-shaped.
Sections \ref{Computation self-intersections} and \ref{Generalities blowing-ups} are preparatory sections, where we recall basic facts regarding how to compute self-intersection numbers on a resolution of a singularity using data coming from intermediate blow-ups. This will be applied in later sections to the resolution of $\Spec B_{\mu}$, where we found it useful, 
instead of starting the resolution process by blowing up the maximal ideal, 
to first blow up an ideal naturally related to the ideal defining the fixed scheme of the action.
We provide in Section \ref{weighted homogeneous singularities} 
the explicit resolution of certain weighted homogeneous singularities of the form
$W^q-U^aV^b(V^d-U^c)=0$,
with $p,q,a,b,c,d$ subject to certain mild restrictions. 
Over ${\mathbb C}$,   
such resolution has already been obtained by Orlik and Wagreich (\cite{O-W}, \cite{Orlik; Wagreich 1971b}, \cite{Orlik; Wagreich 1977}) in full generality. 
The proofs of the theorems presented in this introduction are found in Sections \ref{brieskorn singularities} to \ref{numerically gorenstein}.

\begin{acknowledgement}
The authors gratefully acknowledge funding support from the 
 Research and Training Group in Algebraic Geometry, Algebra, and Number Theory
at the University of Georgia, from the National Science Foundation RTG grant DMS-1344994 and from the Simons Collaboration Grant 245522,
as well as from  the   Research and Training Group
GRK 2240: Algebro-geometric Methods in Algebra, Arithmetic and Topology, funded
by the Deutsche Forschungsgemeinschaft. 
\end{acknowledgement}

\section{Intersection matrices}
\mylabel{intersection matrices}

Let $B$ be a complete noetherian local  ring that is two-dimensional and normal.
Let $C_i$, $i=1,\dots, n$,  
denote the irreducible components of the exceptional divisor of a resolution of singularities of $\Spec B$, with associated intersection matrix $N:=((C_i\cdot C_j))_{1 \leq i,j \leq n}$.
This section collects some facts that depend only on the linear algebra of the   matrix $N$
and which are used in later sections.

An $n \times n$  {\it intersection matrix} $N=(c_{ij})$ is a symmetric 
negative-definite integer matrix with  negative coefficients on the diagonal, and
non-negative coefficients off the diagonal. 
The \emph{discriminant group}   $\Phi=\Phi_N$ is defined as the finite abelian group $\ZZ^n/N\ZZ^n$, 
which has order  $|\det(N)|$.
The  associated  \emph{graph} $\Gamma=\Gamma_N$ arises as follows:
Introduce vertices
$v_1,\dots, v_n$ corresponding to the standard basis vectors in $\ZZ^n$. 
Two vertices $v_i\neq v_j$ are linked by exactly $c_{ij}\geq 0$ edges.
If not stated otherwise, we   tacitly assume that $\Gamma$ is connected.
 
The \emph{degree} or \emph{valency}   of a vertex  $v\in\Gamma$ is the number of edges attached to $v$.
A vertex $v$ with valency at least three is called a \emph{node},
and a vertex $v$ with valency one is called \emph{terminal}. 
A graph is  a {\it chain} if it is connected and does not contain any node.
It is  called {\it star-shaped} if it is a  tree with a unique node.
Given a star-shaped graph $\Gamma$ with node $v_0$, 
we can consider the subgraph $\Gamma \smallsetminus \{ v_0 \}$ obtained by removing   the vertex $v_0$ and all the edges 
containing $v_0$. 
This complement is  the disjoint union of $m\geq 3$ chains $\Delta_1, \dots, \Delta_m$  
that we call the {\it terminal chains} of $\Gamma$.

\begin{emp} \label{chain}
Suppose that $N$ is an intersection matrix whose graph $\Gamma_N$ is a 
chain, with $\ell\geq 1$ consecutive vertices  $v_1, \dots, v_\ell$. For convenience, we label the diagonal entries of $N$ by $c_{ii}=-s_i$,
and we assume below that $s_i \geq 2$ for $i=1,\dots, \ell$. We associate to $N$ with this ordering of the vertices
a unique sequence of positive integers $1=r_\ell < \dots < r_{1}<r_0$  such that the following matrix equality holds,
where the square matrix on the left is $N$:
\begin{equation*}
\begin{pmatrix}
-s_1	& 1\\
1	& -s_2		& \ddots\\
	& \ddots	& \ddots	& 1\\
			& 		& 1	& -s_\ell
\end{pmatrix}
\begin{pmatrix}
r_1\\
\vdots\\
r_{\ell-1}\\
r_\ell
\end{pmatrix}
=
\begin{pmatrix}
-r_0\\
0\\
\vdots\\
0
\end{pmatrix}.
\end{equation*}
When needed, we will denote 
$R=R_N$ the transpose of the vector $(r_1, \dots, r_\ell)$, so that $NR$ is the transpose of $(-r_0,0,\dots,0)$.
It is  known that $|\det(N)|=r_0$, and that $\Phi_N$ is cyclic of order $r_0$ (\cite{Lorenzini 2013}, 3.13). 
To be able  to refer to $r_0$ and $r_1$ without indices, we will relabel them as $r_0:=a$ and $r_1:=b$. Note that by construction, 
$\gcd(a,b)=1$, 
and that we can express the reduced fraction $a/b$ completely in terms of $s_1,\dots, s_\ell$ as a continued fraction
\begin{equation}
\label{continued fraction}
\frac{a}{b} = 
[s_1,s_2,\ldots,s_\ell] :=
s_1- 
  \cfrac {1}{s_2-
    \cfrac{1}{\ddots -
      \cfrac{1}{s_\ell}}}. 
\end{equation}
Clearly, any reduced fraction $a/b$  with $a>b$ determines an intersection matrix $N$  as above. The reduced fraction $a/b=1/1$ determines the matrix $N=(-1)$.
We note that $-a/b= \det(N)/\det(N')$, 
where $N'$ is obtained from  $N$ by removing its first line and first column (recall that the determinant of the empty matrix is $1$ by convention).

As is customary, the vertices of the graph $\Gamma_N$ of an intersection matrix $N=(c_{ij})$ are
labeled with the \emph{self-intersection numbers} $-s_i:=c_{ii}$, and self-intersection numbers $-s_i=-2$ are usually omitted.
For a chain $\Gamma_N$ as above, we get the following drawing:
$$
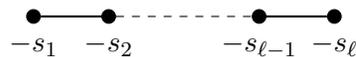
\begin{tikzpicture}
[node distance=1cm, font=\small] 
\tikzstyle{vertex}=[circle, draw, fill,  inner sep=0mm, minimum size=1.0ex]
\node[vertex]	(v1)  	at (0,0) 	[label=below:{$-s_1$}] 		{};
\node[vertex]	(v2)			[right of=v1, label=below:{$-s_2$}]	{};
\node[]	(dummy)			[right of=v2, label=below:{}]	{};
\node[vertex]	(v3)			[right of=dummy, label=below:{$-s_{\ell-1}$}]	{};
\node[vertex]	(v4)			[right of=v3, label=below:{$-s_\ell$}]	{};
\draw [thick] (v1)--(v2);
\draw[dashed] (v2)--(v3);
\draw [thick] (v3)--(v4);
\end{tikzpicture}
$$
We call such  chain a {\it Hirzebruch--Jung chain}. Recall that $p/(p-1)=[2,\dots,2]$ and that the corresponding intersection matrix of size $p-1$ and determinant $(-1)^{p-1} p$ 
is  denoted by $A_{p-1}$. This intersection matrix will be shown to arise in the context of ${\mathbb Z}/p{\mathbb Z}$-singularities in \ref{Ap-1}.
\end{emp}
\begin{emp} \label{notation.starshaped}
Let $m\geq 3$. Let $a_1/b_1, \dots, a_m/b_m$ be reduced fractions with $a_i/b_i \geq 1$ for $i =1,\dots, m$. Let $s_0\geq 1$ be any integer. 
We denote by $N=N(s_0\mid a_1/b_1,\dots, a_m/b_m)$ the following matrix. Its graph $\Gamma=\Gamma_N=\Gamma(s_0\mid a_1/b_1,\dots, a_m/b_m)$ is star-shaped with $m$ terminal chains attached to a central node $v_0$ having  self-intersection number $-s_0$.
Let $\Delta_1,\ldots,\Delta_m$ be the Hirzebruch--Jung chains determined by the fractions $a_1/b_1, \dots, a_m/b_m$. The graph $\Gamma$ is obtained by attaching to $v_0$ with a single edge 
the initial vertex of each chain $\Delta_i$. In this article, when referring to a matrix of the form  $N=N(s_0\mid a_1/b_1,\dots, a_m/b_m)$, we will always assume that it is an intersection matrix, i.e., that $N$ is negative-definite.
\end{emp}

\begin{proposition}
\mylabel{determinant star-shaped} 
Let $N=N(s_0\mid a_1/b_1,\dots, a_m/b_m)$ be an $n \times n$ intersection matrix as above, with star-shaped graph $\Gamma_N$.
Then  $s_0>\sum_{j=1}^m b_j/a_j$, and the following hold:
\begin{enumerate}[\rm (i)]
\item 
We have $\det(N)=(-1)^n(\prod_j a_j)(s_0-\sum_j b_j/a_j)$. In particular, 
there is an  integer factorization 
$$
|\det(N)|= \left(\frac{\prod_j a_j}{{\rm lcm}(a_1,\dots, a_m)}\right) \left({\rm lcm}(a_1,\dots, a_m)(s_0-\sum_j b_j/a_j)\right).
$$
\item 
In the discriminant group $\Phi_N$,
the class of the standard basis vector $e_{v_0}\in\ZZ^n$ corresponding to the central node $v_0$ has order 
${\rm lcm}(a_1,\dots, a_m)(s_0-\sum_j b_j/a_j)$. 
\item 
Let $w_j$ denote the terminal vertex of the chain $\Delta_j$. 
Then $\Phi_N$ is generated by the classes of $e_{w_j}$, $j=1,\dots, m$. Moreover, 
the class of $e_{v_0}$ is equal to the class of $a_je_{w_j}$, and the  
group $\Phi_N$ is killed by ${\rm lcm}(a_1,\dots, a_m)^2(s_0-\sum_j b_j/a_j)$.
\item  
If $a_j=p$ for all $j$ and $ps_0-\sum_j b_j=1$, then $\Phi_N$ is killed by $p$ and has order $p^{m-1}$.
\item 
Assume that $\Phi_N$ is killed by a prime $p$. 
If $p$ divides $a_j$ for some $j$, then the class of $e_{v_0}$  is trivial in $\Phi_N$.
\end{enumerate}
\end{proposition}

\proof
Without loss of generality, we may assume that  $N$ equals the block matrix
$$
N=\begin{pmatrix}
-s_0	& * 	& \cdots 	& *\\
*	& N_1	\\
\vdots	& 	& \ddots\\
*	&	&	& N_m
\end{pmatrix}
\in \Mat_n(\ZZ),
$$
where $N_i$ is the intersection matrix with graph $\Delta_i$,
with vertices numbered consecutively starting from the vertex adjacent to the node $v_0$. 
The $*$'s in the above matrix   stand  for  sequences of appropriate size, starting with $1$   followed by zeros. 
Let $R_i$ denote the positive integer vector
associated to $N_i$, such that 
$$
N_iR_i={}^t(-a_i,0,\ldots,0).
$$
Form the block column integer vector $R$ in $\ZZ^n$ given as
$$R:={\rm lcm}(a_1,\dots, a_m)  \ {}^t(1,a_1^{-1}R_1,\ldots,a_m^{-1}R_m).$$
By construction, the greatest common divisor of the entries in $R$ is $1$, since, 
given a prime $p$ such that $p^s$ exactly divides ${\rm lcm}(a_1,\dots, a_m)$,
there exists at least one index $i$ such that $a_i$ is exactly divisible by  $p^s$. In particular, the coefficient of $R$ corresponding to the last vertex on the chain $\Delta_i$ is coprime to $p$. Let $x:= s_0-\sum_jb_j/a_j$.
Then $$NR= {\rm lcm}(a_1,\dots, a_m) \ {}^t(-x,0,\ldots,0).$$ 
Note that $x>0$, because $N$ is negative-definite, so the integer ${}^tRNR$ must be negative. 
By negative-definiteness, we also know that $\det(N)$ has sign $(-1)^n$.
Using \cite{Lorenzini 2013}, Theorem 3.14  with the matrix $N$ and the vector $R$, we get
$$
\det(N) = (-1)^n (s_0-\sum_jb_j/a_j) \cdot  (\prod_j a_j )
$$
and the assertion (i) follows.
The assertion in (ii) follows immediately from the equality 
$$NR= {\rm lcm}(a_1,\dots, a_m) \ {}^t(-x,0,\ldots,0)$$ 
and the fact that the greatest common divisor of the coefficients of $R$ is $1$. 
For (iii), to show that $e_{v_0}-a_je_{w_j}$ is in the image of $N$, consider the unique positive vector $S_j$ whose first component is $1$ and such that $N_jS_j$ is equal to the transpose of $(0,\dots, 0,-a_j)$. Extend this vector to a vector $\overline{S}_j \in {\mathbb Z}^n$ by setting all other components to $0$. Then $N\overline{S}_j= e_{v_0}-a_je_{w_j}$. The proof that for any vertex $w$ on the chain $\Delta_j$, there exists an integer $c_w$ such that
$e_{w}-c_we_{w_j}$ is in the image of $N$ is similar, and is left to the reader. Using (ii) to find the order of the class of $e_{v_0}$, 
it follows immediately that the class of $e_{w_j}$  is killed by ${\rm lcm}(a_1,\dots, a_m)^2(s_0-\sum_j b_j/a_j)$, for all $j$.
Part (iv) is immediate from (i) and (iii). In Part (v), assume that $p$ divides $a_j$. 
As the class of $e_{w_j}$ is killed by $p$ by hypothesis, we find from (iii) that the class of $e_{v_0}$ is trivial. 

\if false
(v) Let us assume now that the group $\Phi$ is killed by a prime $p$. 
Assume by contradiction that $p^m$ divides $\det(N)$. Let $N^*$ denote the adjoint matrix of $N$, so that $N N^*=N^* N= \det(N) {\rm Id}_n$.
The vector $R$, up to scalar, is the first column of $N^*$. Let $N^*_i$ denote the $i$-th column vector of the matrix $N^*$ divided by the greatest common divisor of its coefficients. Then  Theorem 3.14 in \cite{Lorenzini 2013} produces a formula for $\det(N)$ in terms of the coefficients of $N^*_i$. 
It follows from this formula and from the facts that $\Phi$ is killed by $p$ and $p^m$ divides $\det(N)$ that the coefficient corresponding to the node in the vector $N_i^*$ must be divisible by $p^2$.  Under our assumptions on the form of $N$, the coefficient of the node 
is the first coefficient of the vector $N_i^*$. Since the matrix $N$ is symmetric, so is the matrix $N^*$, and thus each coefficient of the vector $R$ is divisible by $p^2$ except possibly for its first coefficient. 
Hence, for each $1\leq i\leq m$, we find that ${\rm lcm}(a_1,\dots, a_m) /a_i$ is divisible by $p^2$. 
This is a contradiction, since at least one of the integers ${\rm lcm}(a_1,\dots, a_m) /a_i$ is not divisible by $p$.
\fi
\qed

\section{Computation of self-intersections}
\mylabel{Computation self-intersections}

Let $B$ be a complete   local  noetherian ring that is two-dimensional and normal.
It is known that a resolution of singularities $X\ra\Spec(B)$ exists, and that it can be obtained from the sequence 
$$
X=Y_t \lra Y_{t-1} \lra \dots \lra Y_1 \lra Y_0=\Spec(B),
$$
where each $Y_i \to Y_{i-1}$ is the normalization of the blow-up of the finitely many singular points of $Y_{i-1}$ 
(see, e.g., \cite{Lip2}, Theorem on page 151 and Remark B on page 155).   
In this section we develop a method for computing the self-intersection of particular irreducible components
of the exceptional divisor on $X$.
This information is needed in the proofs of  each of our explicit computation of resolutions in 
Theorems \ref{intersection graph}, \ref{resolution peskin},   \ref{antidiagonal E8}, and \ref{a=y, b=x}.
For the sake of exposition, we assume that the residue field $k=B/\maxid_B$ is algebraically closed.

Note that the process described above usually  does not produce
the minimal desingularization, as some irreducible components of the exceptional divisor on $X$ might be $(-1)$-curves,
and thus contract  to smaller resolutions of singularities.
This may even happen for the strict transforms of the exceptional divisors on the first blow-up $Y_1$
(see Example in \cite[page 205]{Lip}).  

\begin{emp} \label{notation resolution}
Let $X\ra\Spec(B)$ be any resolution of singularities,   and write $C_1,\ldots,C_n$
for the  irreducible components of the exceptional divisor. 
We then have intersection numbers 
$$
c_{ij}=(C_i\cdot C_j)_X:=\chi(\O_{C_j}(C_i))-\chi(\O_{C_j}) = \deg(\O_{C_j}(C_i)),
$$
and can form the resulting intersection matrix
$N=(c_{ij})_{1\leq i,j\leq n}$.
Associated with $N$ is the connected graph $\Gamma=\Gamma_N$
with vertices
$v_1,\dots, v_n$, and a pair of vertices $v_i\neq v_j$ is linked by exactly
$c_{ij}$ edges. We call $\Gamma$ the {\it resolution graph} or the {\it dual graph} attached to $X \to \Spec B$.

Now consider a factorization $X\ra Y\ra\Spec B$, where $\pi:X\ra Y$ is
the contraction of certain exceptional curves, say $C_{s+1}\cup\ldots\cup C_n$. We regard the induced morphism
$Y\ra\Spec(B)$ as a partial
resolution of singularities, and by definition of contraction, $Y$ is {\it normal}.
Write $D_1,\ldots,D_s\subset Y$ for the images in $Y$ of the non-contracted curves $C_1,\ldots,C_s\subset X$.
These images are Weil divisors which are not necessarily Cartier.  
Following Mumford \cite{Mumford 1961}, page 17 (see also \cite{Ful}, 7.1.16, or \cite{Schroeer 2019}, Theorem 1.2), 
one has  \emph{rational  intersection numbers} $(D_i\cdot D_j)_Y\in\QQ$ obtained as follows:
First define the rational pull-back
$\pi^*(D_i):=C_i+\sum_{k>s}\lambda_kC_k$,
where $\lambda_{s+1}, \dots, \lambda_n\in\QQ$  are the fractions uniquely determined by the conditions $(\pi^*(D_i)\cdot C_k)_X=0$ for all $s< k\leq n$.
One then sets 
$$
(D_i\cdot D_j)_Y :=  (\pi^*(D_i)\cdot C_j)_X=(\pi^*(D_i)\cdot \pi^*(D_j))_X.
$$
These numbers actually do not depend on the choice of resolution $\pi:X\ra Y$.

{\it Suppose now that $\pi:X\ra Y$ is the contraction of all but the first curve $C_1$. 
Assume furthermore  that $\Gamma$ is a tree.} Let $v$ be the vertex corresponding to $C_1$,
and consider the graph 
$\Gamma\smallsetminus\left\{v\right\}$   obtained from $\Gamma$ by removing the vertex 
$v$ and all the edges   attached to $v$.
The graph $\Gamma\smallsetminus\left\{v\right\}$ decomposes into connected components 
$\Gamma\smallsetminus\left\{v\right\}= \Delta_1 \cup \ldots \cup \Delta_r$,
with corresponding intersection matrices $N_1,\ldots,N_r$ for each component. 
Since $\Gamma$ is a tree, there exists a unique  vertex $w_i\in \Delta_i$   
which is adjacent to $v$ in $\Gamma$. 
Define $\Delta_i':=\Delta_i\smallsetminus\left\{w_i\right\}$, with intersection matrix $N'_i$.
We call 
$$
\delta_i:=-\frac{\det(N_i')}{\det(N_i)}\in\QQ_{>0} 
$$
the \emph{correction term} at $w_i$
(recall that the determinant of the empty matrix is $1$, and we use this  convention if 
$\Delta_i$ is reduced to the single vertex $w_i$). 
The correction terms $\delta_i$   are indeed positive, since the signs of $\det(N_i)$ and $\det(N_i')$
are given by $(-1)^{r_i}$ and $(-1)^{r_i-1}$, where $r_i$ is the number of vertices of $\Delta_i$. 
When $\Delta_i$ is a chain as in \ref{chain} corresponding to a fraction $a_i/b_i$, 
we have $\delta_i=b_i/a_i$.
The geometric meaning of the correction terms is as follows:
\end{emp}

\begin{proposition}
\mylabel{correction self-intersection} 
In the above situation, the integral self-intersection and the rational self-intersection are related by the formula
$$ 
(C_1 \cdot C_1)_X = (D_1 \cdot D_1)_Y - \sum_{i=1}^r\delta_i.
$$
\end{proposition}

\proof 
For ease of notation, we let in this proof $C=C_1$ and $D=D_1$.
Let $N_0$ denote the lower-right principal minor of $N$. Recall from our earlier description
that $N_0$ is a block diagonal matrix with $\det(N_0)=\prod_{i=1}^r \det(N_i)$. Then 
$$
(\lambda_2,\dots,\lambda_n)=-((C\cdot C_2)_X, \dots, (C\cdot C_n)_X)N_0^{-1}.
$$ 
It follows that 
$$
(D\cdot D)_Y=(\pi^*(D) \cdot C)_X= (C \cdot C)_X + \sum_{i=2}^n\lambda_i(C_i \cdot C)_X.
$$
Since $\Gamma_N$ is a tree, we find that if $(C\cdot C_i)_X \neq 0$, then $(C\cdot C_i)_X =1$. 
We only need to compute explicitly $\lambda_i$ when $(C\cdot C_i)_X \neq 0$. 
According to our definitions, there are $r$ such indices $i$, 
and in each case,  the coefficient $\lambda_i$ is the top left corner of 
the corresponding matrix $N_i^{-1}$, that is, $\det(N_i')/\det(N_i)$, as desired.
\qed

\medskip
We will use Proposition \ref{correction self-intersection}  in the following situation. 
Let  $\idealb $ be an ideal in $B$, and let $Z \to \Spec B$ denote the blowing-up with center $V(\idealb)$.
Denote by $E\subset Z$ the schematic preimage of the center. Let $\nu:Y\ra Z$ be the normalization map and denote by 
$D=\nu^{-1}(E)$ the schematic preimage of $E$. {\it Assume that $D$, and  hence  $E$, are irreducible.}
Let $D_{\red}$ denote the support of $D$ endowed with its induced reduced structure. Letting $D_{\red}$ play the role of $D_1$ in Proposition \ref{correction self-intersection}, we find a formula 
for the rational intersection number $(D_\red\cdot D_\red)_Y$ in term of data from a resolution $X \to Y$. 
 Our next proposition shows how to obtain $(D_\red\cdot D_\red)_Y$ from data associated with the blowing-up $Z \to \Spec B$.
 
The exceptional divisor $E\subset Z$
is given by the sheaf of ideals $\O_Z(1)\subset\O_Z$. The reduction $E_\red$ is a projective curve
over the residue field $k$, allowing us to define the integral intersection number 
$$
(E\cdot E_\red)_Z :=  \chi(\O_{E_\red}(E))-\chi(\O_{E_\red})=\deg\O_{E_\red}(-1).
$$
In practice, $(E\cdot E_\red)_Z$ can often be computed, and such computation is done for instance in 
Proposition \ref{exceptional divisor}.

\begin{proposition}
\mylabel{correction degree}
In the above situation where $D$, and  hence  $E$, are assumed irreducible, write $E=mE_\red$ and let $d\geq 1$ be the degree
of the induced map  $\nu:D_\red\ra E_\red$. Then we have 
$$
(D_\red\cdot D_\red)_Y = \frac{d^2}{m} (E\cdot E_\red)_Z.
$$
\end{proposition}

\proof
First, we check that $(D\cdot \nu^{-1}(F))_Y = (E\cdot F)_Z$ for every effective Cartier divisor $F\subset Z$
that does not contain 
the support of $E$.
The two intersection numbers are the $k$-degrees of 
the finite schemes $D\cap \nu^{-1}(F)$ and $E\cap F$, respectively.
Fix a point $z\in E\cap F$, consider the  local ring $A=\O_{F,z}$ and choose an  element $t\in \maxid_A$
defining $F\cap E\subset F$ locally.
Then $A$ is a local noetherian ring of dimension one without embedded components,  and
$M=\O_{\nu^{-1}(F),z} $ is a finite $A$-module of rank one for which the multiplication map $t:M\ra M$ is injective.
According to \cite{EGA}, Chapter IV, Lemma 21.10.13, the modules $A/tA$ and $M/tM$ have the same $A$-length,
hence also the same $k$-vector space dimension.
Applying this with a difference $F-F'$ of effective Cartier divisors that are linearly equivalent to $E$,
we conclude $(D\cdot D)_Y = (E\cdot E)_Z$.

To simplify notation write  $E'=E_\red$ and $D'=D_\red$. With $D=hD'$ we get 
\begin{equation}
\label{kleiman degrees}
h^2(D'\cdot D')_Y = (D\cdot D)_Y = (E\cdot E)_Z = m(E\cdot E')_Z.
\end{equation}
We now use Kleiman's theory of rational degrees $\deg(V'/V)\in\QQ_{\geq 0}$ for  morphisms $V'\ra V$ between irreducible 
proper  schemes that are not necessarily integral  (\cite{Kleiman 1966}, Definition on page 277). 
According to \cite[Lemma 2]{Kleiman 1966}, the  commutative diagram
$$
\begin{CD}
D'	@>>>	D\\
@VVV			@VVV\\
E'	@>>> 	E
\end{CD}
$$
gives the equation
$\deg(D'/E')\cdot\deg(E'/E) =\deg(D'/D)\cdot \deg(D/E) $, and furthermore we have 
$\deg(E'/E)=1/m$ and $\deg(D'/D)=1/h$. Thus $\deg(D'/E')= m/h$.
Inserting this into \eqref{kleiman degrees} yields the assertion.
\qed

\if false
\medskip
Combining Proposition \ref{correction degree} and \ref{correction self-intersection}, 
we obtain the self-intersection  $(C\cdot C)_X<0$, given the
intersection number $(E\cdot E_\red)_Z<0$,  the multiplicity in $E=mE_\red$,
 the degree   $d=\deg(D_\red/E_\red)$, and the correction terms $\delta_i>0$.
 \fi


\section{Blowing up non-reduced centers}
\mylabel{Generalities blowing-ups}

We begin this section with some general facts on the computation of
blowing-ups, needed for instance to fully justify the explicit computations done in Proposition \ref{exceptional divisor}.
Let $B$ be a noetherian ring,  and let $\idealb\subset B$ be an ideal. Endow the associated \emph{Rees ring}
$$
 B[\idealb T]:=B\oplus\idealb T\oplus\idealb^2 T^2\oplus\ldots \subset B[T]
$$
with the grading induced by the standard grading on $B[T]$.
The morphism $\Proj(B[\idealb T]) \to \Spec B$ is called the \emph{blowing-up} of $\Spec(B)$ with center 
$\Spec(B/\idealb)$. We denote $\Proj(B[\idealb T])$ by $\Bl_\idealb(B)$ or, when no confusion may ensue, simply by $Z$.
Let $E$ denote the schematic preimage in $Z$ of the center of the blowing-up.

Assume now that $R$ is a noetherian ring with a surjection $R\to B$.
Let $\ideala$ denote the preimage in $R$ of the ideal $\idealb$.
Consider the blowing-up $Z':=\Bl_\ideala(R)$ with center $V(\ideala)$, and the commutative diagram 
induced by the surjection  $R[\ideala T]\ra B[\idealb T]$ of Rees rings:

$$
\begin{CD}
Z	@>>>	Z'\\
@VVV			@VVV\\
\Spec B	@>>> 	\Spec R.
\end{CD}
$$
The horizontal morphisms are closed immersions. 

Recall that an element $f\in R$ is called \emph{regular} if multiplication by $f$ on $R$ is an injective map.
Assume now that  the kernel of $R\to B$   is generated by a regular element $f\in R$.
Then $\Spec(B)$ is 
an effective Cartier divisor in $\Spec(R)$, and our next proposition provides a criterion for checking whether
the closed subscheme $Z$ is an effective Cartier divisor in $Z'$, when $Z'$ and $V(\ideala)$ are `nice'. This criterion is explicit and in general not very difficult to verify.

Each element $g\in\ideala$ defines a basic open set 
$D_+(g):=\Spec R[\ideala T]_{(gT)}$ of $Z'$
called the \emph{$g$-chart}. When $\ideala=(g_1,\dots, g_r)$, the union $\cup_{i=1}^r D_+(g_i)$ is an affine open cover of $Z'$.

\begin{proposition} 
\mylabel{blowing-up computation}
Let $R$ be a noetherian ring, locally of 
complete intersection\footnote{Recall that  $g_1,\ldots,g_d\in R$ is called a \emph{regular sequence} if   the class of $g_i$ is a regular element
in the ring $R/(g_1,\ldots,g_{i-1})$, for each $1\leq i\leq d$.
The ring $R$ is called \emph{locally of complete intersection}
if for each $\primid \in \Spec R$, the completion  of $R_\primid$ is isomorphic to a ring of the form $A/(a_1,\ldots,a_s)$,
where $A$ is a regular complete local ring, and $a_1,\ldots, a_s$ is a regular sequence.}.
Let $g_1,\ldots,g_r\in R$ be a regular sequence, and set $\ideala:=(g_1,\dots, g_r)$.
Let $f \in R$ be a regular element contained in $\ideala$, and set $B:=R/(f)$ and $\idealb:=\ideala B$. Consider as above the blowing-ups 
$Z \to \Spec B$ and $Z'\to \Spec R$.

For each $i=1,\dots,r$, choose a factorization $f/1=(g_i/1)^{s_i}h_i$ in $R[\ideala T]_{(g_iT)}$, with $s_i \geq 0$ and $h_i\in R[\ideala T]_{(g_iT)}$.
Assume that for each $i$, the closed subscheme $V(h_i,g_i/1)$ of $D_+(g_i)$ has  codimension two in $D_+(g_i)$.
Then
\begin{enumerate}[\rm (a)]
\item 
The closed subscheme $Z$ of $Z'$ is an effective Cartier divisor.
Its restriction on the $g_i$-chart $D_+(g_i)$ is the closed subscheme $V(h_i)$. 
\item The scheme $Z$ is locally of complete intersection.
\end{enumerate}
\end{proposition}
 
\proof Part (a) follows from Proposition \ref{blowing-up local}. Part (b) follows from Proposition \ref{blowing-up serre condition}. \qed

\begin{proposition} \mylabel{blowing-up local} Keep the notation introduced at the beginning of this section. Let $g \in \ideala$. 
Suppose that we have a factorization  $f/1=(g/1)^s h$ in $R[\ideala T]_{(gT)}$,
for some $s\geq 0$ and some element $h\in R[\ideala T]_{(gT)}$.
Suppose also that the following two assumptions hold:
\begin{enumerate}
\item
The closed subscheme $V(h,g/1)$ of $ D_+(g)$ has codimension at least two. 
\item
The basic open set $D_+(g) \subset Z'$ satisfies Serre's Condition $(S_2)$.
\end{enumerate}
Then $Z\cap D_+(g) = V(h)$ 
as  closed subschemes of the $g$-chart $D_+(g)$.
\end{proposition}

\proof By hypothesis, $g/1$ and $h$ define two closed subschemes $V(g/1)$ and $V(h)$ in $D_+(g)$. 
All schemes below are viewed as subschemes in $Z' :=\Bl_\ideala(R)$.
The conclusion of the proposition is implied by the following two claims:

(a) The open subsets $D_+(g)  \cap (Z \setminus E)$ and $V(h) \setminus V(g/1)$ are equal.

(b) The closed subscheme $V(h) \cap V(g/1)$ is an effective Cartier divisor on $V(h)$.

Then, on one hand  the schematic closure of the inclusion $D_+(g)  \cap (Z \setminus E) \to  D_+(g)  \cap Z $ is equal to $D_+(g)  \cap Z$ by Lemma \ref{blowing-up is image}, and on the other hand the schematic closure of the inclusion $V(h) \cap V(g) \to V(h)   $ is equal to $V(h),$ also by Lemma \ref{blowing-up is image}. 

We leave it to the reader to verify (a). To prove (b), note that since $f$ is regular in $R$, the element $f/1$ is regular in $R[\ideala T]_{(gT)}$.
Thus $V(h)$ and $V(g/1)$ are two Cartier divisors in $D_+(g)$. We need to show that the image of $g/1$ is not a zero-divisor in $R[\ideala T]_{(gT)}/(h)$.
Assumption (ii) implies 
that any effective Cartier divisor on $D_+(g)$ satisfies Serre's Condition $(S_1)$. 
In particular, the ring $R[\ideala T]_{(gT)}/(h)$ has no embedded primes
and thus the zero divisors in $R[\ideala T]_{(gT)}/(h)$ are contained in the minimal primes ideals. 
Krull's Principal Ideal Theorem shows  the irreducible components of $V(h)$ all have codimension one in $D_+(g)$.
Assumption (i) implies then that $g/1$ cannot be contained in a minimal prime ideal of $R[\ideala T]_{(gT)}/(h)$. Thus $g/1$ is regular in $R[\ideala T]_{(gT)}/(h)$.
\qed
\if false
Write $U=D_+(g)$ for the $g$-chart, and consider the effective Cartier divisors
$$
D=V(h)\quadand E=V(g/1)\quadand H=V(f/1).
$$
Then $H=sE+D$ as Cartier divisors. 

Assumption (ii) implies 
that any effective Cartier divisor on $D_+(g)$ satisfies Serre's Condition $(S_1)$. In particular, the 
in other words, contains no embedded component 
By Krull's Principal Ideal Theorem, the irreducible components of $V(h)$ and $V(g/1)$ all have codimension one in $D_+(g)$.
With the first assumption, we infer that $g/1$ becomes invertible in $\O_{D,\eta}$ for every generic point $\eta\in D$.
We infer that the scheme $E\cap D$ must be an effective Cartier divisor on $D$.
So Lemma \ref{blowing-up is image} (a) ensures that the schematic closure of $D\smallsetminus E\subset U$ coincides with $D$.

Now recall that we have a canonical embedding  $\Spec(R')\smallsetminus V(\ideala')\ra D\subset \Bl_\ideala(R)$, and its schematic image
is $\Bl_{\ideala'}(R')$, by Proposition \ref{blowing-up is image}. 
The structure morphism $\Bl_\ideala(R)\ra\Spec(R)$ gives an 
identification $\Spec(R')\smallsetminus V(\ideala')= D\smallsetminus E$, and we saw in the preceding paragraph
that its schematic closure on the $g$-chart equals $D$.
Our assertion follows, because schematic closures can be computed locally.
\fi

\begin{lemma}
\mylabel{blowing-up is image}
Let $V$ be the complement of an effective Cartier divisor $F$ on a noetherian scheme $Y$.
Then the schematic image in $Y$ of the open embedding $V\ra Y$ coincides with $Y$.
\end{lemma}

\proof
(a) 
The assertion is local, so we may assume that $Y=\Spec(A)$ and $F=V(g)$, where $g\in A$ is a regular element.
The schematic image is defined by the kernel of the localization map $A\ra A_g$, with $a\mapsto a/1$.
Since $g $ is regular, this kernel is the zero ideal.
\qed

\medskip
In the context of Proposition \ref{blowing-up local}, we say that the equation $h=0$ is the \emph{strict transform} of $f=0$ on the $g$-chart.
One easily sees that
condition (i)   ensures that the exponent $s\geq 0$ is the maximal exponent.
Note that in any case there is  a factorization $f/1=(g/1)^s h$ with   maximal $s\geq 0$, by Krull's Intersection Theorem,
and the resulting factor  $h$ is   unique  because $g/1$ is regular.  
In light of Krull's Principal Ideal Theorem,  when $V(h,g/1)$ of $D_+(g)$ has  codimension at least two in $D_+(g)$,
it has codimension exactly two.
This condition depends only on the radical ideal $\sqrt{(h,g/1)}$,
a remark which usually substantially simplifies the computations. 

\begin{proposition}
\mylabel{blowing-up serre condition}
Suppose the  ideal $\ideala\subset R$ is generated
by a regular sequence $g_1,\ldots,g_d\in R$. If the scheme $S=\Spec(R)$ 
satisfies Serre's Condition $(S_m)$, or is locally of complete intersection, the same holds for the blowing-up $\Bl_{\ideala}(R)$. 
\end{proposition}

\proof
The canonical module surjection $R^{\oplus d}\ra\ideala$  coming from the regular sequence 
yields a  closed embedding $\Bl_\ideala(R)\subset\PP^{d-1}_R$. 
Consider the short exact sequence
$$
0\lra\shF\lra\O_P^{\oplus d}\stackrel{(g_iT)}{\lra} \O_P(1)\lra 0
$$
of locally free sheaves on $P=\PP^r_R$. The kernel has $\rank(\shF)=d-1$. Using the inclusion $\O_P(1)\subset\O_P$,
we get a composite map $\shF\ra\O_P$. According to \cite{SGA 6}, Expos\'e VII, Proposition 1.8,
the image is the quasicoherent ideal corresponding to the closed subscheme $X=\Bl_\ideala(R)$.
Moreover, for each point  $x\in X$, the image of any basis in $\shF_x$
in the local ring $\O_{P,x}$ is a regular sequence contained in the maximal ideal $\maxid_x$.
More explicitly, we have
\begin{equation}
\label{blowing-up ring}
R[\ideala T]_{(Tg_j)} = R[S_1,\ldots,S_d]/(S_1g_j-g_1,\ldots,S_d g_j-g_d),
\end{equation}
where the identification is given by  $S_i=g_iT/g_jT$, and the generators in the above ideal form a regular sequence in the
polynomial ring. This result is due to Micali \cite{Micali}, Theorem 1.
It follows that the scheme $\Bl_{\ideala}(R)$ is locally of complete intersection if this holds for the ring $R$.

Fix a point $x\in X$ and consider the local ring $A=\O_{X,x}$. It remains  to show that   $\depth(A)\geq m$
or $\depth(A)=\dim(A)<m$. For this we may assume that $S=\Spec(R)$ is local, and that $x$ lies over the closed point
$s\in S$. Set $c=d-1$.
The local ring $A'=\O_{P,x}$ has $\dim(A')=\dim(R)+c$ and $\depth(A')=\depth(R)+c$. Moreover, 
the residue class ring $A$ has $\dim(A)=\dim(A')-c$ and $\depth(A)=\depth(A')-c$, 
the former by Krull's Principal Ideal Theorem, the latter by \cite{}, Proposition.
The assertion on the Serre Condition is immediate.
\qed

\medskip
Note that the relation $S_jg_j-g_j=0$ is equivalent to $S_j=1$, because $g_j$ is regular. In other words, in \eqref{blowing-up ring}
one may simply omit the indeterminate $S_j$.
Also note that if $R$ is integral, so is the Rees ring,
and we may regard \eqref{blowing-up ring} as the $R$-subalgebra in   $\Frac(R)$ generated by
the fractions $g_1/g_j,\ldots,g_d/g_j$.

 \if false

 The schematic preimage $E=\Bl_\ideala(R)\otimes_RR/\ideala$ is an effective Cartier divisor.
Indeed,  the blowing-up  has the following universal property as $R$-scheme (\cite{Hartshorne 1977}, Chapter II, Proposition 7.14): 
Its $A$-valued points correspond to those $R$-algebras
$A$ such that the ideal $\ideala A$ is invertible. In particular,  the blowing-up is empty if and only $\ideala$ is nilpotent.

Now suppose that $R'=R/\idealb$ is a residue class ring, and set $\ideala'=\ideala R'$.
The induced surjection  $R[\ideala T]\ra R'[\ideala'T]$ of Rees rings
yields as an inclusion $\Bl_{\ideala'}(R')\subset\Bl_\ideala(R)\otimes_RR'$ as closed subschemes
inside $\Bl_\ideala(R)$. By the universal property, the structure morphism $\Bl_{\ideala'}(R')\ra\Spec(R')$
admits a canonical partial section, defined  over the complement of the closed set $V(\ideala')$.
This leads to the following   description:

In general, it is rather difficult to compute schematic images, but  under suitable conditions one may say more.
Recall that an element $f\in R$ is called \emph{regular} if multiplicatin by $f$ on $R$ is an injective map.
Now assume that  $\idealb$    is generated by a regular element $f\in R$ contained in $\ideala$.
In other words $\Spec(R')$ is 
an effective Cartier divisor in $\Spec(R)$ containing the center. Each element $g\in\ideala$ defines a basic open set 
$$
D_+(g)=\Spec R[\ideala]_{(gT)}\subset\Bl_\ideala(R)
$$
 called the \emph{$g$-chart}. 
For each $g'\in \ideala$ we have $g'=g\cdot g'T/gT$, and 
by abuse of notation we say that the exceptional divisor is defined on the $g$-chart by the equation $g=0$.
This is indeed an effective Cartier divisor, because $g$ becomes a unit in the ring extension $R[T]_{gT}$.
Furthermore, the blowing-up $\Bl_{\ideala'}(R')$ is contained in the effective Cartier divisor $\Bl_\ideala(R)\otimes_RR/(f)$.

Now  suppose we have a factorization  $f/1=(g/1)^s h$ for some exponent $s\geq 0$ and some element $h\in R[\ideala T]_{(gT)}$.
The latter is regular, because it is a factor of a regular element.
The following observation is very useful when it comes to  explicit computations:
\fi 

\begin{emp} \label{emp.notation}
Let us return now to the wild quotient singularities  
recalled in \ref{moderately ramified}.
Let  $R=k[[x,y,z]]$  be a formal power series ring over  a field $k$  of characteristic $p>0$, and consider the element
\begin{equation*}
f:=z^p - (\mu ab)^{p-1}z - a^py + b^px.
\end{equation*}
Here $a,b\in k[[x,y]]$ is a system of parameters, and $\mu\in k[[x,y]]$ is a non-zero element
that is coprime to both $a$ and $b$. Let $B:=R/(f)$. 

Let $\ideala:=(a,b,z) \subset R$. We call $Z:=\Bl_{\ideala B}(B) \to \Spec B$ the \emph{initial blowing-up}.
In Theorem \ref{antidiagonal E8} and Theorem \ref{a=y, b=x}, we will later compute a complete resolution $X \to Z \to \Spec B$ of this initial blowing-up in two special cases. Recall that the exceptional divisor $E\subset Z$
is given by the sheaf of ideals $\O_Z(1)\subset\O_Z$.
Our next proposition computes the term $(E \cdot E_{\red})_Z$, needed for instance when applying Proposition \ref{correction degree}.
\end{emp}

\begin{proposition}
\mylabel{exceptional divisor} 
Keep the assumptions of \ref{emp.notation}. Then the following holds:
\begin{enumerate} 
\item The reduction $E_\red$ is  isomorphic to the projective line $\PP^1_k$.
\item The $z$-chart on $Z$ is disjoint from the exceptional divisor, and thus is regular.
\item The scheme $Z$ is locally of complete intersection.
\item We have $(E\cdot E_\red)_Z=-1$.
\item The local ring  $\O_{E,\eta}$  at the generic point $\eta$ of $E$ has length $p \dim_k  k[[x,y]]/(a,b)$.
\end{enumerate}
\end{proposition}

\proof
The blowing-up $\Bl_\ideala(R)$ is covered by the $a$-chart, the $b$-chart and the $z$-chart.
We start by examining the $a$-chart, which is the spectrum of the ring
$$
R[\ideala T]_{(aT)} = R[b/a,z/a]/(b/a\cdot a-b, z/a\cdot a -z).
$$
Consider the factorization $f=a^ph$ with  
\begin{equation*}
h:=\left(\frac{z}{a}\right)^p -\mu^{p-1} a^{p-1}\left(\frac{b}{a}\right)^{p-1}\left(\frac{z}{a}\right) - y + \left(\frac{b}{a}\right)^px.
\end{equation*}
The radical $J$ of the ideal generated by $h$ and $a$ in $R[\ideala T]_{(aT)}$ clearly contains $b$.
It thus also contains $x$ and $y$, because $a,b$ is a system of parameters in $k[[x,y]]$. Hence, $J$ also contains
$z/a$ and $z$. It follows that the subscheme $V(h,a)$ of the $a$-chart is one-dimensional. According to Proposition \ref{blowing-up computation},
the scheme $\Bl_{\ideala B}(B)$ coincides  on the $a$-chart with the effective Cartier divisor defined  by the 
equation $h=0$.
The exceptional divisor   is given by the additional equation $a=0$, thus equals $\Spec A$, where $A$ is the quotient of  $k[[x,y,z]][b/a,z/a]$ modulo the ideal generated by $a$, $b$, $z$, and $(z/a)^p-y+(b/a)^px$. Let $Q:=(x,y,z/a) \subset A$. Since the classes of $x,y,z/a$ are nilpotent, 
and since the quotient $A/Q$ is isomorphic to the domain $k[b/a]$, we find that $Q$ is the minimal prime ideal of $A$.

One easily sees that the $z$-chart on $\Bl_\ideala(R)$ is disjoint from the exceptional divisor.
The situation for the $b$-chart is similar to the $a$-chart
and it follows that $\Bl_{\ideala B}(B)$ is locally of complete intersection.
Moreover,   the reduced exceptional divisor
$E_\red=\Spec k[b/a]\cup\Spec k[a/b]$ is a copy of $\PP^1_k$.

The restriction to $E_\red$ of the invertible sheaf $\O_Z(1)=\O_Z(-E)$ 
is generated by the elements $aT/1$ and $bT/1$ on the two charts, respectively.
Viewing $a/b\in k[a/b,b/a]^\times$ as a cocycle, one   deduces that $\O_Z(1)$ has degree
$1$ on $E_\red$, so that $(E\cdot E_\red)_Z=-1$. 

It remains to compute the length of $\O_{E,\eta}$.  The coordinate ring of the exceptional divisor $E$
on the $a$-chart is given by 
$$
R[b/a,z/a]/(b/a\cdot a - b, z/a\cdot a-z, h, a).
$$
Clearly, the ideal on the right is also generated by $b, z, h,a$.
In turn, the above ring is isomorphic to $k[x,y,b/a,z/a]/(a,b,h)$. Regard the latter
as $\Lambda[z/a]/(h)$, where $\Lambda$ is the polynomial ring in the indeterminate $b/a$ over the local Artin ring $ k[x,y]/(a,b)$.
The ring extension $\Lambda\subset\Lambda[z/a]/(h)$ if finite and free,
because $h$ is a monic   in $z/a$. All  coefficients of $h$ except the leading one are nilpotent
in $\Lambda$, consequently $z/a$ becomes nilpotent modulo $h$. 
It follows that $\Lambda\subset\Lambda[z/a]/(h)$ induces bijections on all residue fields.
Clearly, the minimal prime $\primid\subset\Lambda$ is generated by $x$ and $y$. In turn, the local Artin ring 
$\Lambda_\primid$ has length $\dim_k k[x,y]/(a,b)$, whereas 
the local Artin ring  $\O_{E,\eta}=\Lambda_\primid[z/a]/(h)$ has length $\deg(h)\cdot\operatorname{length}(\Lambda_\primid) = p\cdot\dim_k k[x,y]/(a,b)$.
%
%
\qed

\if false
Proof in the special case where a=x.
When $a=x$, we find that in $k[[x,y]]$, $(a,b)=(x,y^r)$ with $r=\dim_k  k[[x,y]]/(a,b)$. 
We can then identify $A$ with $ k[y]/(y^r) [b/a][z/a]$ modulo $((z/a)^p-y)$, and the ideal $Q=(x,y,z/a)$ corresponds to the ideal generated by $z/a$.
We are reduced to computing the length of the maximal ideal $(y,z/a)$ in the local ring  $k[y,z/a]/(y^r, (z/a)^p-y)$, and we leave it to the reader to check that this length is $pr$, as desired. When $a=y$, we use the $b$-chart and a similar argument to conclude.
\fi

\begin{remark}
The ring $B=k[[x,y,z]]/(f)$ can be identified with the ring of invariants $A^G$
for an  action of the group $G=\ZZ/p\ZZ$ on the ring $A:=k[[u,v]]$, as recalled in \ref{moderately ramified},
where the generator  acts via  $u\mapsto u+\mu a$ and $v \mapsto v+\mu b$. We note below
that   the initial blowing-up $\Bl_{\ideala B}(B)\ra\Spec(B)$ considered in \ref{exceptional divisor}  is canonically associated to the 
action. 

Indeed,
the fixed scheme of the action is by definition the largest closed subscheme of $\Spec A$ on which the action is trivial, and we find that 
for the above action it corresponds to the ideal $I:=(\sigma(u)-u, \sigma(v)-v)= (\mu a,\mu b)$ in $A$.  
Under the above identification $B=A^G$ we have $z=ub-va$, and therefore $\mu z\in I$.
We find that $(\mu a,\mu b,\mu z) \subseteq I \cap B $. 
The reverse inclusion also holds since $A$ is flat over $k[[x,y]]$ (same proof as in \cite{Schroeer 2009}, Lemma 1.5, when $p=2$ and
a similar choice of initial blow-up was also used). 
Thus the ideals $I\cap B$ and $\ideala B=(a,b,z)$ coincide up to the factor $\mu$ and, hence, the total spaces of the
resulting blowing-ups coincide. 
\end{remark}

\section{Some weighted homogeneous singularities}
\mylabel{weighted homogeneous singularities}

Let $k$ be an algebraically closed field of characteristic exponent $p \geq 1$. 
The goal of this section is to describe a resolution of the singularity at the origin on 
the hypersurface given by the equation
\begin{equation*}
W^q - U^aV^b(V^d-U^c) = 0
\end{equation*}
when the   integers $p, q,a,b,c,d\geq 1$ are subject to certain mild restrictions. 
This is achieved in Theorem \ref{intersection graph}. Note that this   singularity is not necessarily isolated.
The above polynomial is weighted homogeneous, and resolutions of such singularities
were already studied by  Orlik and Wagreich in \cite{O-W}, \cite{Orlik; Wagreich 1971b} and \cite{Orlik; Wagreich 1977},
exploiting $\GG_m$-actions corresponding to the weights.
The former two papers rely on transcendental methods, and  the latter  mainly treats   the case of isolated singularities.
Our method is completely algebraic, and relies on the theory of toric varieties and Hirzebruch--Jung singularities.

To compute a resolution of our surface singularity, 
we first make an initial blow-up that separates the irreducible components
of the plane curve $U^aV^b(V^d-U^c) = 0$. We then pass to certain nicer subrings of the charts,
and identify their formal completions with suitable monoid rings.
This necessitates taking roots of power series along the way, requiring 
some restrictions on the integers $p, q,a,b,c,d$ as in \ref{weighted homogeneous}. 

\begin{emp} \label{emp.HJ}
Let us start with a brief review of the theory of   \emph{Hirzebruch--Jung singularities}.
Suppose that $t,r\geq 1$ and $ s\geq 0$ 
 are integers such that 
$\rho:=\gcd(t,r,s)$ is   prime to $p$. 
Consider the ring 
$$
R:=k[U,V,W]/(W^t-U^rV^s).
$$
We have a factorization $W^t-U^rV^s=\prod (W^{t/\rho}-\zeta U^{r/\rho}V^{s/\rho})$, where the product runs over the $\rho$-th
roots of unity $\zeta$ in $k$. The corresponding minimal primes $\primid_1,\ldots,\primid_\rho\subset R$
define a partial normalization $R\subset \prod R/\primid_i$, and it usually suffices to understand the 
rings $R/\primid_i$.

Assume from now on that $\rho=1$, so that $R$ is an integral domain. Let $R'$ be its normalization.
Let $D_U$ and $D_V$ denote the preimages in $\Spec R'$  of the closed subsets of $\Spec R$ defined by  $U=W=0$ and $V=W=0$, respectively.
To describe the resolution of the singularity of $\Spec R'$ at the maximal ideal $(U,V,W)$ when $\Spec R'$ is singular at this point, it is standard to first 
express $R'$ as the normalization of a different domain $R_0$, as we now recall. Given the triple $(t,r,s)$, 
we associate below a unique new triple $(t',1,s')$
such that $R'$ can be identified with the normalization of the ring $R_0:=k[u,v,w]/(w^{t'}-uv^{s'})$. {\it This identification is such that the 
closed subsets $D_U$ and $D_V$ on $\Spec R'$ are again equal to the preimages under the new normalization map $\Spec R' \to \Spec R_0$ of the 
closed subsets   of $\Spec R_0$ defined by  $u=w=0$ and $v=w=0$, respectively.} We leave it to the reader to check this claim using the explicit description of $R_0$ recalled below.

Write $r=r_0+ct$ and  $s=s_0+dt$ for some integers $r_0,s_0,c,d \geq 0$ with $r_0,s_0<t$. Then the fraction $W/(U^cV^d)$ is integral over $R$ since it satisfies the equation
$(W/(U^cV^d))^t=U^{r_0}V^{s_0}$. We can thus replace $R$ by $R[W/(U^cV^d)]$. In particular, if either $r$ or $s$ is divisible by $t$, then $R'$ is regular above $(U,V,W)$. We define in this case the {\it fraction type} of $R$ or $R'$ to be $0$. If $R'$ is not regular, then upon replacing $R$ with $R[W/(U^cV^d)]$
we may assume that $0<r,s<t$.

Let $h:=\gcd(t,r)$ and $h':=\gcd(t,s)$. Since $\gcd(t,r,s)=1$, we find that $\gcd(r,h')=\gcd(s,h)=1$. Thus we can write $ar=1+bh'$ and $cs=1+dh$ for some non-negative integers $a,b,c,d$. 
Let $U_1:=W^{at/h'}/(U^{(ar-1)/h'}V^{as/h'})$ and $V_1:=W^{ct/h}/(U^{cr/h}V^{(cs-1)/h})$. We find that $U_1^{h'}=U$ and $V_1^{h}=V$. In the integral extension
$R[U_1,V_1]$, we find that $W^{t/(hh')}=U_1^{r/h}V_1^{s/h'}$. If $t/h'$ divides $r$, or if $t/h$ divides $s$, we find that $R'$ is regular above $(U,V,W)$, and we define again in this case the {\it fraction type} of $R$ or $R'$ to be $0$.

 
Assume then that $R'$ is not regular. Replacing $R$ with $R[U_1,V_1]$, we may assume now that $h=h'=1$, and upon replacing $R$ by a larger integral extension if necessary, we can also assume that
$0<r,s<t$.
 
There exists a unique integer $e$ with  $0< e< t$ and $er=s+ct$ for some integer $c$.
Since $s <t$ by assumption, we find that $c\geq 0$.
Consider the ring $R_1:=k[U,V,Z]/(Z^t-U^rV^{s+ct})$. We find that this ring has two natural integral extensions. 
Indeed, $R_1[Z/V^c]$ is isomorphic to the ring $R$. Writing $r\rho=1+ft$ for some integers $\rho, f \geq 0$,  we find that $w:=Z^\rho/(UV^e)^f$
is such that $w^r=Z$ and $w^t=UV^e$. Thus $R_1[w]$ is integral over $R_1$ and  isomorphic to $R_0:=k[U,V,W]/(W^t-UV^e)$.
We define in this case the {\it fraction type} of $R$ or $R'$ to be $(t-e)/t$, with $0<(t-e)/t<1$ and $\gcd(e,t)=1$.

Given a resolution of singularities $X\ra\Spec R'$, we write $C\subset X$ for the exceptional curve,
and $C_U$ and $C_V$ for the strict transforms in $X$ of the Weil divisors  
$D_U$ and $D_V$ on $\Spec R'$,
respectively.
We endow all these closed subsets with the induced reduced structure of scheme.  
The following theorem is well-known (see, e.g., the pictures in \cite[page 37]{Kempf et al. 1973} or \cite[Theorem 2.4.1]{CES}), but we did not find a suitable reference in the literature which also proved the statement regarding the divisors $C_U$ and $C_V$.  We include a sketch of proof below, with references, for the convenience of the reader.
\end{emp}

\if false
The key observation is that the ring extension $R\subset R'$ can be described in terms
of monoids:
Let $S\subset \QQ^2$ denote the additive submonoid generated by the standard basis vectors $(1,0)$ and $(0,1)$ together
with the vector $\frac{1}{t}(r,s)$. Note that this latter vector uniquely determines
the integers $t,r,s$ since we assume that $\gcd(t,r,s)=1$.
Let $k[[S]]$ be the formal completion of the monoid ring $k[S]$ with respect to the maximal ideal
generated by the non-invertible monoid elements, which
here coincide with the non-zero elements. 
Then $U\mapsto (1,0)$, $V \mapsto (0,1)$, and $ W \mapsto \frac{1}{t}(r,s)$, 
defines a surjective homomorphism of complete local rings $R\ra k[[S]]$. Both rings are integral domains and two-dimensional,
so Krull's Principal Ideal Theorem ensures that the map is bijective.
 
Let $S^\text{grp}$ denote the subgroup of $\QQ^2$
generated by $S$. Let  $S'$ be the \emph{saturation} of $S$ in $S^\text{grp}$, i.e., $S'$ is the set of elements of  $S^\text{grp}$ having a non-zero multiple in $S$. 
According to   \cite[Proposition 1.3.8]{Cox; Little; Schenck 2011} 
the normalization $R\subset R'$ is given by the inclusion $k[[S]]\subset k[[S']]$.
Abusing notation slightly, we say that the equation $W^t-U^rV^s=0$, or the ring $R$, or its normalization $R'$,
define a \emph{Hirzebruch--Jung singularity of monoid type $\frac{1}{t}(r,s)$}.

Note that different submonoids $S_1,S_2\subset\QQ^2$ may yield the same saturated submonoid $S'\subset\QQ^2$.
The following  two steps   described in \cite{Barth; Peters; Van de Ven 1984}, page 83,
may change   $S$  without changing  its saturation $S'$:

\begin{enumerate}
\item
Replace $t,r,s$  by the integers $t/hh',r/h,s/h'$, where $h:=\gcd(t,r)$ and $h':=\gcd(t,s)$.

\item
When  $h=h'=1$, replace $t,r,s$ by the integers $t,1,e$,  where $0\leq e\leq t-1$  is 
the unique integer with $er\equiv s$ modulo $t$.  
\end{enumerate}

\noindent
In this situation where the monoid type is reduced to $\frac{1}{t}(1,s)$ with $\gcd(t,s)=1$ and $0\leq s<t$, we also say that the normal Hirzebruch--Jung singularity $R'=k[[S']]$
has \emph{fraction type} $t/(t-s)\in\QQ_{\geq 1}$. Note that  case $t/(t-s)=1$ occurs only for $t=1$ and $s=0$, in which case the
ring $R=k[U,V]$ is regular.
\fi

\begin{theorem}
\mylabel{orientation dual graph}  Let $s$ and $t$ be coprime integers with   $0< s<t$.
Let $R:=k[U,V,W]/(W^t-UV^s)$ and denote by $R'$ its normalization.
There is a resolution of singularities $X\ra\Spec R'$ such that  $C_U\cup C\cup C_V$
is  a divisor with simple normal crossings having the following dual graph:
$$
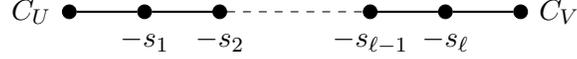
\begin{tikzpicture}
[node distance=1cm, font=\small]
\tikzstyle{vertex}=[circle, draw, fill,  inner sep=0mm, minimum size=1.0ex]
\node[vertex]	(v1)  	at (0,0) 	[label=below:{$-s_1$}]               {};
\node[vertex]	(v2)			[right of=v1, label=below:{$-s_2$}]    {};
\node[]	(dummy)			[right of=v2]	                       {};
\node[vertex]	(v3)			[right of=dummy, label=below:{$-s_{\ell-1}$}]       {};
\node[vertex]	(v4)			[right of=v3, label=below:{$-s_\ell$}]          {};
\draw [thick] (v1)--(v2);
\draw[dashed] (v2)--(v3);
\draw [thick] (v3)--(v4);
\node[vertex]	(vU)			[left of=v1, label=left:{$C_U$}]          {};
\node[vertex]	(vV)			[right of=v4, label=right:{$C_V$}]          {};
\draw [thick] (vU)--(v1);
\draw [thick] (v4)--(vV);
\end{tikzpicture}
$$
The integer $\ell\geq 1$ and the self-intersection numbers $-s_i$ are computed 
from the continued fraction expansion $t/(t-s)=[s_1,\ldots,s_\ell]$
as described in \eqref{continued fraction}. Moreover, the irreducible components of $C$
are isomorphic to $\PP^1_k$.
\end{theorem}

\proof
\newcommand{\Temb}{\operatorname{Temb}}
\newcommand{\bO}{\bar{O}}
The proof relies on the theory of toric varieties, and we refer the reader to the monographs \cite{Cox; Little; Schenck 2011}, \cite{Danilov}, or
\cite{Kempf et al. 1973}, for the general theory. The book \cite{Cox; Little; Schenck 2011} assumes from the onset that the characteristic of $k$ is $0$, but
the proofs of the results quoted below are valid in all characteristics and can be applied to our purposes. We identify $Z=\Spec R$ as an explicit (non-normal) toric variety, and use the general theory of toric varieties to describe the normalization $Y \to Z$ and the toric resolution $X_{\Sigma} \to Y$ attached to an explicit fan $\Sigma$.

Consider the lattices $N:=\ZZ^2$ and 
$M:=\Hom(N,\ZZ)$. Write $e_1,e_2\in N$ for the standard basis of $N$, and $e_1^*,e_2^*\in M$ for the dual basis.
Let $\sigma\subset N_\RR:=N \otimes_{\ZZ} \RR$ be the closed convex cone generated by the vectors $e_2$ and $te_1 - (t-s)e_2$. 
The dual cone $\sigma^\vee\subset M_\RR$ is generated by    $\alpha:=(t-s)e_1^*+te_2^*$ and $\beta:=e_1^*$.
Let  $\gamma:=e_1^*+e_2^*$, and let $S\subset M$ be the submonoid generated by $\alpha,\beta,\gamma$. We have the relation 
 $t\gamma=\alpha+s\beta$, and can identify $k[U,V,W]/(W^t-UV^s)$ with  the monoid ring $k[S]$ via $U\mapsto \alpha$, $V\mapsto \beta$, and $W\mapsto \gamma$.

Let $S':=\sigma^\vee\cap M$.
Clearly, the abelian group $M$ is generated by $\beta$ and $\gamma$. It follows that $\alpha \in M$ and, hence, $S \subseteq S'$. 
Since $S'$ is always saturated, $S'$ is equal to the saturation of the monoid $S$. It follows that the normal toric variety $Y$ attached to $N$ and $\sigma$, namely $Y:=\Spec k[\sigma^\vee\cap M]$, is 
 the normalization of the non-normal toric variety $Z:=\Spec k[S]$.
 
 The cone $\sigma$ is in normal form, and when $t>s>0$, \cite[Theorem 10.2.3]{Cox; Little; Schenck 2011} provides an explicit description of a refinement fan $\Sigma$ of $\sigma$
 such that the induced morphism $X_{\Sigma}\to Y$ is a toric resolution of singularities. Using the Hirzebruch-Jung continued fraction $[s_1,\dots, s_\ell]$ of $t/(t-s)$, 
 one constructs a sequence of vectors $u_0:=e_2$, $u_1,\dots, u_\ell$, $u_{\ell+1}:=te_1 - (t-s)e_2$ such that $\sigma= \cup_{i=1}^{\ell+1} \sigma_i$ with $\sigma_i$ the cone generated by $u_{i-1}$ and $u_i$. The fan $\Sigma$ consists in the cones $\sigma_i$ and their faces.
 
Using the Orbit-Cone Correspondence \cite[Theorem 3.2.6]{Cox; Little; Schenck 2011}, we find that the ray generated by $u_i$, $i=0,\dots,\ell+1$, corresponds to a curve $C_i$ on $X_{\Sigma}$. Since $\Sigma$ is a simplicial fan, the intersection products $(C_i \cdot C_j)_{X_{\Sigma}}$ with $0\leq i \neq j \leq \ell+1$ can be computed as in \cite[Corollary 6.4.3]{Cox; Little; Schenck 2011}, and are found to equal $1$. The self-intersections $(C_i \cdot C_i)_{X_{\Sigma}}$ for $i=1,\dots ,\ell$ are computed to equal $-s_i$ using \cite[Theorem 10.2.5]{Cox; Little; Schenck 2011} along with \cite[Theorem 10.4.4]{Cox; Little; Schenck 2011}.

The curve $C_1 \cup \dots \cup C_\ell$ is the exceptional divisor of the toric desingularization $X_{\Sigma}\to Y$. Using the Orbit-Cone Correspondence
for the surface $Y$, we let $D$ and $D'$ denote the curves on $Y$ corresponding to the rays in the cone $\sigma$ generated $e_2$ and  
$te_1 - (t-s)e_2$, respectively. The natural properties of the map $X_{\Sigma}\to Y$ implies that $D$ is the image of $C_0$, and $D'$ is the image of $C_{\ell+1}$. It remains to show that $D$ is the reduced preimage of the Weil divisor $U=W=0$ on $Z$, and that similarly, $D'$ is the reduced preimage of the Weil divisor $V=W=0$ on $Z$. 
\qed

\if false 
Recall that the scheme $Y$ is regular if and only if our generators of the cone $\sigma$ form 
a basis of the lattice $N$, and each subdivision $\sigma=\bigcup\sigma_i$ into such cones gives 
a toric resolution of singularities
$$
f:X=\Temb_N(\Delta)\lra \Temb_N(\sigma)=Y,
$$
where $\Delta$ is the fan generated by the $\sigma_i$. Such a subdivision exists in all dimension, according
to \cite{Cox; Little; Schenck 2011}, Theorem 11.1.9. 

Write  $T= \Spec k[M]$ for the torus acting on our torus embeddings.
The quotient torus $O(\eta)=\Spec k[\eta^\perp\cap M]$ is an orbit for the $T$-action on $X=\Temb_N(\Delta)$,
and $\eta\mapsto O(\eta)$ gives a correspondence between cones and   orbits.
We have $\eta_1\subset\eta_2$ if and only if the reverse inclusion $\bO(\eta_1)\supset\bO(\eta_2)$
for the orbit closures holds. Moreover, two orbits closures $\bO(\eta_1),\bO(\eta_2)$ have
non-empty intersection if and only if the cones $\eta_1,\eta_2$ are contained in a common larger cone.

Our original affine toric variety $Y=\Temb_N(\sigma)$ has precisely four orbits.
Let $\rho_U$ and $\rho_V$ be the rays generated by $e_2$ and $te_1 - (t-s)e_2$, respectively.
According to \cite{Cox; Little; Schenck 2011}, page 121 the ideal for the orbit closure $\bO(\rho_V)\subset Y$ is generated
by the set
$\sigma^\vee\cap M \smallsetminus (\rho_V)^\perp=\sigma^\vee\cap M\smallsetminus \RR e_1^*$. This  
obviously contains $\alpha,\gamma$. Under the identification \eqref{monoid ring identification}, we see that $\bO(\rho_U)\subset Y$
is the reduced preimage of the Weil divisor $U=W=0$ on $Z=\Spec(B)$.
In a similar way, one checks that $\bO(\rho_V)$ is the reduced preimage of the Weil divisor $V=W=0$.

Following \cite{Cox; Little; Schenck 2011}, Proposition 10.2.2 we use the continued fraction development $t/(t-s)=[s_1,\ldots,s_\ell]$
to define coordinates $P_0=1,P_1=s_1$ and $Q_0=0,Q_0=1$ and 
$$
P_i=s_iP_{i-1} - P_{i-2},\quad Q_i=s_iQ_{i-1} - Q_{i-2}.
$$
The resulting vectors $u_i=P_{i-1}e_1+Q_{i-1} e_2\in M$ yield cones $\sigma_i=\operatorname{Cone}(u_{i-1},u_i)$,
which in turn generate a fan $\Delta$.  According to loc.\ cit., Theorem 10.2.3 
the resulting morphism 
$$
f:X=\Temb_N(\Delta)\lra \Temb_N(\sigma)=Y
$$
is a resolution of singularities. The orbit closures  $C_i=\bO(\sigma_i)$ are copies of $\PP^1_k$, and their
union is a chain of projective lines. We have $?$, hence  $C_i^2=-s_i$, by the discussion in loc.\ cit., 
Example 10.4.7.
By equivariance, the orbit for $\rho_U\in \Delta$ maps to the orbit for $\rho_U\subset \sigma$.
In turn, the reduced strict transform $C_U\subset X$ of the Weil divisor $U=W=0$ on $\Spec(B)$
intersects the curve $C_1$. Likewise, $C_V$ intersects $C_\ell$.
Moreover, the union $C_U\cup C_1\cup\ldots\cup C_\ell\cup C_V$ has simple normal crossings,
by \cite{Kempf et al. 1973}, Chapter II, Proposition 2.
\fi

\if false
\smallskip
We say that the normal Hirzebruch--Jung singularity $R'$
has \emph{fraction type} $t/(t-s)\in\QQ_{\geq 1}$ when $R'$ is the normalization of a ring $k[U,V,W]/(W^t-U^rV^s)$ where the triple $(t,r,s)$ is of the form $(t,1,s)$ with $\gcd(t,s)=1$ and $0\leq s<t$, . Note that the  case $t/(t-s)=1$ occurs only for $t=1$ and $s=0$, in which case the
ring is regular, isomorphic to $k[U,V]$.
\fi

\begin{emp} \mylabel{weighted homogeneous} 
Let 
$q,a,b,c,d\geq 1$ be integers. Set 
$$m:=ad+bc+cd \quadand g:=\gcd(c,d).$$
Noting that $m/g$ is an integer, we further set
$$w:=\gcd(q, m/g),\quadand w_a:=\gcd(q,m/g, a), \quad w_b:=\gcd(q,m/g, b).$$
In our main result below on the resolution  
of the hypersurface singularity $W^q - U^aV^b(V^d-U^c)=0$, we  assume  that
\begin{equation}
\label{gcd condition}
\gcd(a,c/g)=\gcd(b,d/g)=1\quadand \gcd(p,wg) =1.
\end{equation}
Note that the latter condition automatically holds when $p=1$. 
The reader will easily check that the condition $\gcd(a,c/g)=1$ is equivalent to the condition $\gcd(m/g,c/g)=1$. 
Similarly, $\gcd(b,d/g)=1$ if and only if $\gcd(m/g,d/g)=1$. 

Denote by  $\alpha,\beta,\gamma\in\QQ_{< 1}$  the fraction types
of the normal Hirzebruch--Jung singularities  associated with the triples
$(t,r,s)$ given by 
$$
(\frac{qc}{gw_a}, \frac{m}{gw_a}, \frac{a}{w_a}), \quad 
(\frac{qd}{gw_b} , \frac{m}{gw_b}, \frac{b}{w_b}), \quadand 
(q,\frac{m}{g},1),
$$
respectively.   
Finally, set
\begin{equation}
\label{selfintersection node}
s_0:= \frac{w^2 g^2}{qcd} + w_a\alpha+w_b\beta+g\gamma.
\end{equation}
We  are now ready to state the main result of this section. Three complements to Theorem
\ref{intersection graph} are given in \ref{genus central node}, \ref{strict transform weighted homogeneous}, and \ref{special intersection graph}.
\end{emp}

\begin{theorem}
\mylabel{intersection graph}
Set $B:=k[U,V,W]/(W^q - U^aV^b(V^d-U^c))$, and assume that   the conditions 
\eqref{gcd condition}   holds.  With the above notation,
we have the following:
\begin{enumerate}
\item The fraction $s_0>0$ is an integer.
\item The hypersurface singularity   has a  
resolution  of singularities $X\ra\Spec(B)$ where, using the notation in {\rm \ref{notation.starshaped}}, the exceptional divisor $C\subset X$ has star-shaped dual graph
$$
\Gamma=\Gamma(s_0\mid 
\underbrace{\alpha^{-1},\ldots,\alpha^{-1}}_\text{$w_a$ },
\underbrace{\beta^{-1},\ldots,\beta^{-1}}_\text{$w_b$ },
\underbrace{\gamma^{-1},\ldots,\gamma^{-1}}_{\text{$g$ }}).
$$
when $\alpha,\beta,\gamma>0$. When one  of $\alpha,\beta,\gamma,$ equals $0$ (e.g., when $q$ divides $m/g$), the graph $\Gamma $ is as above, except that the corresponding
$w_a$ chains (resp. $w_b$ or $g$ chains) are removed.
\item 
The curve $C$ has simple normal crossings.
All irreducible components  of $C$  are copies of $\PP^1_k$, except possibly for the central node.
When $w=1$, the central node is also isomorphic to $\PP^1_k$. 
\end{enumerate}
\end{theorem}

\proof 
Since our ground field $k$ is algebraically closed,  we can rewrite  the defining polynomial   for our hypersurface singularity as
$$
f=W^q - U^aV^b\prod_\zeta (V^{d/g} -\zeta U^{c/g}),
$$
where  the product runs over the  $g$-th roots of unity $\zeta\in k$.
Assumption \eqref{gcd condition} ensures that   we  have exactly $g\geq 1$ distinct factors in the product. 
 
To construct the desired resolution of singularities $X\ra \Spec(B)$, we first make an initial blowing-up  $Z=\Bl_{\ideala B}(B) \to  \Spec(B)$, for the ideal 
$\ideala:=(U^{c/g},V^{d/g})$ in the polynomial ring $R=k[U,V,W]$. The ambient blowing-up $\Bl_\ideala(R)$ has two charts,
the $U^{c/g}$-chart and the $V^{d/g}$-chart. The former is given by 
four generators $U,V,W,V^{d/g}/U^{c/g}$ subject to the single relation
\begin{equation}
\label{relation on chart}
\left(\frac{V^{d/g}}{U^{c/g}}\right)\cdot U^{c/g} =V^{d/g},
\end{equation}
as recalled in Proposition \ref{blowing-up serre condition}. On this chart we   rewrite the defining polynomial as
\begin{equation}
\label{defining polynomial}
f= W^q-U^{a+c}V^b\cdot \prod_\zeta(V^{d/g}/U^{c/g}-\zeta).
\end{equation}
Clearly, the radical of the ideal generated by $f$ and $U^{c/g}$ contains $U,V,$ and $W$. Hence, its zero-locus is  one-dimensional, and
according to Proposition \ref{blowing-up computation} 
the blowing-up $Z=\Bl_{\ideala B}(B)$ on the $U^{c/g}$-chart of $\Bl_\ideala(R)$ is the effective Cartier divisor
with equation $f=0$. In other words, write 
$A'$ for the coordinate ring of  
the blowing-up $Z=\Bl_{\ideala B}(B)$ on the $U^{c/g}$-chart. Then this ring is generated by   four indeterminates
 $U,V,W,V^{d/g}/U^{c/g}$
subject to the two relations \eqref{relation on chart} and $f=0$ with $f$ as in \eqref{defining polynomial}.

\begin{emp} \mylabel{emp.First Self-Intersection} 
The exceptional divisor $E\subset Z$ is given by $f=U^{c/g}=0$ on this chart.
The reduction $E_\red$ is defined by $U=V=W=0$, and  $V^{d/g}/U^{c/g}$ can be regarded as a coordinate function.
The situation on the $V^{d/g}$-chart is symmetric, and we conclude that $E_\red=\PP^1_k$ is a projective line.
This description also yields  the intersection number:
Recall that the ambient $\Bl_\ideala(R)$ is the homogeneous spectrum of the Rees ring $R[\ideala T]$, so 
the invertible sheaf $\O_Z(1)$ is generated by $TU^{c/g}$ and $TV^{d/g}$ on our two charts.
In turn, the restriction to $E_\red=\PP^1_k$ is given by the cocycle
$U^{c/g}/V^{d/g}$, and it follows that $(E\cdot E_\red)_Z=-1$.
\end{emp}
\begin{emp} \mylabel{emp.Multiplicity} 
Let us note here also that the multiplicity of $E$ is $qcd/g^2$. This can be seen as follows. 
On the $U^{c/g}$-chart, the scheme $E_\red$ is defined by the ideal $Q:=(U,V,W)$.
Thus the multiplicity of $E$ can be computed as the length of the ring  $(A'/(U^{c/g}))_Q$.
It is easy to verify that the ring $A'/(U^{c/g})$ is $k$-isomorphic to the ring $\left( k[U,V,W]/(U^{c/g}, V^{d/g}, W^q)\right)[V^{d/g}/U^{c/g}]$, and the claim follows.
\if false
It is easy to verify that the class of any monomial $U^iV^jW^\ell$ with either $i \geq c/g$, $j \geq d/g$, or $\ell \geq q$, is trivial in $A'/(U^{c/g})$.
Thus we find that the multiplicity of $E$ is at most $qcd/g^2$.
Let $M$ denote the set of monomials $U^iV^jW^\ell$ in $A'/(u^{c/g})$ with $i < c/g$, $j < d/g$, and $\ell < q$, and endow this set with a total ordering such as the graded lexicographic order.  For each monomial $m \in M$, consider
the ideal $J_m:=\{U^iV^jW^\ell   \mid (i,j,\ell) \geq m\}$. Clearly, if $m>m_0$, then $J_m \subseteq J_{m_0}$. We need to show that there are exactly $qcd/g^2$ distinct ideals of the form $J_m$ with $m \in M$.
I AM STUCK HERE.
\fi
\end{emp} 

The ring $A'$ is locally of complete intersection,
but  usually fails to be normal.
\if false 
EXPLANATION
Indeed, if $A'_Q$ were normal, the multiplicity of $E$ could be computed as the valuation $v(U^{c/g})$ of $U^{c/g}$ in $A'_Q$. 
We find that $q(d/g)v(W)= [(a+c)d/g + bc/g]v(U)$. Recall that by assumption
$\gcd(b,d/g)=1=\gcd(a,c/g)$. 
It follows that $qd/(wg)$ divides $v(U)$ and $(ad+bc+cd)/(wg)$ divides $v(W)$. We then find that $qc/(wg)$ divides $v(V)$. 
Since $Q=(U,V,W)$, one of the valuations $ v(U), v(V), v(W)$ must equal $1$. This is clearly not possible for instance when $c/g,d/g>1$ and $w=1$. 
\fi
Let $\nu:Y\ra Z=\Bl_{\ideala B}(B)$ denote the normalization morphism.
To understand the normalization and   minimal resolution of the singularities of the chart $\Spec A'$ of $Z$, 
we  
pass to a subring $A$ of $A'$ with only three generators and one relation that has the same normalization as $A'$. 
It turns out that on formal completions, the resolution of singularities of $A$ is given by the theory of toric surface (i.e.,  Hirzebruch--Jung) singularities.
This formal passage to toric varieties requires the existence of certain roots of formal power series.
When $p>1$, their existence  follows from   Hensel's Lemma together with the conditions \eqref{gcd condition}, which imply that 
$\gcd(m/g,c/g)$ and $\gcd(m/g,d/g)$ are coprime to $p$.

We proceed as follows:  Let $A$ be the $k$-subalgebra of $A'$ generated by
the three elements $U,W$, and $V^{d/g}/U^{c/g}$.
The ring extension $A\subset A'$ is finite, because $A'=A[V]$ and the generator $V$ satisfies the integral equation $V^{d/g}-U^{c/g}(V^{d/g}/U^{c/g})=0$ in
\eqref{relation on chart}. Clearly, $V^{d/g} \in A$, and the relation \eqref{defining polynomial} shows that $V^b \in \Frac(A)$. 
Since we assume that $\gcd(b,d/g)=1$ in \eqref{gcd condition}, we find that $V$ can be written as 
rational function  in $V^b$ and $V^{d/g}$ and, hence, 
$V \in \Frac(A)$. 
It follows that the rings $A$ and $A'$ have the same integral closure in $\Frac(A)$.
The reduced exceptional divisor on $\Spec(A')$ is defined by the ideal $(U,V,W)$, and the restriction of $\Spec(A')\ra\Spec(A)$
to it is a closed embedding,
because $V^{d/g}/U^{c/g}\in A$ and thus the map $A \to A'/(U,V,W)=k[V^{d/g}/U^{c/g}]$ is surjective.

It turns out that the subring $A$ has a much nicer description than   $A'$, in particular when passing
to formal completions along the exceptional divisor.
Recall that $m:=ad+cd+bc$.
Taking the $d/g$-power of \eqref{defining polynomial} and using equation \eqref{relation on chart}  we get a single relation
\begin{equation}
\label{relation in subring}
W^{qd/g} = U^{m/g} \left(\frac{V^{d/g}}{U^{c/g}}\right)^{b}\cdot \prod_\zeta(V^{d/g}/U^{c/g}-\zeta)^{d/g}.
\end{equation}
Since  $b$ and $d/g$  are 
coprime by assumption \eqref{gcd condition}, we find that $w^{qd/g} = u^{m/g} z^{b}  \prod_\zeta(z-\zeta)^{d/g}$ is an irreducible polynomial in $k[u,w,z]$.
 By abuse of notation, we will also say that the \emph{equation \eqref{relation in subring} is irreducible}.
Using Krull's Principal Ideal Theorem, we conclude that the  algebra $A$ is generated by 
$U,W,V^{d/g}/U^{c/g}$ subject to the single relation \eqref{relation in subring}.

To understand the normalization of $A$, we pass to formal completions $\AAA$ 
with respect to maximal ideals $\maxid$ of the form $(U,W,V^{d/g}/U^{c/g}-\xi)$  
for various scalars $\xi\in k$.  Note that these maximal ideals correspond to points on the exceptional divisor.

Let us start with the simplest case where $\xi $ is neither zero nor a $g$-th root of unity;
here it turns out that the normalization of $\AAA$ is regular. Indeed,
the relation \eqref{relation in subring} now takes the form 
\begin{equation}
\label{xi is general}
W^{qd/g} = U^{m/g} \cdot\delta
\end{equation}
for some unit $\delta\in \AAA$. 
To proceed, we first verify that 
$\gcd(qd/g,m/g,p)=1$.
This is clear when $p=1$,
so let us assume that $p\geq 2$ is   prime.
Suppose that $p$ divides both $qd/g$ and $m/g$. Since $p$ does not divide $w:=\gcd(q,m/g)$ by hypothesis,
we have $p\nmid q$ and, hence, $p\mid d/g$, contradicting $\gcd(d/g,m/g)=1$, which we also assume in \eqref{gcd condition}. 

We conclude that there exist positive integers $r$ and $s$ such that
$\ell:= r(m/g) - s(qd/g) $  is coprime to $p\geq 1$.
With Hensel's Lemma we find roots $\delta_1:=\delta^{r/\ell}$
and $\delta_2:=\delta^{s/\ell}$ in $\AAA$, and obtain a factorization 
$\delta = \delta_1^{m/g}/\delta_2^{qd/g}$. It follows that 
$\AAA$ is isomorphic to the complete local ring described by the same three generators, but with
a modified  relation \eqref{xi is general} in which $\delta=1$.
This shows that $\AAA$ is   isomorphic to a complete local ring for a point on  
the product of a plane curve with the affine line. Consequently, the
normalization is indeed regular. Note that the plane curve is usually reducible, and
the number of irreducible components is our integer $w:=\gcd(q,m/g)=\gcd(qd/g,m/g)$.

Next, assume that $\xi=\zeta$ is one of the  $g$-th root of unity.
Rewrite \eqref{relation in subring} as 
\begin{equation}
\label{xi is zeta}
W^{qd/g} = U^{m/g} \left(\frac{V^{d/g}}{U^{c/g}}-\zeta\right)^{d/g}\cdot\delta
\end{equation}
for some unit $\delta\in \AAA$. As in the preceding paragraph, one reduces to the situation $\delta=1$.
Since $\gcd(d/g,m/g)=1$, the above relation is then irreducible.

Consider the triple $(t,r,s)=(qd/g,m/g,d/g)$. We can identify $\AAA$ with the completion of $k[u,v,w]/(w^t-u^rv^s)$ at $(u,v,w)$. 
Using the results reviewed in \ref{emp.HJ} and \ref{orientation dual graph} regarding the desingularization of $\Spec k[u,v,w]/(w^t-u^rv^s)$, we find that the singularity on $\AAA$  
is a Hirzebruch--Jung singularity
of fraction type $\gamma$.

\if false
Consider the submonoid $S\subset\QQ^2$ generated by the standard basis vectors, together with $\frac{1}{qd/g}(m/g, d/g)$.
We then obtain an identification $\AAA=k[[S]]$, defined by sending algebra generators to monoid generators via
\begin{equation}
\label{algebra monoid identification}
U\longmapsto (1,0)\quadand
\left(\frac{V^{d/g}}{U^{c/g}}-\zeta\right)\longmapsto (0,1)\quadand
W\longmapsto  \frac{1}{qd/g}(m/g,d/g).
\end{equation}
The normalization is given by $k[[S']]$, where $S'$ is the saturation of $S\subset S^\text{grp}$, which is a Hirzebruch--Jung singularity
of fraction type $\gamma$.
\fi

Finally, assume that $\xi=0$. Our relation becomes
$$
W^{qd/g} = U^{m/g} \left(\frac{V^{d/g}}{U^{c/g}}\right)^{b}\cdot \delta
$$
for some unit $\delta\in \AAA$, and again we reduce to the situation $\delta=1$. 
The above equation is usually not irreducible, and the number of irreducible factors is our
integer $w_b:=\gcd(q,m/g,b)$, which also equals $ \gcd(qd/g,m/g,b)$ since we noted in \ref{weighted homogeneous} that $\gcd(d/g,m/g)=1$. 
Let $\primid_1,\ldots,\primid_{w_b}\subset \AAA$ be the resulting minimal prime ideals.

Consider the triple $(t,r,s)=(qd/(gw_b),m/(gw_b),b/(gw_b))$. We can identify $\AAA/\primid_i$ with the completion of $k[u,v,w]/(w^t-u^rv^s)$ at $(u,v,w)$. 
Using the results reviewed in \ref{emp.HJ} and \ref{orientation dual graph} regarding the desingularization of $\Spec k[u,v,w]/(w^t-u^rv^s)$, we find that the singularity on $\AAA/\primid_i$  
has the resolution of a Hirzebruch--Jung singularity
of fraction type $\beta$.
\if false
Arguing as in the preceding paragraph, we see that  each $\AAA/\primid_i$ is isomorphic to the monoid ring $k[[S]]$, where
$S\subset\QQ^2$ is the submonoid generated by the standard basis vectors, together with
$\frac{1}{qd/(gw_b)}(m/(gw_b),b/w_b)$, and the normalization $k[[S']]$ yields a Hirzebruch--Jung singularity
of fraction type $\beta$. 
\fi 
The number of such singularities on the normalization of 
$\AAA$ is  $w_b\geq 1$.

The situation on the   $V^{d/g}$-chart is symmetric, where $w_a\geq 1$  Hirzebruch--Jung singularities
of fraction type $\alpha$ appear.
Summing up, we have described the singularities appearing on the normalization $\nu:Y\ra Z=\Bl_{\ideala B}(B)$.

Recall from \ref{emp.First Self-Intersection} that the exceptional divisor $E\subset Z$ has reduction  $E_\red=\PP^1_k$, 
with coordinate rings  $k[ V^{d/g}/U^{c/g}]$
and $k[ U^{c/g}/V^{d/g}]$.
Write $D:=\nu^{-1}(E)$ for the preimage of the exceptional divisor under the map $\nu$.
We now analyze the induced morphism $ D_\red\ra E_\red$. This morphism is flat, because $E_\red$ is regular.
The formal description of the normalization 
$\nu:Y\ra Z$ via inclusions $k[[S]]\subset k[[S']]$ of monoid rings shows that $D_\red$ is regular.
Equation \eqref{xi is general} implies that  
\begin{equation}
\label{degree is w}
\deg(D_\red/E_\red) = w =\gcd(q,m/g).
\end{equation}
 In a similar way, Equation  \eqref{xi is zeta} tells us    that $D_\red \ra E_\red$ 
is completely ramified over the points   where $V^{d/g}=\xi$ is a $g$-th root of unity.
Hence, the curve $D_\red$ is connected. Since it is also regular, it is in fact irreducible. We can then apply Proposition \ref{correction degree} along with \ref{emp.First Self-Intersection} and \ref{emp.Multiplicity}
and obtain that $$(D_\red\cdot D_\red)_Y = \frac{w^2}{(qcd/g^2)} (E\cdot E_\red)_Z= -w^2g^2/qcd.$$

Let $X\ra Y$ be the 
resolution of singularities obtained by resolving the Hirzebruch--Jung singularities
of fraction types $\alpha,\beta$ and $\gamma$ occurring on $Y$.   
The resulting dual graph $\Gamma$ is star-shaped, with the central node corresponding to 
the strict transform $C_0\subset X$ of  $D_\red \subset Y$. When $\gamma>0$, there are $g$ terminal chains obtained from the
continued fraction development of $1/\gamma=[s_1,\ldots,s_\ell]$. Using  the
identification of 
$\AAA$ with the completion of $k[u,v,w]/(w^{qd/g}-u^{m/g}v^{d/g})$ at $(u,v,w)$ discussed above, as well as Theorem \ref{orientation dual graph} and the identifications reviewed in \ref{emp.HJ}, one sees that the vertex of the terminal chain
adjacent to the central node has self-intersection $-s_1$. The situation for the other  Hirzebruch--Jung singularities
is similar.

It is now an easy matter to compute the  self-intersection $(C_0\cdot C_0)_X$ using Proposition 
\ref{correction self-intersection}, which asserts that 
$(C_0 \cdot C_0)_X = (D_\red \cdot D_\red)_Y - \sum_{i}\delta_i$. There are $w_a$ correcting terms $\alpha$, $w_b$ correcting terms $\beta$,
and $g$ correcting terms $\gamma$ (see just before \ref{correction self-intersection} for the correcting term of a chain). Hence, $|(C_0 \cdot C_0)_X|=s_0$, as desired.
Since $(C_0 \cdot C_0)_X$ is the self-intersection of a curve on a regular surface, we find that it must be a negative integer, proving (i).
To complete the proof of Theorem \ref{intersection graph} it remains to show in (iii) that the central node $C_0$ is a rational curve when $w=1$. This is done using the following proposition.
\qed


\begin{proposition} \mylabel{genus central node}
Keep the hypotheses of Theorem {\rm \ref{intersection graph}}.
Let $v_0\in\Gamma$ be the central node, and let $C_0\subset X$ be the corresponding curve on 
the resolution $X\to \Spec B$.
We have $h^1(\O_{C_0}) = ( g(w-1) +2 -w_a -w_b) /2$. In particular, when $w=1$, $h^1(\O_{C_0})=0$.
\end{proposition}
\proof
Consider the ramified covering $C_0\ra E_\red=\PP^1_k$ induced from the morphism $X\ra Y$.
It follows from  \eqref{degree is w}  that the degree of this map is $w$. 
Assumption \eqref{gcd condition} ensures that this degree is coprime to the characteristic exponent, so that the map is 
separable. 
Let us regard the closed  points on $\PP^1_k$ as elements $\xi\in k\cup\{\infty\}$.
The description of the normalization of the rings $\AAA$ in the preceding proof shows
that $C_0\ra\PP^1_k$ is totally ramified over each of the $g$-th roots of unity in $k$, and therefore the ramification indices are coprime to $p$.
Furthermore, there are $w_a$ points in $C_0$ over $\xi=0$ and all these points
have the same ramification index $w/w_a$. Similarly, there are $w_b$ points in $C_0$ over $\xi=\infty$ with ramification index $w/w_b$.
Applying the Riemann--Hurwitz Formula $2h^1(\O_{C_0}) - 2 = w(2h^1(\O_{\PP^1_k})-2) +\sum_x(e_x-1)$, we get the desired formula.
\qed

\medskip
The scheme $\Spec(B)$ contains two copies   of the affine line, given by the   equations $U=W=0$
and $V=W=0$. 
Write $C_U$ and $C_V$ for their respective strict transforms in $X$ with respect to the resolution 
$X\ra\Spec(B)$. For  a later application 
in Theorem \ref{antidiagonal E8}, we explicitly determine below how these curves
intersect the exceptional divisor $C\subset X$  {\it when $w=1$}.  
Under this additional hypothesis, the partial resolution  $Y \to \Spec B$ contains  exactly one Hirzebruch--Jung singularity of fraction type $\alpha$ and one of type $\beta$.
Let $\Delta_\alpha$ and $\Delta_\beta$ be the terminal chains of $\Gamma$ resulting from resolving these two singularities.
Write  $C_\alpha$ and $C_\beta$ for the irreducible components of $C$ corresponding to the terminal vertices of $\Gamma$
lying on $\Delta_\alpha$ and $\Delta_\beta$, respectively.

\begin{proposition}
\mylabel{strict transform weighted homogeneous}
Keep the hypotheses of Theorem {\rm \ref{intersection graph}}. Assume that $w=1$.
Then the strict transform $C_V$ intersects the exceptional divisor $C$ only in $C_\beta$, with 
intersection number $(C_V\cdot C_\beta)_X=1$.
Likewise, $C_U$ intersects $C$ only in $C_\alpha$, with $(C_U\cdot C_\alpha)_X=1$.
\end{proposition}

\proof
By symmetry, it suffices to verify the first assertion.  
Let us first work with the effective Cartier divisor on $\Spec(B)$
given by $V^{d/g}=0$. Its strict transform $C_V'\subset X$ has the same support as $C_V$.
Using the notation from the proof of Theorem \ref{intersection graph}, we see that its image on $\Spec(A)$
is given by $V^{d/g}/U^{c/d}=0$. Using Theorem \ref{orientation dual graph}
 one infers that $C_V'$ intersects only $C_\beta$,
and that its reduction has intersection number $(C_V\cdot C_\beta)_X=1$.
\qed

 
\begin{proposition}
\mylabel{special intersection graph}
Keep the hypotheses of Theorem {\rm \ref{intersection graph}}, and suppose furthermore   that $p=q$.
Set $a_p:=1$ if $p\mid a$, and  $a_p=0$ otherwise. Similarly, set $b_p:=1$ if $p \mid b$, and $b_p=0$ otherwise.
Let $N$ denote the intersection matrix   of the resolution  of the hypersurface singularity
$$
W^p-U^aV^b(V^d-U^c)=0
$$
described in Theorem {\rm \ref{intersection graph}}.
Then $|\Phi_N|=p^{g+1-a_p-b_p}$, and the group $\Phi_N$ is killed by $p$.
\end{proposition}

\proof
First note that for $q=p=1$ the assertion is trivially true, because then our hypersurface singularity is actually regular.
So we may assume that $q=p\geq 2$ is a prime number.
The triples $(t,r,s)$ in \ref{weighted homogeneous} specialize to  
$(pc/g, m/g, a)$, 
 $(pd/g, m/g, b)$, and 
$(p,m/g,1)$.
\if false
$$
\frac{1}{pc/g} (\frac{ad+bc+cd}{g}, a)\quadand
\frac{1}{pd/g} ( \frac{ad+bc+cd}{g}, b) \quadand 
\frac{1}{p}(\frac{ad+bc+cd}{g},1).
$$
\fi
From our assumptions \eqref{gcd condition} one easily  sees that $m/g$ is coprime to   $p$,  $pc/g$ and  $pd/g$. 
In particular we have $w=w_a=w_b=1$.
Furthermore, the resulting reduced fractions   $\alpha,\beta,\gamma\in\QQ$ have as  denominators the integers
$p^{1-a_p}c/g$, $p^{1-b_p}d/g$,  and $p$, respectively. According to Theorem \ref{intersection graph}, the  graph $\Gamma_N$ is star-shaped.
Thus we may compute the determinant of the intersection matrix with Proposition \ref{determinant star-shaped}
and obtain
$$
|\det(N)|= (p^{1-a_p}c/g)(p^{1-b_p}d/g)p^g (s_0-\alpha -\beta -g\gamma).
$$
The last factor is $g^2/pcd$  in light of the formula \eqref{selfintersection node} for the self-intersection $-s_0$
of the central node in   Theorem \ref{intersection graph}. Thus
$|\Phi_N|=|\det(N)|=p^{g+1-a_p-b_p}$.

The group structure of $\Phi_N$ can be obtained by computing the Smith Normal Form of the matrix $N$, 
using a row and column reduction of $N$. 
Reducing the intersection matrix of each terminal chain as in 
\cite{Lor1992} Lemma 2.5, we find that the matrix $N$ is equivalent 
to a block diagonal matrix with two blocks, a square matrix $A$
of size $(g+3)\times(g+3)$ that we describe below, and an identity matrix:
$$
A:=\left( 
\begin{array}{cccccc}
-s_0 & *    & * & * & \dots & * \\
1    & -p^{1-a_p}c/g & 0 &  0 &       & 0 \\
1    & 0    & -p^{1-b_p}d/g& 0 &       & 0 \\
1& 0 & 0    & -p &  &        0 \\
\vdots&  &  &  & \ddots &   \vdots \\ 
1 & & & & & -p
\end{array}
\right).
$$
The matrix $N\otimes\FF_p$ has $g+a_p+b_p$ rows  equal to $(1,0,\dots,0)$, and we see that its rank
is at most $r=1+ b_p+a_p+1$. In turn, the vector space dimension of the cokernel
is at least $g+3-r=g+1-a_p-b_p$. It follows that $\Phi_N=\Phi_N\otimes\FF_p$.
\qed

\begin{remark}
The explicit resolution of $ W^p-UV(V-U^p)=0 $ is needed in the proof of Theorem \ref{antidiagonal E8}. 
In this case, the intersection matrix is $N=N(2 \mid p/(p-1), p/(p-1), p^2/(2p-1))$, with $|\Phi_N|=p^2$.
When $p$ is odd, we do not know if this intersection matrix can occur as the intersection matrix of the resolution of a ${\mathbb Z}/p{\mathbb Z}$-quotient singularity. When $p=2$, this equation defines the singularity $D^0_{6}$ with trivial local fundamental group  \cite{Artin 1977}. The singularity $D^1_6$ is a wild ${\mathbb Z}/2{\mathbb Z}$-quotient singularity (\ref{E7 homogeneous singularities}).

More generally, one might wonder whether every intersection matrix arising in Proposition \ref{special intersection graph} 
can occur as the intersection matrix of the resolution of a ${\mathbb Z}/p{\mathbb Z}$-quotient singularity. We discuss the case of
$W^p-U^pV^p(V^{pm+1}-U^{pn+1})$ and $W^p-UV(V^{pm-1}-U^{pn-1})$ in Theorem \ref{Brieskorn quotient sing}.
We note in \ref{E7 homogeneous singularities} how the intersection matrix of the resolution of the singularity defined by  $ W^p-UV(V^{pm}-U^{pn-1})=0 $
might occur as the intersection matrix of the resolution of a ${\mathbb Z}/p{\mathbb Z}$-quotient singularity.
\end{remark}

\begin{remark} The resolution $X \to Y \to \Spec B$ 
provided in Theorem {\rm \ref{intersection graph}} is not always minimal.
This can be seen already in the case where $q=1$, in which case $\Spec B$ is regular, but the exceptional divisor $C$ on $X$ is not reduced to a point.
Indeed, the graph $\Gamma$ consists in this case of a central node of self-intersection $-1$ with two terminal chains obtained by resolving Hirzebruch--Jung singularities associated with the triples $(c/g, m/g,a)$ and $(d/g,m/g,b)$. The fraction types of these triples are independent of $a$ and $b$. Indeed, let $\alpha, \beta >0$ be the unique positive integers such that $\alpha (d/g)+\beta(c/g)= 1 +(c/g)(d/g)$.
Then the triple $(c/g,m/g,a)$ reduces to $(c/g,1,\alpha)$, and $(d/g,m/g,b)$ reduces to $(d/g,1,\beta)$.

\if false
Explanation of the exceptional divisor having self-intersection -1 using the formula that we give:

Write 
$\alpha (m/g) = a + (c/g) \alpha'$, with $0<\alpha<c/g$. It follows that $\alpha' < m/g$.
$\beta (m/g) = b + (d/g) \beta'$, with $0<\beta<d/g$. It follows that $\beta' < m/g$.
The self-intersection is 
$s_0:= g^2/cd + (c/g -\alpha)/(c/g)+ (d/g-\beta)/(d/g)$.
We claim that $s_0=1$. In other words, we need $\alpha(d/g) + \beta(c/g)= (c/g)(d/g)+1$.
We simply need to show that 
 $\alpha(d/g)(m/g) + \beta(c/g)(m/g)= (m/g)(c/g)(d/g)+(m/g)$.
From 
$\alpha (m/g) = a + (c/g) \alpha'$, and $\beta (m/g) = b + (d/g) \beta'$,
we get
$\alpha(d/g)(m/g) + \beta(c/g)(m/g)=a(d/g) + b(c/g) + (c/g)(d/g)( \alpha'+ \beta')= (m/g) + (c/g)(d/g)( \alpha'+ \beta'-g)$.
Thus it suffices to show that 

$( \alpha'+ \beta'-g)= m/g$.
Clearly the quantity $g( \alpha'+ \beta'-g)$ is divisible by $m$.
Since $\alpha'g < m$ and $\beta' g< m$, we find that $( \alpha'g+ \beta'g-g^2)<2m$. Hence, $g( \alpha'+ \beta'-g)=m$, as desired. 

The fraction types are independent of $a$ and $b$. Indeed, let $\alpha, \beta >0$ be the unique integers such that $\alpha (d/g)+\beta(c/g)= 1 +(c/g)(d/g)$.
Then the triple $(c/g,m/g,a)$ gives $(c/g,1,\alpha)$ and the triple $(d/g,m/g,b)$ gives $(d/g,1,\beta)$.

\fi

Other examples where the resolution is not minimal 
can also be obtained when $q>1$; for instance, when $p=2$, the resolution graphs appearing in Example \ref{ex.second} are the resolution graphs associated with the cases $ W^2-U^2V^2(V^7-U^3)=0 $ and $ W^2-U^2V(V^4-U^3)=0 $, respectively.
\end{remark}

\if false
\begin{remark} \label{diagonal case}
Recall the quotient singularities $\Spec B_{\mu}$ introduced in \eqref{special normal form}, and defined by the equation
$$z^p - (\mu ab)^{p-1}z - a^py + b^px=0.$$
We consider in this article the case where $a=y^n$ and $b=x^m$. Clearly, the case where $a=x^m$ and $b=y^n$ is also of interest. We note however that in this latter case, the singularity on $\Spec B_{\mu}$ when $\mu \in k^*$ is obtained from a diagonal action on a product of curves, and thus its resolution is already completed understood thanks to the results in \cite{Mit}. Indeed, in this case the ring $A=k[[u,v]]$ is obtained as a tensor product of $k[[x]][u]/(u^p-(\mu a)^{p-1}u-x)$ and 
$k[[y]][v]/(v^p-(\mu b)^{p-1}v-y)$, and each of these one-dimensional rings is endowed with an action of ${\mathbb Z}/p{\mathbb Z}$, sending $u$ to $u+\mu a$
and $v$ to $v+\mu b$, respectively. We will call this case the {\it diagonal} case.

As in the case considered in this article, preliminary computations in the diagonal case also seem to indicate that the combinatorial type of the resolution is independent of $\mu$. 
In the case $\mu=0$, we obtain a singularity $\Spec B_{\mu=0}$ given by the equation $ z^p-xy(x^{pm-1}-y^{pn-1})=0 $. Let $g:=\gcd(pm-1,pn-1)$.
Theorem \ref{intersection graph} can be applied to this equation in characteristic $p$, and we thus obtain the explicit combinatorial type 
of the resolution of $\Spec B_{\mu=0}$, with intersection matrix $N$. Proposition \ref{special intersection graph}
shows that we have  $|\Phi_N|=p^{g+1}$. We note that in the diagonal case, the exponent $g+1$ is easily seen to  never be congruent to $1$ modulo $p$, 
since $p$ never divides $g$ in this case.
\end{remark}
\fi

\section{Brieskorn singularities}
\mylabel{brieskorn singularities}

Let $k$ be an algebraically closed field of characteristic exponent $p \geq 1$. Let $q,c,d \geq 2 $ be integers, with $q$ coprime to $cd$. Let 
$$B:=k[[x,y,z]]/(z^q+x^c+y^d).$$
We study in this section properties of the singularity $\Spec B$.
Let $g:=\gcd(c,d)$.

\begin{theorem} \label{thm.BrieskornResolution} Assume that $\gcd(p,g)=1$.
Then $\Spec B$ admits a star-shaped resolution of singularities $X \to \Spec B$ whose associated intersection matrix is 
$$
N=N(s_0\mid a_1/b_1, a_2/b_2, \underbrace{a_0/b_0,\ldots,a_0/b_0}_{\text{$g$ entries}}),
$$
where $N$ is specified as follows (notation as in {\rm \ref{notation.starshaped}}). Let
$$
a_1:=c/g, \quad a_2:=d/g, \ \text{\rm and \ } a_0:=q.
$$
Set 
$ 
\ell_0:= cd/g$, $\ell_1:= dq/g$, and $ \ell_2:= cq/g,
$  
and define $b_i$ by $b_i \ell_i \equiv -1 \mod a_i$.
Finally, set 
$$
s_0:= g^2   / cd q  + b_1/a_1 + b_2/a_2 + gb_0/q.
$$
In case $a_1=1$ (resp. $a_2=1$), in which case $b_1=0$ (resp. $b_2=0$),   we remove the term $a_1/b_1$ (resp. $a_2/b_2$) from the matrix $N$.

When $q=p$, the associated discriminant group $\Phi_N$ is killed by $p$ and has order $p^{g-1}$.
\end{theorem}
\proof
Theorem \ref{intersection graph}
provides a resolution of the weighted homogeneous singularity
$$C:=k[[x,y,Z]]/(Z^q-x^qy^q(y^d-x^c))$$
when $q$ is coprime to $cd$. The scheme $\Spec C$ is not normal, and the natural map
$C \to B$, with $Z\mapsto zxy$, induces a finite birational morphism $\Spec B \to \Spec C$. Hence, $\Spec B$ has the same resolution as $\Spec C$.
The reader will check that the matrix $N_C$ associated to the resolution of $\Spec C$ in Theorem \ref{intersection graph} is the same as the matrix $N$ appearing 
in the statement of Theorem \ref{thm.BrieskornResolution}. The discriminant group $\Phi_N$ is computed in Proposition \ref{special intersection graph}.
\qed

\begin{remark} 
A resolution of the Brieskorn singularity of the form 
$x^{c}+y^{d}+z^e=0$ is known over the complex numbers
thanks to the work of \cite[Theorem, page 232]{H-J} when $c,d$, and $e$ are pairwise coprime,  
and \cite{O-W} in general.
An explicit description for the intersection matrix $N$ and dual graph $\Gamma_N$ of a resolution is found for instance in \cite{Tom},
page 284, with a formula giving the self-intersection  $-s_0$ of the node given on page 287.

Let now $p>1$ be prime. When  $p$ is coprime to $cd$, the intersection matrix for the resolution  of $z^p+x^{c}+y^{d}=0$
obtained  in Theorem \ref{thm.BrieskornResolution} is the same as the intersection matrix obtained in characteristic $0$. 
Some characteristic $p>1$ examples appear explicitly already in the literature, such as the case of  $z^p+x^2+y^{p+2}=0$ when $p$ is odd, treated 
in \cite{M-R}, Lemma 3.13.

Assume that $p>1$ is prime and divides $ cd$.  The Brieskorn singularity $z^p+ x^c+y^d=0$ has then a resolution in characteristic $p$ which is quite different than in characteristic $0$. Indeed, assume that $c=p\gamma$ for some integer $\gamma$, and $\gcd(p,d)=1$. Then in characteristic $p$, $z^p+ x^c+y^d=(z+ x^\gamma)^p+y^d.$ It follows that the normalization of $k[[x,y]][z]/(z^p+ x^c+y^d)$ is regular when ${\rm char}(k)=p$.  
On the other hand, in the case for instance of  $z^2+ x^3+y^6=0$ in characteristic $0$ (a case which is not covered by Theorem \ref{thm.BrieskornResolution}),
 the minimal resolution is a smooth elliptic curve of self-intersection $-1$. This explicit example of 
a resolution in characteristic $0$ (and many others) is found for instance in \cite[page 1290]{Lau77}.
\end{remark}

\begin{theorem} \label{Brieskorn quotient sing} Let $B:=k[[x,y,z]]/(f)$, where $f(x,y,z)$ is a  weighted homogeneous polynomial of the following form, with $n,m \geq 1$:
\begin{enumerate}[\rm (i)]
\item $z^p+x^{pm+1}+y^{pn+1}$,
\item $z^p+xy(x^{pm-1}-y^{pn-1})$,   and
\item $z^p-x^{2}+2y^{p+1}$ when $p\geq 3$. 
\end{enumerate}
Then $\Spec B$ is a wild $\ZZ/p\ZZ$-quotient singularity. Moreover, the fundamental group of the punctured spectrum $\Spec B \setminus \{\maxid_B\}$ is trivial. 
\end{theorem}
\proof
The proof of the theorem is similar for each of the three types of homogeneous polynomials. In each case, there exists a family of rings 
$B_{\ttt}$, $\ttt$ homogeneous in $ k[x,y]$,
such that the ring $B$ can be identified with the ring $ B_{\ttt=0}$, and such that when $\deg(\ttt)$ is large enough, there is an isomorphism
between $B_{\ttt=0}$ and $B_{\ttt}$. The family $B_\ttt$ is constructed such that when $\ttt \neq 0$ is chosen adequately, the ring $B_\ttt$ is a wild $\ZZ/p\ZZ$-quotient singularity.

For the weighted homogeneous form in (iii), we use the family $B_\ttt$ (with $\ttt \in k[y]$) described in Theorem \ref{resolution peskin}.
For the weighted homogeneous forms in (i) and (ii), we use the families discussed in \cite{LS1} and recalled in \ref{moderately ramified}.
More precisely, fix a system of parameters $a,b$ in $k[[x,y]]$. 
Consider the family of hypersurface singularities $\Spec B_{\ttt}$, $\ttt \in k[[x,y]]$, with 
$$
\R_{\mu}:=k[[x,y,z]]/(z^p - (\ttt ab)^{p-1}z - a^py + b^px).
$$
Let $G=\ZZ/p\ZZ$.  When $\mu$ is a unit in $k[[x,y]]$, $B_{\ttt}$ is isomorphic to the ring of invariants $A^G$ of an action of $G$ on $A=k[[u,v]]$,
and in this case the morphism $\Spec A \to \Spec A^G$ is ramified precisely at the closed point.
When $\mu $ is not zero and is coprime to $a$ and coprime to $b$, then again $B_{\ttt}$ is isomorphic to the ring of invariants $A^G$ of an action of $G$ on $A=k[[u,v]]$,
but in this case the morphism $\Spec A \to \Spec A^G$ is ramified in codimension $1$. 
The cases (i) and (ii) are obtained by setting $a=y^n$ and $b=x^m$, and $a=x^m$ and $b=y^n$, respectively. 

We now claim that it is possible to find a homogeneous polynomial $\mu$ of large enough degree such that $B:=k[[x,y,z]]/(f)$ is isomorphic over $k$ to $B_\ttt$.
In the cases (i) and (ii), we note that  the homogeneous polynomial $\mu:=x^t+y^t $ is coprime to both $a$ and $b$, so that 
the corresponding $\Spec B_\ttt$ is a quotient singularity associated with an action that is ramified in codimension $1$.

To prove the existence of a $k$-isomorphism from $B:=k[[x,y,z]]/(f)$  to $B_\ttt$, we use the Lemma in \cite{Greuel; Kroening 1990}, 2.6, page 345.
For the details of the proof of this Lemma, the authors of \cite{Greuel; Kroening 1990} refer the reader to the paper \cite{B-L}.
Recall that the Tjurina ideal of $f$ is $j(f):= (f, \frac{\partial f}{\partial x}, \frac{\partial f}{\partial y}, \frac{\partial f}{\partial z})$, 
and that there exists an integer $s>0$ such that $(x,y,z)^s \subseteq j(f)$ if and only if the Tjurina number $\tau:={\rm dim}_k(k[[x,y,z]]/j(f))$ is finite. 
Then the Lemma in \cite{Greuel; Kroening 1990}, 2.6, implies that if $\deg(\ttt g) > 2\tau$ (with $g\in k[[x,y,z]]$), then $B:=k[[x,y,z]]/(f)$ is isomorphic over $k$ to 
$k[[x,y,z]]/(f+\ttt g)$. 

In each case above, we have shown that $\Spec B$ is isomorphic to a  quotient singularity $\Spec B_\mu$ such that $B_\mu$ is the ring of invariants of an action of ${\mathbb Z}/p{\mathbb Z}$ on the ring $A:=k[[u,v]]$ such that the morphism $\Spec A \to \Spec B_\mu$ is ramified in codimension $1$. 
Corollary 1.2 (ii) in  \cite{Artin 1977} shows that the fundamental group of the punctured spectrum of $\Spec B$ is trivial.
\qed

\begin{remark} Consider the equation $f:=z^q+x^c+y^d$ with $q,c,d$ three distinct primes. Let $k$ be a field of characteristic $p$.  
Let $B:=k[[x,y,z]]/(f)$. Theorem \ref{intersection graph} shows that the intersection matrix
of the resolution of $\Spec B$ is the same in all three characteristics $p=q,c,d$, and has determinant $1$. It is natural to wonder whether this matrix can occur in 
more than one 
characteristic as the intersection matrix attached to a resolution of a wild ${\mathbb Z}/p{\mathbb Z}$-quotient singularity. 

Consider the intersection matrix with resolution graph $E_8$. 
In Artin's notation  in \cite{Artin 1977}, $f:=z^2+x^3+y^5$ defines the singularity $\Spec B$ denoted by $E_8^0$,
with resolution graph $E_8$. This singularity is a wild ${\mathbb Z}/p{\mathbb Z}$-quotient singularity when $p=2$ (see Theorem \ref{Brieskorn quotient sing}, (i)).  
When $p=5$,  a different singularity, denoted by $E^1_8$ in \cite{Artin 1977}, also has resolution graph $E_8$ and is a wild ${\mathbb Z}/5{\mathbb Z}$-quotient singularity.
\end{remark} 

\if false
\medskip
Recall that the equation $z^2+x^3+y^5=0$ defines the rational double point of type $E_8^0$ in characteristics $2,3,5$, and 
$z^3+x^2+y^4=0$ defines the rational double point of type $E_6^0$ in characteristic $3$.
\begin{proposition} The singularity $E_8^0$ in characteristic $p=2$ (resp. $3,5$) is a wild 
${\mathbb Z}/p{\mathbb Z}$-quotient singularity (resp. ${\mathbb Z}/2p{\mathbb Z} \times {\mathbb Z}/2{\mathbb Z}$,  ${\mathbb Z}/2p{\mathbb Z}$). The singularity $E_6^0$ in characteristic $p=3$ is a wild ${\mathbb Z}/2p{\mathbb Z}$-quotient singularity.
\end{proposition}
\proof
The case of $E_8^0$ when $p=2$ is treated in Theorem \ref{Brieskorn quotient sing} (i). 

Below the mistake is that the Z/2Z actions do not lift correctly. 
Consider now the case of $E_8^0$ when $p=3$. The wild ${\mathbb Z}/p{\mathbb Z}$-quotient singularity $z^3+(xy^3)^2z+x^4+y^{10}=0$ (see \ref{moderately ramified}) admits two involutions $x \mapsto -x$ and $y \mapsto -y$. The ring of invariants is given by the equation $z^3+(xy^3)z+x^2+y^5=0$, 
with Tjurina number $12$. A computation with Magma shows that the resolution of this singularity has dual graph $E_8$. Thus the singularity is isomorphic to $E_8^0$.
Similarly, consider the case of $E_6^0$ when $p=3$. The wild ${\mathbb Z}/p{\mathbb Z}$-quotient singularity $z^3+(xy)^2z+x^4+y^4=0$ admits the involution
$y \mapsto -y$. The ring of invariants is given by the equation $z^3+x^2yz+x^2+y^4=0$, 
with Tjurina number $9$. A computation with Magma shows that the resolution of this singularity has dual graph $E_6$. Thus the singularity is isomorphic to $E_6^0$.

Consider now the case of $E_8^0$ when $p=5$. Peskin's singularity introduced in \eqref{peskin equation}, with equation $z^5+2y^6-x^2 + z^2y^6 -2z^3y^4=0$ when $\mu=1$,
is a wild ${\mathbb Z}/p{\mathbb Z}$-quotient singularity, and it has the additional involution $y \to -y$. The ring of invariants is given by the equation
$z^5+2y^3-x^2 + z^2y^3 -2z^3y^2=0$. A computation with Magma shows that the resolution of this singularity has dual graph $E_8$. The Tjurina number of the singularity is $10$, and so this singularity is isomorphic to $E_8^0$.

Here the Z/2 action does not lift. There is a Z/10Z action, but the quotient has resolution of discriminant 10. 
The action is u--> -u, v--> -v, y--> -y, which induces x-->-x, so is not a lift of the y-->-y action used above.
\qed
\fi

 \begin{theorem} \label{thm.exhibit} Let $p $ be prime. Let $s\geq 0$. 
\begin{enumerate}[\rm (a)]
\item Assume that 
either 
$s \not \equiv 1 \mod{p}$, or that $p$ is odd and $s=1$. Then there exists a $\ZZ/p\ZZ$-quotient singularity $\Spec A^G$
with associated action ramified precisely at the origin, and such that the discriminant group of a resolution of the singularity has order $p^s$.
\item Assume that either $p$ is odd and that $s \equiv 1 \mod{p}$, or that $p=2$ and $s=1$. Then there exists a $\ZZ/p\ZZ$-quotient singularity $\Spec A^G$
with associated action ramified in codimension $1$ and such that the discriminant group of a resolution of the singularity has order $p^s$.
\end{enumerate}
\end{theorem}
\proof (a) The cases $s=0$  and $s=1$ are covered by Theorem \ref{antidiagonal E8} and Theorem \ref{resolution peskin}, respectively.
The cases with $s \geq 2$ and $s \not \equiv 1 \mod{p}$ were obtained earlier in the papers \cite{Lorenzini 2018} and  \cite{Mit}.

(b) When $s \equiv 1 \mod{p}$ and $s\geq p+1$, we use the Brieskorn singularities  exhibited in Lemma \ref{alls}, and apply
Theorem \ref{thm.BrieskornResolution} and Theorem \ref{Brieskorn quotient sing}. The case $p=2$ and $s=1$ was noted by Artin 
and is discussed in section \ref{analogues E7}. The case $s=1$ is  treated in Theorem \ref{Ap-1}. \qed

\begin{lemma} 
\label{alls}
Let $p$ be an odd prime, and $r$ be any positive integer. 
Then there are  integers $m,n>0$ such that the discriminant group $\Phi_N$ of
the intersection matrix $N$ associated with the Brieskorn singularity $z^p+x^{pm+1} +y^{pn+1}=0$ 
described in {\rm \ref{thm.BrieskornResolution}} is isomorphic to $({\mathbb Z}/p{\mathbb Z})^{pr+1}$.
\end{lemma}      

\proof
In view of Theorem \ref{thm.BrieskornResolution},
we need to produce integers $n$ and $m$ such that $\gcd(pn+1,pm+1) =pr+2$. 
For this, it suffices to take 
$n:=(pr+r+2)/2$, so that $pn+1=(pr+2)(p+1)/2$, and to set $m=n+(pr+2)$ so that $m:=(3pr+r+6)/2$. 
\qed

\medskip
Note that not all elementary abelian $p$-groups appear as discriminant groups $\Phi_N$ attached 
to the intersection matrix $N$ associated with a Brieskorn singularity $z^p+x^{pm+1} +y^{pn+1}=0$. 
Indeed, for all $m,n>0$, the integer $g=\gcd(pm+1,pn+1)$ is never divisible by $p$. Thus in the above setting
 $\Phi_N$ cannot be isomorphic to $({\mathbb Z}/p{\mathbb Z})^{pr-1}$ for any $r>0$.

\if false
\begin{remark}  
The Brieskorn singularity $B:=k[[x,y,z]]/(z^p-x^c-y^d)$ is known to be sandwiched between two regular local rings with purely inseparable morphisms 
of degree $p$, as follows:
$$\Spec k[[u,v]] \to \Spec B \to \Spec k[[x,y]],$$ 
where we set 
$x:=u^p$, $y:=v^p$, and $z:=u^c+v^d$. It is clear that $z^p=x^c+y^d$. 
In fact,  there is an action of the group ${\mathbb \alpha}_p/k$ on $\Spec k[[u,v]]]$ such that $\Spec B$ is isomorphic over $k$ 
to the quotient $(\Spec k[[u,v]]])/{\mathbb \alpha}_p$.
Indeed, let $A:=k[[u,v]]$. Consider the following $k$-derivation $D: A \to A$. Set $D(u):=dv^{d-1}$ and $D(v):= -cu^{c-1}$. Then $D(z)=0$, and we find that ${\rm Ker}(D) \subseteq B$. We claim that in fact ${\rm Ker}(D) = B$. Such derivation can be used to endow $\Spec A$ with an action of the group scheme ${\mathbb \alpha}_p$, and   
the quotient $(\Spec A)/{\mathbb \alpha}_p$ can be identified with $\Spec B$. A summary of properties of such actions can be found in 
\cite{Schroeer 2007}, sections 1 and 2.
\end{remark}
\fi

\begin{remark} \label{remark.classgroup}
Let $B$ be a complete noetherian local  ring that is two-dimensional and normal, with algebraically closed residue field.
Consider a resolution  of singularities $X \to \Spec B$, with associated intersection matrix $N$. Recall that
there is a natural surjection ${\rm Cl}(B) \to \Phi_N$ (see \cite{Lip}, 14.4). In particular, when $\det(N) \neq 1$,
we obtain  a natural non-trivial finite quotient of ${\rm Cl}(B)$ from the computation of a resolution of $\Spec B$.

\if false
Let us recall here how the group $\Phi_N$ is related with the class group ${\rm Cl}(A^G)$. Consider a resolution $\pi:X \to \Spec A^G$. 
Let  $C_1,\dots, C_n$ denote the  irreducible components in the exceptional divisor of $\pi$.
Let $e(C_i)$ denote the geometric multiplicity of $C_i/k$. When $k$ is algebraically closed, $e(C_i)=1$ for $i=1,\dots, n$. As usual, $N:=((C_i \cdot C_j)_X)$ 
is the symmetric intersection matrix, and we let $N' \in M_n({\mathbb Z})$ denote the following modified matrix: 
$$N':={\rm diag}(1/e_1,\dots, 1/e_n)N.$$
Let $\Phi':={\mathbb Z}^n/{\rm Im}(N').$
Let $U:=X \setminus \{ \cup_{i=1}^n C_i\}$.
When $A^G$ is Henselian, recall that ${\rm Pic}(U)$ is isomorphic to ${\rm Cl}(A^G)$, and that we have a natural surjection
${\rm Pic}(U) \to \Phi'$ (see \cite{Lip}, 14.4).
\fi 

The study of the class group ${\rm Cl}(B)$ of 
$B:=k[[x,y]][z]/(z^p-x^c-y^d)$
was initiated by Samuel in \cite{Sam}, Proposition (3) in section 6 (see also \cite{Fos}, Chapter IV, section 17).
When $p=2$, Samuel is able to exhibit by a completely algebraic method a finite quotient of ${\rm Cl}(B)$ of order $p^{g-1}$, where $g:=\gcd(c,d)$.
Under the hypothesis of Theorem \ref{thm.BrieskornResolution}, $p^{g-1}$ would also be the order of the corresponding group $\Phi_N$.
\end{remark}

\if false
Let $k$ be a field of characteristic $p>0$. Set $R=k[[x,y]]$, and let $a,b\in R$ be
a system of parameters. In \cite{LS1}, we introduced the class of hypersurface singularities
given by an   equation of the form
$$
z^p-(ab)^{p-1}z -a^py + b^px =0.
$$
The resulting local ring $B=k[[x,y,z]]/(z^p-(ab)^{p-1}z -a^py + b^px)$ is the 
ring of invariants $B=A^G$ for some regular local ring $A=k[[u,v]]$
endowed with an action of $G=\ZZ/p\ZZ$ given by $u\mapsto u+a$ and $v\mapsto v+b$.
This action is ramified only at the origin and satisfies further peculiar properties.
We called such actions \emph{moderately ramified}.
More generally, we studied equations of the form
$$
z^p-(\mu ab)^{p-1}z -a^py + b^px =0,
$$
where $\mu\in R$ is an additional formal power series. 
If $\mu$ is non-zero and coprime to $a$ and $b$  
the local ring $B$ remains  the ring of invariants $B=A^G$,
where $A=k[[u,v]]$ and  the action is given by $u\mapsto u+\mu a$ and $v\mapsto v+\mu b$, which  is ramified
along the subscheme defined by the ideal $\mu(a,b)\subset A$.

In the special case $\mu=0$ and monomials $a=y^n$ and $b=x^m$, our equation becomes
$z^p-y^{pm+1} + x^{pm+1}=0$. This is a weighted homogeneous equation,
and the resulting singularities are called \emph{Brieskorn singularities}.
Their resolution over the complex numbers
is described in \cite{O-W}.

Let $f:X\ra \Spec(B)$ be the minimal   resolution of singularities, $E_1+\ldots+E_r\subset X$
be the   the exceptional divisor, $N=(E_i\cdot E_j)$
the resulting intersection matrix, $\Phi_N=\ZZ^r/N\ZZ^r$ the discriminant group,
and $\Gamma_N$ the dual graph for the exceptional divisor.
In characteristic two, the special case  $a=y^2$ and $b=x$ gives
rational double points of type $E_8$.
It follows from the algorithm given in \cite{Greuel; Kroening 1990}, Section
that is is an $E_8^2$ singularity if and only if $\mu\in R^\times$.
Otherwise, it is either $E_8^0$ or $E_8^1$, and the singularities become
simply-connected. In any case,   the discriminant group $\Phi_N$ is trivial.
Our main result here is that the triviality of the discriminant group extends to arbitrary primes $p>0$.
In particular, wild quotient singularities with $\det(N)=\pm 1$ indeed do exist.
Recall that they do  not  occur for diagonal actions on products of curves.
\fi

\section{Analogues of the \texorpdfstring{$E_6$}{E6} singularities}
\mylabel{analogues E6}

Let $k$ be an algebraically closed field of characteristic $p\geq 3$. 
Let $\ttt \in k[y]$, $\ttt \neq 0$.
Consider the  automorphism $\sigma$ of the polynomial ring $k[u,v,y]$ given by
$$
u\longmapsto u+\ttt v, \quad v\longmapsto v+\ttt y,\quadand y\longmapsto y.
$$
This automorphism has order $p$. We exclude the case $p=2$ in this section because when $p=2$, $\sigma$ has order $4$.
Let 
$$
N_u:= {\rm Norm}(u) =  \prod_{d=0}^{p-1} \sigma^d(u)  =   \prod_{d=0}^{p-1}\left(u+d\ttt v+\frac{d(d-1)}{2}\ttt^2 y\right),
$$
and 
$$
x:= {\rm Norm}(v) =v^p-(\ttt y)^{p-1}v.
$$
Finally, let
$$z:=v^2 -\ttt yv-2yu.
$$

Let $G:={\mathbb Z}/p{\mathbb Z}$ act on $k[u,v,y]$ through $\sigma$. When $\ttt =1$, the ring of invariants 
$ k[u,v,y]^G$ is   known to be generated by $x$, $y$, $z$, and $N_u$, subject to a single relation (see e.g., \cite{C-W}, 4.10). This relation was made explicit by Peskin who showed in \cite{Peskin 1983}, Lemma 5.6,  that   $h_{\ttt =1}(x,y,z, N_u)=0$, where,  
\begin{equation*}
h_{\ttt =1}(x,y,z, N_u) :=z^p +   2y^p N_u    - x^2 + \sum_{n=2}^{(p+1)/2} (-1)^nC_{n-1} y^{2p-2n}z^n.
\end{equation*}
Here  $C_{n-1}:=(2n-2)!/n!(n-1)!$ are the Catalan numbers.

When $\ttt \neq 1$, the above result can be used to show that  
$x$, $y$, $z$, and $N_u$, are subject to the relation  
\begin{equation*}
h(x,y,z, N_u) :=z^p +   2y^p N_u    - x^2 + \sum_{n=2}^{(p+1)/2} (-1)^nC_{n-1} (\ttt y)^{2p-2n}z^n=0.
\end{equation*}
Indeed, the morphism $k[u,v,y] \to k[U,V,y]$, which sends $u \mapsto \mu^2U$, $v \mapsto \mu V$, and $y \mapsto y$, is $G$-equivariant
when $k[u,v,y] $ is endowed with the action of $\sigma$, and $k[U,V,y]$ is endowed with the action of $\sigma_1$, with $\sigma_1(U)=U+V$ and $\sigma_1(V)=V+y$.

For any choice of $c(y) \in yk[y]$, we can consider the ring 
$$
\PeskinA:=k[u,v,y]/(N_u-c(y)).
$$
We will slightly abuse notation and denote again by $x$, $y$, $z$, $u$, $v$, the classes of these elements in $A_0$.
Clearly, the automorphism $\sigma$ fixes the polynomial $N_u-c(y)$, and thus induces an automorphism on
$\PeskinA$, again denoted by $\sigma$. This endows $\PeskinA$ with an action of $G$. 
Let $A$ denote the formal completion $\widehat{\PeskinA}$ of the ring $\PeskinA$ at the maximal ideal $(u,v,y)$. 

The fixed scheme of the $G$-action on $\Spec(\PeskinA)$ is given by the ideal $I:=(\ttt v,\ttt y)$.
When $\ttt \in k^*$,
$I=(v,y)=( u^p,v,y)$, and thus its radical is the maximal ideal $(u,v)$.
Hence, the morphism $\Spec A \to \Spec A^G$ is ramified precisely at the origin. 
When 
$\ttt \neq 0$   is not a unit,  the morphism $\Spec A \to \Spec A^G$ is ramified in codimension $1$. 

The study of the singularities of the rings $\Spec A^G$ when $\mu=1$ was initiated by Peskin
(\cite{Peskin 1980}, Chapter III, \S 4, and  \cite{Peskin 1983}, Section 5). 
In this section we treat the case where $c(y)=y$, and obtain in this way a family of wild quotient singularities $A^G$ of multiplicity $2$ whose discriminant groups have order $|\Phi|=p$.
For $p>3$, these singularities can be viewed as analogues of the rational double point of type $E_6^1$ in characteristic $p=3$,
which was shown to be a wild $\ZZ/3\ZZ$-quotient singularity by Artin \cite{Artin 1977}.

\begin{proposition}
\mylabel{peskin completion} Let $c(y):=y$. Let $\ttt \in k[y]$. 
Then the ring $\PeskinA $ is a domain, the formal completion $A$
is regular, and the canonical map $k[[u,v]]\ra A$ is bijective.
\end{proposition}

\proof
The expression $f(u,v,y):=N_u-y$ is a monic polynomial of degree $p$ in the variable $u$ over the factorial ring $k[v,y]$, 
with constant term $f(0,v,y)=-y$. Its Newton polygon with respect to the $y$-adic valuation is
thus the straight line from $(0,0)$ to $(p,1)$ in $\RR^2$, and we conclude
with the Eisenstein--Dumas Theorem \cite{Mott 1994} that $f$ is irreducible as a polynomial  
over $k(v,y)$, and with the Gau\ss{} Lemma that it is also irreducible as  a polynomial over  $k[v,y]$.

The ring $\PeskinA$ and its formal completion $A$ are thus two-dimensional domains.
To see that the local ring $A$ is regular, we have to check that the cotangent space $\maxid_A/\maxid_A^2$ has vector space
dimension at most two. Indeed, this vector space is generated by $u,v,y$.
In light of the relation $N_u-y=0$, the class of $y$ vanishes. 
In turn, the canonical map $k[[u,v]]\ra A$ between complete local rings
induces a bijection on cotangent spaces, and is thus bijective.
\qed

\medskip Let $\ttt \in k[y]$.
Abusing notation slightly, 
we let $h(x,y,z) \in k[x,y,z]$ be defined as
\begin{equation}
\label{peskin equation}
h(x,y,z) :=z^p +   2y^{p+1} - x^2 + \sum_{n=2}^{(p+1)/2} (-1)^nC_{n-1} (\ttt y)^{2p-2n}z^n.
\end{equation}
We let 
$B_{\ttt}:= k[[x,y,z]]/(h)$.

\begin{proposition} 
Let $\ttt \in k[y]$, $\ttt \neq 0$.
Then the canonical map $k[[x,y,z]]/(h)\ra A^G$ is bijective. In particular, the wild quotient singularity $A^G$ 
is a complete intersection of multiplicity two.
\end{proposition}

\proof
Both local rings $k[[x,y,z]]/(h)$ and $A^G$ are Cohen--Macaulay, and  finite $k[[x,y]]$-algebras of rank $p$.
One easily sees that $h(x,y,z)=0$ defines an isolated singularity, by using the relations $h_x=-2x$ and  $2z(\mu+y\mu_y)h_z+yh_y=2y^{p+1}$
between partial derivatives.
It follows that $k[[x,y,z]]/(h)$ is normal, and that the canonical map induces a bijection on the field of fractions.
The map in question is thus bijective, by Zariski's Main Theorem.
Clearly, the monomial $x^2$ is the lowest term in $h(x,y,z)$, and it follows that the complete intersection $A^G$ has multiplicity two.
\qed

\begin{theorem}
\label{resolution peskin} 
Let $\mu \in k[y]$.
Let $X\ra\Spec(B_{\ttt})$ be the minimal resolution of singularity, with associated intersection matrix $N$.
Then the dual graph $\Gamma_N$ is independent of $\ttt$, and   
takes
the form:
$$
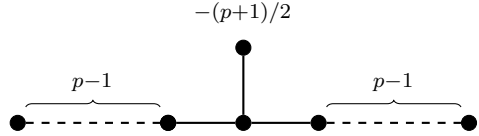
\begin{tikzpicture}
[node distance=1cm, font=\small] 
\tikzstyle{vertex}=[circle, draw, fill, inner sep=0mm, minimum size=1.1ex]
\node[vertex]	(v1)  	at (0,0) 	[label=below:{}] 		{};
\node[]	(dummy1)		[right of=v1, label=above:{}]	{};
\node[vertex]	(v2)			[right of=dummy1, label=below:{}]	{};
\node[vertex]	(v2)			[right of=dummy1, label=below:{}]	{};
\draw [decorate,decoration={brace, raise=6pt}] (v1)--(v2) node [black, midway,xshift=-0cm,yshift=0.5cm] {$^{p-1}$};
\node[vertex]	(v3)			[right of=v2, label=below:{}]	{};
\node[vertex]	(v)			[above of=v3,  label=above:{ $^{-(p+1)/2}$}]	{};
\node[vertex]	(v4)			[right of=v3, label=below:{}]	{};
\node[]	(dummy2)		[right of=v4, label=above:{}]	{};
\node[vertex]	(v5)			[right of=dummy2, label=below:{}]	{};
\draw [decorate,decoration={brace, raise=6pt}] (v4)--(v5) node [black, midway,xshift=-0cm,yshift=0.5cm] {$^{p-1}$};
\draw [thick, dashed] (v1)--(v2);
\draw [thick] (v2)--(v3)--(v4);
\draw [thick] (v)--(v3);
\draw[thick, dashed] (v4)--(v5);
\end{tikzpicture}
$$
The associated discriminant group $\Phi_N$ has order $p$.
\end{theorem}

\proof
Consider the blow-up  $Z \to \Spec(B_{\ttt})$ of $\Spec(B_{\ttt})$ with respect to the ideal $(x,y,z)$. Let $Y \to Z$ denote the normalization of $Z$.
Let $E$ denote the exceptional divisor of the blow-up, and let $D$ denote its schematic preimage in $Y$.

The blow-up $Z$ is covered by three charts that we call the $x$-chart, $y$-chart, and $z$-chart.
We consider in detail below the $y$-chart and show that its normalization contains a unique singular point $y_0$. 
Proceeding in an analogous way as for the $y$-chart, 
the reader will check that the normalizations of the $x$-chart and the $z$-chart are regular.

On the $y$-chart, the  strict transform of $h(x,y,z)=0$ becomes
$$
\left(\frac{z}{y}\right)^py^{p-2} + 
2y^{p-1} - \left(\frac{x}{y}\right)^2 +
\sum_{n=2}^{(p+1)/2} (-1)^nC_{n-1} \ttt^{2p-2n} y^{2p-n-2}\left(\frac{z}{y}\right)^n  =0.
$$
The fraction $x/y^{(p-1)/2}$ satisfies the integral equation
\begin{equation}
\label{normalized blowing-up}
\left(\frac{z}{y}\right)^py  + 2y^2 - \left(\frac{x}{y^{(p-1)/2}}\right)^2 +
\sum_{n=2}^{(p+1)/2} (-1)^nC_{n-1} \ttt^{2p-2n} y^{p-n+1}\left(\frac{z}{y}\right)^n =0.
\end{equation}
Write $g=(z/y)^py + 2y^2 -  (x/y^{(p-1)/2})^2 + \ldots$ for the   polynomial on the left.
Up to radical, its jacobian ideal contains $y$, because this defines the exceptional divisor on the $y$-chart
and there are no singularities outside. Obviously it also contains $x/y^{(p-1)/2}$. 
Using the partial derivative $g_y = (z/y)^p + \ldots $, we see that it furthermore contains $z/y$.
Thus the normalization of the $y$-chart is given by the three variables $z/y,y,x/y^{(p-1)/2}$
and the equation $g=0$. 

{\it We claim that $D_{\red}$ is a smooth rational curve, and that $(D_{\red}\cdot D_{\red})_Y = -1/2$.}
For this it suffices to check analogously as in Proposition \ref{exceptional divisor}  that the curve $E_{\red}$ is regular, and that $(E \cdot E_{\red})_Z=-1$. 
Then one checks that the natural map $D_{\red} \to E_{\red}$ is an isomorphism. Finally, noting that the multiplicity of $E$ is $\ell=2$, 
we apply the formula $(D_{\red}\cdot D_{\red})_Y = (E \cdot E_{\red})_Z/\ell$ in Proposition \ref{correction degree} to obtain the claim.

Regarded as a formal power series, the initial term of $g$ is the quadratic polynomial 
$2y^2-(x/y^{(p-1)/2})^2$, which is thus a product of two linear factors since $k$ is algebraically closed.
According to Lemma \ref{detection rdp} below,
the singularity must be a rational double point of type $A_m$ for some integer $m\geq 1$.
To determine this integer, we compute the Tjurina number of the singularity, which is the colength
of the ideal generated by $g$ and its partial derivatives.
Setting $x'=x/y^{(p-1)/2}$ and $z'=z/y$,  the partial derivatives take the   form
$$
g_{x'} = 2x',\quad g_y = z'^p+ y\cdot\text{unit} \quadand g_{z'} = \sum_{n=2}^{(p+1)/2} (-1)^nnC_{n-1}  \ttt^{2p-2n} y^{p-n+1}  z'^{n-1}.
$$
We now use $g_y=0$ to substitute for $y$ in the equations $g(0,y,z')=0$ and $g_{z'}(0,y,z')=0$,
and infer that the jacobian ideal has colength $\tau = 2p$.
Recall that the Tjurina number for the $A_m$-singularity, which is formally isomorphic to $Z^{m+1}-XY=0$,
is given by 
$$
\tau = \begin{cases}
m 	& \text{if $p$ does not divide $m+1$;}\\
m+1	& \text{else.}
\end{cases}
$$
It follows that either $m=2p-1$ or $m=2p$, and we shall see below that $m$ is odd.

Write $X\ra Y$ for the minimal resolution of singularities of the rational double point, 
such that the composite map $X\ra Y \to \Spec(B_{\ttt})$ is a resolution of the singularity.
The dual graph of this resolution contains a chain $C_1,\ldots,C_m$     of 
$(-2)$-curves, together with the strict transform $C_0$ of the divisor $D_{\red}$ on $Y$.

Suppose that $C_0$ intersects two distinct exceptional curves $C_i\neq C_j$. 
Then $(\bigcup_{i \geq 1} C_i)\cap C_0$ is an Artin scheme of length $\geq 2$ on $C_0$. We claim that this is not possible. 
Indeed, consider the blow-up $X \to Y$. The induced morphism $C_0 \to D_{\red}$ is an isomorphism since we have shown above that the point $y_0$ is a regular point of $D_{\red}$.
The scheme $(\bigcup_{i \geq 1} C_i)$, which is proper, has schematic image in $Y$ the reduced closed point $y_0$. The same is true for any closed subscheme of the exceptional divisor, including the subscheme $(\bigcup_{i \geq 1} C_i)\cap C_0$. This is a contradiction since we have on the other hand an isomorphism 
$C_0 \to D_{\red}$, and a closed subscheme of length bigger than one in the source cannot be sent to a closed subscheme of length $1$ in the target.
Thus  $C_0$ hits precisely one divisor $C_i$. If $(C_0\cdot C_i)_X>1$, a similar argument leads again to a contradiction, 
and thus we must have $(C_0\cdot C_i)_X=1$.

Consider now the involution on $B_{\ttt}$ given
by $x\mapsto -x$, $y\mapsto y$ and $z\mapsto z$. This involution fixes Peskin's equation \eqref{peskin equation},
and induces an involution on the initial blow-up $Z$ and its normalization $Y$.
There the  equation $z/y=0$ defines an invariant Cartier divisor on the $A_m$-singularity $\Spec \calO_{Y,y_0}$,
which is the union  of two regular Weil divisors $D_1$ and $D_2$, and these divisors are interchanged by the involution.
The blow-up $Y'\ra Y$ of the singular point $y_0\in Y$ with reduced structure introduces two
exceptional curves $F_1$ and $F_2$, and the strict transforms of  $D_1$ and $D_2$ in $Y'$ are disjoint.
The intersection $F_1\cap F_2$ consists in a single point $y_0'$, and the local ring $\calO_{Y',y_0'}$ is a rational
double point of type $A_{m-2}$.

{\it We now show that $m $ is odd.} First, suppose that the strict transforms of $D_1$ and $D_2$ in $Y'$ do not intersect the same 
exceptional component of the blow-up $Y'\to Y$. 
It then follows that the involution acts non-trivially on 
the dual graph attached to the resolution of singularities $X\ra Y$. 
 If $m=2p$ was even, the curve $C_0$ would pass through the sole fixed point $C_p\cap C_{p+1}$ of the exceptional divisor, and as we have seen above, this is a contradiction.
It follows that $m=2p-1$ must be odd in this case,   and that $(C_0\cdot C_p)_X= 1$.
The assertion on the dual graph $\Gamma_N$ follows.

Suppose now that the strict transforms of $D_1$ and $D_2$ in $Y'$ intersect the same 
exceptional component of the blow-up $Y'\to Y$. We are going to show that this case cannot happen.
Indeed, then the Weil divisors $D_1,D_2\subset Y$ define the same class in the class group $\Cl(\calO_{Y,y_0})=\ZZ/(m+1)\ZZ$
of the rational double point of type $A_m$.
Since the curves $D_i$ are regular, the divisors $D_i\subset Y$ are not Cartier. It follows that $D_i$ has order two in $\Cl(\calO_{Y,y_0})$
since the sum of $D_1$ and $D_2$ is a Cartier divisor on $Y$.
On the other hand, the strict transform of $D_i$ in $X$ intersects a terminal vertex of the exceptional divisor of $X \to Y$, and this fact along with a computation using the intersection matrix of the chain of $m$ curves 
implies that $D_i$ has order $m+1$ in the class group. This gives $m=1$, contradicting
$m\geq 2p-1\geq 5$.

To completely determine the intersection matrix $N$ of the resolution $X \to \Spec(B_{\ttt})$, it remains to compute the self-intersection number $(C_0 \cdot C_0)_X$. We have already observed above that $(D_0 \cdot D_0)_Y = -1/2$, and Proposition \ref{correction self-intersection}
shows that $(C_0 \cdot C_0)_X=(D_0 \cdot D_0)_Y - \delta$, where the correcting term $\delta$ is computed as follows.
The determinant of the intersection matrix of the full chain of length $2p-1$ is $-2p$. 
Removing the vertex adjacent to $C_0$ from this chain yields two chains of length $p-1$. The determinant of the 
associated intersection matrix is then $p^2$.
It follows that $\delta = p^2/2p=p/2$. Hence,
$$
(C_0 \cdot C_0)_X=-1/2 - p/2 = -(p+1)/2.
$$
Proposition \ref{determinant star-shaped} shows that $|\Phi_N|=p$.
\qed

\medskip
In the course of the proof we have used the following well-known general observation:

\begin{lemma}
\mylabel{detection rdp}
Let $f\in k[[x,y,z]]$ by  a power series over an arbitrary field $k$. Write $f = \sum_{j=0}^\infty f^{(j)}$,
where $f^{(j)}$ is a homogeneous polynomial of degree $j$.
Suppose that $f^{(0)}=f^{(1)}=0$, and that $f$ defines an isolated singularity.
Assume also that the quadratic part $f^{(2)}$ is the product of two   non-associated linear forms.
Then   $k[[x,y,z]]/(f)$ is isomorphic to $ k[[x,y,z]]/(z^{m+1}-xy)$ for some integer $m\geq 2$. 
In other words, the singularity in question is a rational double point of type $A_m$.
\end{lemma}

\proof
After a linear change of coordinates, we may assume that $f=xy+O(3)$, where we denote by $O(d)$ an element of $\maxid^d$.
By induction on $d\geq 3$, one makes  further coordinate changes of the form $x':=x+a(x,y,z)$, $y':=y+b(x,y,z)$
with $a,b\in\maxid^{d-1}$
sending $f$ to a power series of the form $x'y' + \sum_{i=3}^d\lambda_iz^i + O(d+1)$.
This shows that we may assume $f=xy+\sum_{i=3}^\infty\lambda_iz^i$.
If all coefficients $\lambda_i$ vanish, the singularity would not be isolated.
Thus our equation is of the form $xy+z^{m+1}\epsilon$ for some $m\geq 2$ and  unit $\epsilon$.
Multiplying with $\epsilon^{-1}$, we get the equation $(\epsilon^{-1}x)y+z^{m+1}$
for the rational double point of type $A_m$.
\qed

\if false
\begin{remark}
It $2\in k^\times$ is not a square, the field $k'=k(\sqrt{2})$ defines a Galois extension of degree two.
Over the ground field extension $k'$,  the Galois group acts on the dual graph in Theorem \ref{resolution peskin}
by interchanging the two terminal chains of length $p-1$.
Over the ground field $k$, the minimal resolution has an exceptional divisor $E'+E'_1+E'_2+\ldots+E'_p$,
with $E'=E'_1=\PP^1_k$, and $E'_i=\PP^1_{k'}$ for $i\geq 2$.
\end{remark}
\fi

\medskip 
The description of the fundamental cycle ${\mathbf Z}$ in our next proposition follows from the general description in \cite{Tom}
of the fundamental cycle of a star-shaped dual graph.
We leave the proof of the following proposition to the reader.
\begin{proposition}
\mylabel{fundamental cycle E6}
The multiplicities in the fundamental cycle ${\mathbf Z}$ 
of the resolution of $\Spec B_{\ttt}$ are  indicated below next to the corresponding vertex.
$$
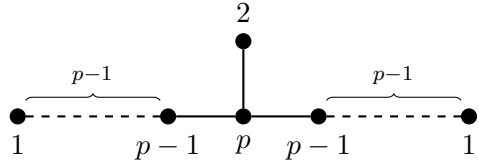
\begin{tikzpicture}
[node distance=1cm, font=\small] 
\tikzstyle{vertex}=[circle, draw, fill, inner sep=0mm, minimum size=1.1ex]
\node[vertex]	(v1)  	at (0,0) 	[label=below:{$1$}] 		{};
\node[]	(dummy1)		[right of=v1, label=above:{}]	{};
\node[vertex]	(v2)			[right of=dummy1, label=below:{$p-1$}]	{};
\draw [decorate,decoration={brace, raise=6pt}] (v1)--(v2) node [black, midway,xshift=-0cm,yshift=0.5cm] {$^{p-1}$};
\node[vertex]	(v3)			[right of=v2, label=below:{$p$}]	{};
\node[vertex]	(v)			[above of=v3,  label=above:{ ${2}$}]	{};
\node[vertex]	(v4)			[right of=v3, label=below:{$p-1$}]	{};
\node[]	(dummy2)		[right of=v4, label=above:{}]	{};
\node[vertex]	(v5)			[right of=dummy2, label=below:{$1$}]	{};
\draw [decorate,decoration={brace, raise=6pt}] (v4)--(v5) node [black, midway,xshift=-0cm,yshift=0.5cm] {$^{p-1}$};
\draw [thick, dashed] (v1)--(v2);
\draw [thick] (v2)--(v3)--(v4);
\draw [thick] (v)--(v3);
\draw[thick, dashed] (v4)--(v5);
\end{tikzpicture}
$$
\if false
$$
\begin{matrix}
1	& 	& 2	& \ldots  	& p 	 	& \ldots  & 2 & 1\\
  	&       &	& 		& 2 \\
\end{matrix}
$$
\fi
We have ${\mathbf Z}^2=-2$, and the  fundamental genus  is  $h^1(\O_{\mathbf Z})=(p-3)/2$.
The canonical cycle is given by   $K=-\frac{p-3}{2}{\mathbf Z}$. 
\end{proposition}
\if false
\proof
One may apply Tomaru's results for star-shaped dual graphs \cite{Tom}, as in the proof for Proposition \ref{fundamental cycle E8}.
Here, however, it is simpler to apply   Artin's Algorithm \cite{Artin 1966}:
One starts with the cycle $C$ having constant coefficients  
$
\begin{smallmatrix} 
1	& 	& 1	& \ldots  	& 1	 	& \ldots  & 1 & 1\\
  	&       &	& 		& 1
\end{smallmatrix}$.
This has strictly positive intersection number on the node $E_0\in\Gamma_N$.
Thus one increases the corresponding multiplicity $m_0$ by one. The new cycle $C$ has positive intersection number 
on the adjacent vertices of the node in the two terminal chains of length $p-1$,
and one increases their multiplicities by one.
Proceeding along these terminal chains,
one ends with the cycle 
$
\begin{smallmatrix} 
1	& 	& 2	& \ldots  	& 2 	 	& \ldots  & 2 & 1\\
  	&       &	& 		& 1
\end{smallmatrix}$.
Now one repeats the process, starting again at the   node $E_0$.
After $p-1$ steps, one obtains the cycle
$
\begin{smallmatrix} 
1	& 	& 2	& \ldots  	& p 	 	& \ldots  & 2 & 1\\
  	&       &	& 		& 1
\end{smallmatrix}$.
This   has positive intersection number with the terminal vertex $E_1\in\Gamma_N$ with self-intersection number
$E_1^2=-(p+1)/2$. Increasing $m_1$ by one gives the fundamental cycle:
It has $(Z\cdot E_1)=-1$, and all other intersection numbers are trivial.

These values for $(Z\cdot E_i)$ immediately give $Z^2=2$.
Furthermore, it follows that the canonical cycle must be  $K=-lZ$, with $l=\frac{p-3}{2}$. The  Adjunction Formula gives
$$
\deg(K_Z) = (K+Z)\cdot Z = (1-l) Z^2 = 2l-2=p-5.
$$
Writing $\deg(K_Z) = 2g-2$, where $g=h^1(\O_Z)$ is the fundamental genus,
we get the desired value $g=(p-3)/2$.
\qed
\fi

\section{Analogues of the \texorpdfstring{$E_8$}{E8} singularities}
\mylabel{analogues E8}

Let $k$ be an algebraically closed field of characteristic $p>0$.
We compute in this section the resolution of the singularity of $\Spec B_{\mu}$
introduced in \ref{moderately ramified}, for any value of the parameter $\mu\in k[[x,y]]$ when $a=y^2$ and $b=x$. 
The ring $B_{\mu}$ is given in this case by 
$$
B_\mu:=k[[x,y]][z]/(z^p-(\mu xy^2)^{p-1}z - x^{p+1}+y^{2p+1}).
$$

When $p=2$, the resolution of $\Spec B_{\mu}$ is known to have  dual graph $E_8$ when 
$\mu=0$, $\mu=1$ and $\mu=y$: these values produce the rational double points
$E_8^0$, $E_8^2$, and $E_8^1$, respectively (\cite{Artin 1977}; see also \cite{Peskin 1980}).
The index of determinacy of a singularity $E^r_8$ in characteristic $2$ is computed to be $5$ 
in \cite{Greuel; Kroening 1990}, page 346. It follows that when $\mu \in (x,y)^{2}$, then $\Spec B_{\mu}$ is isomorphic to $E^0_8$.
For $\mu \in k^\times$, we find that $B_{\mu}$ is isomorphic to $E^2_8$ through the change of variables $X=\mu^{15/11}x$, $Y=\mu^{10/11}y$, and $Z=\mu^{6/11}z$.

\begin{theorem}
\mylabel{antidiagonal E8}
Let $p \geq 3$. Then $\Spec B_{\mu}$ has a resolution of singularities independent of $\mu$
with the following dual graph $\Gamma_N$:

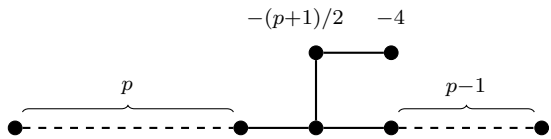
\begin{equation*}
\begin{gathered}
\begin{tikzpicture}
[node distance=1cm, font=\small]
\tikzstyle{vertex}=[circle, draw, fill,  inner sep=0mm, minimum size=1.0ex]
\node[vertex]	(E1)  	at (0,0) 	[label=below:{}]               {};
\node[]	        (E11)  	         	[right of=E1,label=above:{}] {};
\node[]	        (E2)  	         	[right of=E11,label=above:{}]               {};
\node[vertex]	(E3)  	         	[right of=E2,label=below:{}]               {};

\node[vertex]	(D)			[right of=E3, label=right:{ }]    {};

\node[vertex]	(E4)			[above of=D, label=above:{$^{ -(p+1)/2  \quad \ }$}]    {};
\node[vertex]	(E5)			[right of=E4, label=above:{$^{-4}$}]    {};

\node[vertex]	(E6)			[right of=D, label=below:{}]    {};
\node[]	        (E7)			[right  of=E6, label=above:{}]    {};
\node[vertex]	(E8)			[right  of=E7, label=below:{}]    {};
\draw [decorate,decoration={brace, raise=6pt}] (E6)--(E8) node [black, midway,xshift=-0cm,yshift=0.5cm] {$^{p-1}$};

\draw [decorate,decoration={brace, raise=6pt}] (E1)--(E3) node [black, midway,xshift=-0cm,yshift=0.5cm] {$^p$};
 
\draw [thick,dashed] (E1)--(E3);
\draw [thick,dashed] (E6)--(E8);

\draw [thick] (E3)--(D)--(E4)--(E5);
\draw [thick] (D)--(E6);

\end{tikzpicture}
\end{gathered}
\end{equation*}
The associated discriminant group $\Phi_N$ is trivial. 
\end{theorem}

\proof
Set $R=k[[x,y,z]]$ and $f=z^p-(\mu xy^2)^{p-1}z - x^{p+1}+y^{2p+1}$ and write $B=R/(f)$.
We start with an initial blowing-up $Z=\Bl_{\ideala B}(B)$ for the ideal $\ideala=(x,y^2,z)$,
as in \ref{exceptional divisor}.
As usual, let $E\subset Z$ denote the exceptional divisor of the blow-up, and $E_\red$ its reduction.
Proposition \ref{exceptional divisor} shows that $E_\red$ is a smooth rational curve, that $E=2pE_\red$, and that $(E \cdot E_\red)_Z=-1$.
One checks that the blow-up is regular on the $y^2$-chart and the $z$-chart,
and contains a unique singularity, which is located
at the origin of the $x$-chart.

The $x$-chart is given by four variables $x,y,y^2/x,z/x$ modulo the two relations
$$
y^2=\left(\frac{y^2}{x}\right)x \quadand 
\left(\frac{z}{x}\right)^p - \mu^{p-1}x^{p-1}\left(\frac{y^2}{x}\right)^{p-1}\frac{z}{x} - x + \left(\frac{y^2}{x}\right)^py =0.
$$
The exceptional divisor is given by $x=0$. Its reduction is defined by $x=y=z/x=0$.
Let us rewrite the second equation above as 
\begin{equation}
\label{help}
\left(\frac{z}{x}\right)^p + \left(\frac{y^2}{x}\right)^py =x( \mu^{p-1}x^{p-2}\left(\frac{y^2}{x}\right)^{p-1}\frac{z}{x} +1).
\end{equation}
On the formal completion along the  exceptional divisor, 
$1+\mu^{p-1}x^{p-2}(y^2/x)^{p-1}(z/x)$ is invertible, and we denote by $\epsilon $ its inverse. 
The unit $\epsilon$ admits a $(p+1)$-st root $\delta$ (with $\delta^{p+1}=\epsilon$). 
After extracting an expression for $x$ from \eqref{help} and substituting it in the expression $y^2=\frac{y^2}{x}x$, 
we find that
$$
y^2= \frac{y^2}{x}  \left( \left(\frac{z}{x}\right)^p + \left(\frac{y^2}{x}\right)^p y\right)\epsilon. 
$$
This is formally isomorphic to the equation
$$
y^2 - U^{p+1}y - UW^p=0
$$
in the  new set of variables $y,U,W$, via the map given by
$y\mapsto y$, $U\mapsto (y^2/x)\delta$ and $W\mapsto (z/x)\delta$. 
Note that the reduced exceptional divisor is given by $x=y=z/x=0$ in the old coordinates,
and by  $y=W=0$ in the new ones.
Let 
$$
B':=k[[y,U,W]]/(y^2 - U^{p+1}y - UW^p).
$$
We now make a  second blow-up $Z' \to \Spec(B')$, with nonreduced center given by $(y,U,W^p)$. 
Let $E'$ denote the exceptional divisor of this blow-up. Using Proposition \ref{blowing-up computation}, we infer
that the $U$-chart of $Z'$ is described by four  variables $U,W,y/U,W^p/U$ and two relations
$$
W^p = \left(\frac{W^p}{U}\right) U \quadand
\left(\frac{y}{U}\right)^2 - U^p\left(\frac{y}{U}\right) - \frac{W^p}{U} =0.
$$
Substituting the latter in the former and renaming $y/U$ by $V$ gives
\begin{equation}
\label{secondary blowing-up}
W^p = U V\left( V  - U^{p}\right).
\end{equation}
The origin $(U,V,W)$ is obviously singular on this chart, and  this is a singularity analyzed in Theorem \ref{intersection graph}. 
The reader will check that $Z'$ has no further singularities on other charts,
and that the only singularity on the $U$-chart
is located at the origin. 
On this chart,  the exceptional divisor is given by $U=0$. Its reduction has $U=W=0$. The reader will check
that the exceptional divisor $E'$ of this blow-up is a smooth projective line. Note also that the strict
transform of the exceptional divisor from the initial blow-up is given by $V=0$, with
reduction $V=W=0$, and that this strict transform is also
a smooth projective line.

Theorem \ref{intersection graph} lets us describe explicitly the intersection matrix 
$N(s_0 \mid \alpha^{-1}, \beta^{-1}, \gamma^{-1})$ of the unique singularity
in the $U$-chart. Using  the notation from \ref{weighted homogeneous},
we set $q=p$, $a=b=1$, $c=p$ and $d=1$, and find that $g:=\gcd(c,d)=1$ and
$(ad+bc+cd)/g=2p+1$. It follows that
$$
\alpha^{-1}=p^2/(2p-1)\quadand \beta^{-1}=\gamma^{-1}=p/(p-1).
$$
Recall that  $p\geq 3$ and set $e:=(p+1)/2$.  
The reader will check that 
the continued fraction expansion of $\alpha^{-1}=p^2/(2p-1)$ is
$\alpha^{-1}=[e,5,2,\ldots,2]$ with  $2+(p-3)/2$ overall entries, starting with the relations
$$
p^2=e(2p-1) - (p-1)/2,\quadand (2p-1) = 5(p-1)/2 - (p-3)/2.
$$
The self-intersection $-s_0$ of the node of the star-shaped graph is computed as: 
$$
s_0=\frac{1}{p^2} + \frac{2p-1}{p^2} + 2 \frac{p-1}{p} = 2.
$$
Having resolved the singularity (\ref{secondary blowing-up}), we 
get a resolution  for our original singularity $\Spec B_{\mu}$ with the following 
resolution graph:
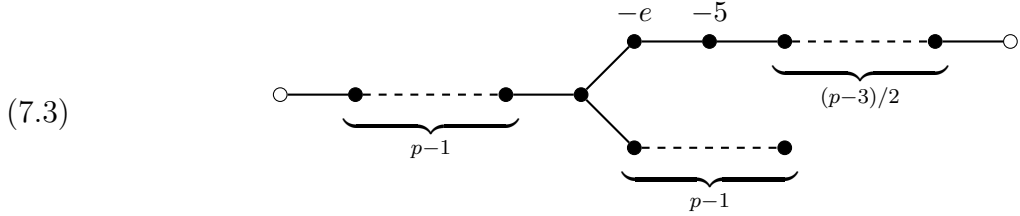
\begin{equation}
\label{nonminimal resolution}
\begin{gathered}
\begin{tikzpicture}
[node distance=1cm, font=\small]
\tikzstyle{vertex}=[circle, draw, fill,  inner sep=0mm, minimum size=1.0ex]
\node[vertex, fill=none]	(E0)  	at (0,0) 	[label=below:{}]               {};
\node[vertex]	(E1)  	         	[right of=E0,label=below:{}]               {};
\node[]	(E2)  	         	[right of=E1,label=below:{$\underbrace{\hspace{14ex}}_{p-1}$}]               {};
\node[vertex]	(E3)  	         	[right of=E2,label=below:{}]               {};

\node[vertex]	(D)			[right of=E3, label=right:{ }]    {};

\node[vertex]	(E4)			[above right of=D, label=above:{$-e$}]    {};
\node[vertex]	(E5)			[right of=E4, label=above:{$-5$}]    {};
\node[vertex]	(S6)			[right of=E5, label=above:{}]    {};
\node[]	(S7)			[right of=S6, label=below:{$\underbrace{\hspace{14ex}}_{(p-3)/2}$}]    {};
\node[vertex]	(S8)			[right of=S7, label=above:{}]    {};
\node[vertex, fill=none]	(S9)			[right of=S8, label=above:{}]    {};

\node[vertex]	(E6)			[below right of=D, label=below:{}]    {};
\node[]	(E7)			[right  of=E6, label=below:{$\underbrace{\hspace{14ex}}_{p-1}$}]    {};
\node[vertex]	(E8)			[right  of=E7, label=below:{}]    {};

\draw [thick] (E0)--(E1);
\draw [thick,dashed]  (E1)--(E3);
\draw [thick,dashed] (E6)--(E8);
\draw [thick,dashed] (S6)--(S8);

\draw [thick] (E3)--(D)--(E4)--(E5)--(S6);
\draw [thick] (D)--(E6);
\draw [thick] (S8)--(S9);
\end{tikzpicture}
\end{gathered}
\end{equation}
According to Proposition \ref{strict transform weighted homogeneous}, 
the white terminal vertex to the left corresponds to the strict transform
of the exceptional divisor on the initial blow-up, whereas
the white terminal vertex on the top right corresponds to the strict transform of the exceptional divisor
on the second blow-up. 

It remains to determine the self-intersection of both of these strict transforms
in the resolution of $\Spec B$.
Recall that  $E'$ is the exceptional divisor for the second blow-up $Z' \to \Spec(B')$.
Computing in the affine charts, one sees that $E'_\red$ is  a projective line, with $E'=pE'_\red$ and $(E'\cdot E'_\red)_{Z'}=-2$.
Since the $U$-chart is regular away from the origin, we can conclude using Proposition \ref{correction degree}
that the self-intersection of the strict transform of $E'_\red$ in the normalization of $Z'$ is $-2/p$. 
Proposition \ref{correction self-intersection} shows that 
the strict transform $C'$ of $E'_\red$ in $X$ has thus $(C'\cdot C')_X =-2/p-\delta$ for some correction term $\delta\in\QQ_{>0}$.
The term $\delta$ is computed as follows. 
Let $\Gamma_1$ be the star-shaped subgraph in \eqref{nonminimal resolution} consisting of all the black vertices,
and let $\Gamma_1'\subset\Gamma_1$ be the star-shaped subgraph obtained from $\Gamma_1$ by removing the terminal black vertex in the top right position.
Let $N_1$ and $N_1'$ be the resulting intersection matrices.
According to Proposition \ref{correction self-intersection}, we have $\delta = -\det(N_1')/\det(N_1)$.
Using Proposition \ref{determinant star-shaped}, we compute that $|\det(N_1)|= p^2$ and $|\det(N_1')|=p^2-2p$.
Hence,  $\delta=-(p^2-2p)/p^2=-1+2/p$, and it follows  that the white terminal vertex on the top right has self-intersection $-1$.
We can thus contract this divisor.
Successively contracting $(-1)$-curves from the right, we get the desired graph as in the statement of Theorem \ref{antidiagonal E8} with a terminal vertex of self-intersection
number $-4=-5+1$ on the top right.  

Recall that we denoted by $E$ the exceptional divisor of $Z \to \Spec B$, and determined 
using Proposition \ref{exceptional divisor} that $E_\red$ is a smooth rational line, that $E=2pE_\red$, and that $(E \cdot E_\red)_Z=-1$.
As above, Proposition \ref{correction self-intersection} shows that 
the strict transform $C$ of $E_\red$ in $X$ has $(C\cdot C)_X =-1/2p-\delta$ for some correction term $\delta\in\QQ_{>0}$.
Let $\Gamma_2$ be the star-shaped subgraph in \eqref{nonminimal resolution} consisting of all the black vertices and 
the terminal white vertex (of self-intersection $(-1)$)
in the top right position.
Let $\Gamma_2'$ be the star-shaped subgraph obtained from $\Gamma_2$ by removing the terminal black vertex of  $\Gamma_2$ 
attached to the terminal white vertex on the left corresponding to $E$. Let $N_2$ and $N_2'$ be the resulting intersection matrices.
According to Proposition \ref{correction self-intersection}, we have $\delta = -\det(N_2')/\det(N_2)$.
The matrix $N_2 $ has the same determinant as $N(2 \mid p/(p-1), p/(p-1), (2p+1)/4)$, and $N_2'$ has the same determinant
as $N(2 \mid (p-1)/(p-2), p/(p-1), (2p+1)/4)$.
Using Proposition \ref{determinant star-shaped}, we compute that $|\det(N_2)|= 2p$ and $|\det(N_2')|=4p-1$.
Hence, $(C \cdot C)_X=-2$.

Now that the intersection matrix $N$ of the resolution has been determined, with $N=N(2\mid p/(p-1),(p+1)/p,(2p+1)/4)$,
Proposition \ref{determinant star-shaped} can be used to show that   $|\det(N)|=1$.
\qed

\medskip

\if false
Recall that the fundamental cycle $Z=\sum m_iE_i$ for the rational double point of type $E_8$ has the 
multiplicities $\begin{smallmatrix} 2 & 4 & 6 & 5 & 4 & 3 & 2\\&&3\end{smallmatrix}$, arranged
according to the shape of the dual graph. For $p=2$, we see that there is a terminal chain of length $p$
where the multiplicities are multiples of $p$, and another terminal chain of length $p-1$ with 
multiples of $p+1$. This pattern generalizes as follows:
\fi

The description of the fundamental cycle ${\mathbf Z}$ in our next proposition follows from the general description in \cite{Tom}
of the fundamental cycle of a star-shaped dual graph.
We leave the proof of the following proposition to the reader.

\begin{proposition}
\mylabel{fundamental cycle E8}
Keep the assumptions of Theorem \ref{antidiagonal E8}.
The multiplicities in the fundamental cycle ${\mathbf Z}$ 
of the resolution of $\Spec B_{\mu}$ are  indicated below next to the corresponding vertex.
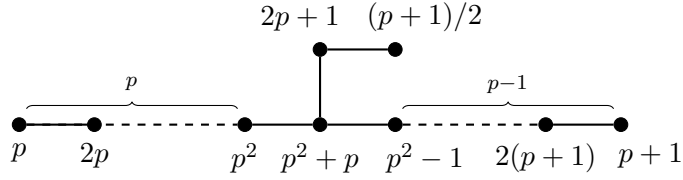
\if false
\begin{equation*}
\label{multiplicities fundamental cycle}
\begin{gathered}
\begin{tikzpicture}
[node distance=1cm, font=\small]
\tikzstyle{vertex}=[circle, draw, fill,  inner sep=0mm, minimum size=1.0ex]
\node[vertex]	(E1)  	at (0,0) 	[label=below:{$p$}]               {};
\node[vertex]	(E2)		        [right of=E1, label=below:{$2p$}]    {};
\node[]	(E3)		        [right of=E2, label=below:{}]    {};
\node[vertex]	(E4)			[right of=E3, label=below:{$p^2$}]    {};

\node[vertex]	(D)			[right of=E4, label=right:{$ \ p^2+p$}]    {};

\node[vertex]	(E5)			[above right of=D, label=above left:{$2p+1$}]    {};
\node[vertex]	(E6)			[right of=E5, label=above right:{$(p+1)/2$}]    {};

\node[vertex]	(E7)			[below right of=D, label=below:{$p^2-1$}]    {};
\node[]	(E8)			[right  of=E7, label=below:{}]    {};
\node[vertex]	(E9)			[right  of=E8, label=below:{$2(p+1)$}]    {};
\node[vertex]	(E10)			[right  of=E9, label=below right:{$p+1$}]    {};

\draw [thick] (E1)--(E2);
\draw [thick, dashed] (E2)--(E4);
\draw [thick] (E4)--(D)--(E5)--(E6);
\draw [thick] (D)--(E7);
\draw [thick, dashed] (E7)--(E9);
\draw [thick] (E9)--(E10);

\end{tikzpicture}
\end{gathered}
\end{equation*}
\fi
\begin{equation*}
\begin{gathered}
\begin{tikzpicture}
[node distance=1cm, font=\small]
\tikzstyle{vertex}=[circle, draw, fill,  inner sep=0mm, minimum size=1.0ex]
\node[vertex]	(E1)  	at (0,0) 	[label=below:{$p$}]               {};
\node[vertex]	        (E11)  	         	[right of=E1,label=below:{$2p$}] {};
\node[]	        (E2)  	         	[right of=E11,label=below:{}]               {};
\node[vertex]	(E3)  	         	[right of=E2,label=below:{$p^2$}]               {};

\node[vertex]	(D)			[right of=E3, label=below:{ $p^2+p$}]    {};

\node[vertex]	(E4)			[above of=D, label=above:{${ 2p+1   \quad \ }$}]    {};
\node[vertex]	(E5)			[right of=E4, label=above:{${\quad \quad (p+1)/2}$}]    {};

\node[vertex]	(E6)			[right of=D, label=below:{$\quad \quad p^2-1$}]    {};
\node[]	        (E7)			[right  of=E6, label=above:{}]    {};
\node[vertex]	(E8)			[right  of=E7, label=below:{$2(p+1)$}]    {};
\node[vertex]	(E12)			[right  of=E8, label=below:{$\quad \quad p+1$}]    {};
\draw [decorate,decoration={brace, raise=6pt}] (E6)--(E12) node [black, midway,xshift=-0cm,yshift=0.5cm] {$^{p-1}$};

\draw [decorate,decoration={brace, raise=6pt}] (E1)--(E3) node [black, midway,xshift=-0cm,yshift=0.5cm] {$^p$};
 
\draw [thick,dashed] (E1)--(E3);
\draw [thick,dashed] (E6)--(E8);

\draw [thick] (E1)--(E11);
\draw [thick] (E8)--(E12);
\draw [thick] (E3)--(D)--(E4)--(E5);
\draw [thick] (D)--(E6);

\end{tikzpicture}
\end{gathered}
\end{equation*}
We have  ${\mathbf Z}^2=-(p+1)/2$, and the fundamental genus is $h^1(\O_{\mathbf Z})=(p^2-p+2)/2$. 
The canonical cycle for the singularity is given by
$$
K= -(2p-4){\mathbf Z} + \frac{p-3}{2}E_j,
$$
where $E_j\in\Gamma_N$ is the terminal vertex on the top right. 
\end{proposition}
\if false
\proof
The self-intersection numbers along the three terminal chains in the dual graph $\Gamma_N$ yield continued fractions
$$
\frac{p+1}{p}=[2,\ldots,2]\quadand \frac{p}{p-1} = [2,\ldots,2]\quadand \frac{2p+1}{4}=[(p+1)/2,4].
$$
Write $E_0\in \Gamma_N$ for the central vertex. According to \cite{Tom}, equation (3.4) on page 282, its multiplicity $m_0\geq 1$
in the fundamental cycle is the smallest integer $m\geq 1$ such that 
\begin{equation}
\label{inequality}
2m - \lceil \frac{mp}{p+1}\rceil - \lceil \frac{m(p-1)}{p}\rceil - \lceil \frac{4m}{2p+1}\rceil \geq 0.
\end{equation}
Here the ceiling  $\lceil r/s\rceil$ denotes the smallest integer $n\geq r/s$.
For $m=(p+1)p$ the first two fractions on the left are integers, whereas the last summand becomes
$$
\lceil \frac{4m}{2p+1}\rceil = \lceil 2p + \frac{2p}{2p+1}\rceil = 2p+1,
$$
and it follows that  \eqref{inequality} holds.
To get the desired multiplicity  $m_0=(p+1)p$, we now assume  $m<(p+1)p$ and have to verify that \eqref{inequality}  fails.
Since the fraction $p/(p+1)$ and $(p-1)/p$ are reduced, and one of the integers  $p$ or $p+1$ does  not divide $m$,
one  of the fractions $mp/(p+1)$ and $m(p-1)/p$ is not an integer. Using $1/p>1/(p+1)$, we obtain
$$
\lceil \frac{mp}{p+1}\rceil + \lceil \frac{m(p-1)}{p}\rceil\geq  \frac{mp}{p+1} + \frac{m(p-1)}{p} + \frac{1}{p+1}.
$$
In turn, the left hand side of \eqref{inequality} is bounded above by 
$$
2m - \frac{mp}{p+1}-\frac{m(p-1)}{p}-\frac{4m}{2p+1} -\frac{1}{p+1}  = \frac{m}{p(p+1)(2p+1)}-\frac{1}{p+1}.
$$
This is  bounded above by $1/(2p+1)-1/(p+1)<0$, because $m<p(p+1)$. As desired, the  \eqref{inequality} fails.

Summing up, the multiplicity   for the central vertex $E_0$ in the fundamental cycle   is $m_0=p(p+1)$.
By \cite{Tom}, Lemma 3.4 we have $(Z\cdot E_i)=0$ for each vertex $E_i\in\Gamma_N$ belonging
to the terminal chains comprising $(-2)$-curves. Since the intersection matrix $N$ is negative-definite,
this property determines the multiplicities. The given multiplicities in \eqref{multiplicities fundamental cycle} 
indeed have this property.

It remains to determine the multiplicity at the terminal chain of length two.
Without loss of generality, we may assume that $E_1,E_2\in \Gamma$ are the vertices, with self-intersection numbers
$E_1^2=-(p+1)/2$ and $E_2^2=-4$. Then
$$
0\geq (Z\cdot E_1) = (m_0E_0+ m_1E_1+m_2E_2)\cdot E_1 =  p(p+1)- m_1(p+1)/2 + m_1.
$$
This gives $m_1\geq 2(p^2+p+1)/(p+1)$, and thus $m_1\geq 2p+1$.
A similar argument shows $m_2\geq (p+1)/2$. Taking  $m_1=2p+1$ and $m_2=(p+1)/2$
one indeed gets $(Z\cdot E_i)=0$ for $i\neq 2$, and $(Z\cdot E_2)=-1$.
In turn, the multiplicites for the fundamental cycle are as in \eqref{multiplicities fundamental cycle}.
Futhermore, we have $Z^2=(Z\cdot m_2E_2)= -(p+1)/2$. 

The canonical cycle $K=\sum n_iE_i$ is uniquely determined by the Adjunction Formula
$(K+E_i)\cdot E_i=-2$. For the $(-2)$-curves, this reduces to  $(K\cdot E_i)=0$.
The latter holds for all cycles of the form  $rZ+sE_2$.
Making an Ansatz, one computes that for the coefficients
$r=-(2p-4)$ and $s=(p-3)/2$ we indeed get the canonical cycle.
The formula
$$
2g-2=(K+Z)\cdot Z = (r+1)Z^2 +s(E_2\cdot Z) = (r+1)\frac{p+1}{2} -s
$$
directly gives the  fundamental genus $g=h^1(\O_Z)$.
\qed

\begin{remark} 
Artin showed in \cite{Artin 1977} that $ E_8$ occurs as  the  dual graph
of  wild $\ZZ/2\ZZ$-quotient singularities for two  different characteristics, $p=2$ and $q=5$.
Since in such case the order of the discriminant group is a power of both $p$ and $q$, this could happen only
with $\|\Phi_N|=1$. Furthermore, the fundamental cycle $Z$ satisfies $|Z^2|\leq p,q$
according to \cite{Lorenzini 2013}, 2.4.
So far, we see no further example of this phenomenon.
Peskin  mentions in \cite{Peskin 1980}, page 111, 
that in characteristic $q=5$ there we are lacking an  explicit description
of the action of $G=\ZZ/q\ZZ$ on the formal power series ring $A=k[[u,v]]$ whose ring of invariants
is the rational double point of type $E_8^1$.
\end{remark}
\fi

\section{Analogues of the \texorpdfstring{$E_7$}{E7} singularities}
\mylabel{analogues E7}

When $p=2$, the blow-up  at the maximal ideal of the ${\mathbb Z}/2{\mathbb Z}$-quotient singularity $E_8^2$ given by 
$$z^2+x y^2z+x^3+y^5=0$$
has a new singularity,
namely the singularity $E^1_7$ given by 
the equation 
$$z^2+xy^2z+yx^3+y^3=0$$
(see for instance \cite{Roc}, 1.1). The singularity $E_8^2$ has resolution graph the Dynkin diagram $E_8$ with trivial discriminant group, while the resolution of $E^1_7$ has resolution graph $E_7$ with discriminant group of order $2$.

Artin (\cite{Artin 1977}, bottom of page 18, or \cite{Peskin 1980}, (2.16), page 104) shows that the Dynkin diagram $E_7$ 
cannot be obtained 
as the resolution graph of a wild ${\mathbb Z}/2{\mathbb Z}$-quotient singularity whose associated action is ramified precisely at the origin. He shows however that 
the singularity $E^1_7$ does occur as the resolution graph of  a wild ${\mathbb Z}/2{\mathbb Z}$-quotient singularity for an action 
that is    ramified in codimension $1$. 

When $p=2$, we have not been able to exhibit any wild ${\mathbb Z}/2{\mathbb Z}$-quotient singularity whose action is {\it ramified precisely at the origin} and whose associated intersection matrix has discriminant group of order $2^s$ with $s$ odd. 
We suggest in \ref{explicit8} for each $s$ odd the existence of explicit examples with group 
$({\mathbb Z}/2{\mathbb Z})^s$. In each case, these wild ${\mathbb Z}/2{\mathbb Z}$-quotient singularities are associated to  actions that are ramified in codimension 1.

The above considerations have analogues for any prime $p$. Indeed,  
consider the singularity at the maximal ideal of $\Spec \SB_n$, where 
$$\SB_n:=k[[x,y,z]]/(z^p - (x y^n)^{p-1}z - y^{pn+1} + x^{p+1}).$$
This singularity is a special case of the singularity recalled in \ref{moderately ramified}, where we have set $a=y^n$ and $b=x$. 
In particular, this singularity is a ${\mathbb Z}/p{\mathbb Z}$-quotient singularity whose moderately ramified action is ramified precisely at the origin.
When $n=p=2$, this singularity is $E_8^2$.

Consider the blow-up of $\Spec \SB_n$ at the maximal ideal $(x,y,z)$. Then the chart defined by the variables $y, x/y, z/y,$
has a singular point whose local ring is isomorphic to the local ring $\SC_n$, where 
\begin{equation} \label{eqC2}
\SC_n:= k[[x,y,z]]/(z^p-(x y^{n})^{p-1}z- y^{(n-1)p+1} + y x^{p+1}).
\end{equation}
When $n>1$, the closed point of $\Spec \SC_n$ is singular,
 and we show below in \ref{pro.wild} that 
the singularity of $\Spec \SC_n$ is again a ${\mathbb Z}/p{\mathbb Z}$-quotient singularity, but for an action that is ramified in codimension $1$. 

When $n=2$, the singularity of $\Spec \SB_2$ is treated in Theorem \ref{antidiagonal E8} and generalizes the $E_8^2$-singularity. 
The singularity of $\Spec \SC_2$ is the $E_7^1$-singularity when $p=2$, and thus $\Spec \SC_2$ is a natural 
generalization for all primes $p$ of  the $E_7^1$-singularity. Our educated guess for the resolution of $\Spec \SC_2$ is discussed in Example \ref{ex.C2}.
In the examples that we were able to compute,
the discriminant groups $\Phi_{\SB_n}$ and $\Phi_{\SC_n}$ 
of the intersection matrices of the resolutions of $\Spec \SB_n$ and $\Spec \SC_n$
when $n>1$ satisfy 
$|\Phi_{\SC_n}|= p |\Phi_{\SB_n}|$.

\if false
\begin{lemma} \label{BlowupNewQuotient} 
Consider the blow-up of $\Spec \SB_n$ at the maximal ideal $(x,y,z)$. Then the chart defined by the variables $y, x/y, z/y,$
has a singular point whose local ring is isomorphic to the local ring $\SC_n$.
\end{lemma}
 \fi

We can further generalize the ring $C_n$ as follows.
Let $a,b \in k[[x,y]]$, not both $0$. Set
$$A_0:=k[[x,y]][U,V]/(U^p-(ay)^{p-1}U-y, V^p-(by)^{p-1}V-xy).$$
 Let $L$ denote the field of fractions of $A_0$. 
The  ring $A_0$ and the field $L$ are endowed with an automorphism $\sigma $ of order  $p$ fixing $k[[x,y]]$ and with 
$$
\begin{array}{rcl}
\sigma(U)&:=&U+  ay, \\
\sigma(V)&:=&V+ by.
\end{array}
$$
As usual, we set $G:=\langle \sigma \rangle$.
Let $z:=aV-bU$. Then $\sigma(z)=z$, and we find that
\begin{equation}
\label{E7modified}
z^p-(aby)^{p-1}z-a^pxy+b^py=0.
\end{equation}
Let $B$ denote the subring $k[[x,y]][z]$ of $A_0$. Let $A$ denote the subring $A_0[\frac{V}{U}]$ of $L$. The group $G$ acts on $A$, since $\sigma(V/U)=(V/U + b y/U)(1+ay/U)^{-1}$ and $1+ay/U$ is a unit in $A_0$.

\begin{proposition} \label{pro.wild}
Keep the above notation. The ring homomorphism $A \to k[[u,v]] $, which sends $U$ to $u$ and $V/U $ to $v$, is a $k$-isomorphism.
In the special case where either $a=x^m$ and $b=y^n$, or $a=y^n$ and $b=x^m$ for some integers $m,n \geq 1$, then the ring of invariants  $A^{G}$ is equal to the ring $B$. In particular, $\Spec C_n$ is a wild ${\mathbb Z}/p{\mathbb Z}$-quotient singularity.
\end{proposition}
\proof 
The equation $U^p-(ay)^{p-1}U-y=0$ first shows that $y/U$ is in the maximal ideal of $A_0$, and then that $y/U^p$ is in $A_0$ and is a unit.
The ring $A_0$ is not integrally closed, 
since it is clear from the equation $V^p-(by)^{p-1}V-xy=0$
that 
$$ (\frac{V}{U})^p-(\frac{by}{U})^{p-1}(\frac{V}{U})-\frac{y}{U^p}x=0$$
is an integral relation for $\frac{V}{U}$ over $A_0$.  Since $x$ and $y$ can be expressed in terms of $U$ and $V/U$, we find that 
$A:=A_0[\frac{V}{U}]$, viewed as a subring of $L$, is in fact isomorphic to the power series ring $k[[u,v]]$, with $u:=U$ and $v:=V/U$.

Consider the ring $B':=k[[x,y]][Z]/(Z^p-(aby)^{p-1}Z-a^pxy+b^py)$ and the natural map $\varphi:B' \to A^G$ which sends $Z$ to $z$.
Assume that either $a=x^n$ and $b=y^m$, or that $a=y^n$ and $b=x^m$ for some integers $m,n \geq 1$.
We claim that $\varphi$ is an isomorphism. One can show that $B'$ is an integral domain, and that its field of fractions injects in $\Frac(A)$, and has image
by degree considerations equal to $\Frac(A^{G})$. The ring $B'$ is Cohen--Macaulay since it is free as a module over the regular ring $k[[x,y]]$. 
Thus $B'$ is normal as soon as it is regular in codimension $1$. This can be shown, because of the special forms of $a$ and $b$, using the Jacobian criterion.
Let $f:=Z^p-(aby)^{p-1}Z-a^pxy+b^py$. Then if a prime ideal $\primid $ of $B'$  contains the classes of $f$, 
and of the partial derivatives $f_x, f_y, f_Z$, then 
$\primid$ 
contains $(x,y,Z)$.

The reader will check that the ring $C_n$ is isomorphic to $B$ when $a=-x$ and $b=-y^{n-1}$. 
When $p=n=2$, the proposition is proved in \cite{Peskin 1980}, (2.16), page 104. 
\qed

\begin{example} \label{ex.second} 
Let $p=2$. Computations with Magma and Singular indicate that $\Spec \SB_3$ admits a resolution with smooth rational curves and 
dual graph drawn on the left below, with trivial discriminant group,

$$
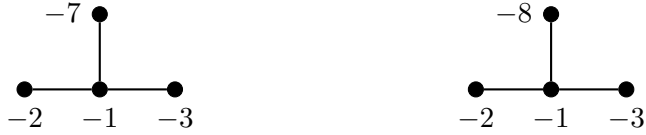
\begin{tikzpicture}
[node distance=1cm, font=\small] 
\tikzstyle{vertex}=[circle, draw, fill,  inner sep=0mm, minimum size=1.1ex]
\node[vertex]	(v1)  	at (0,0) 	[label=below:{$-2$}] 		{};
\node[vertex]	(v2)			[right of=v1, label=below:{$-1$}]	{};
\node[vertex]	(v3)			[above of=v2, label=left:{$-7$}]	{};
\node[vertex]	(v6)			[right of=v2, label=below:{$-3$}]	{};
\draw[thick] (v1)--(v2);
\draw[thick] (v2)--(v3);
\draw[thick] (v2)--(v6);
\node[]	(dummy1)		[right of=v6, label=above:{}]	{};
\node[]	(dummy2)		[right of=dummy1, label=above:{}]	{};
\node[]	(dummy3)		[right of=dummy2, label=above:{}]	{};
\node[vertex]	(w1)  	  	[right of=dummy3, label=below:{$-2$}] 		{};
\node[vertex]	(w2)			[right of=w1, label=below:{$-1$}]	{};
\node[vertex]	(w3)			[above of=w2, label=left:{$-8$}]	{};
\node[vertex]	(w6)			[right of=w2, label=below:{$-3$}]	{};
\draw[thick] (w1)--(w2);
\draw[thick] (w2)--(w3);
\draw[thick] (w2)--(w6);
\end{tikzpicture}
$$
 while 
$\Spec \SC_3$ admits a resolution with smooth rational curves and 
dual graph drawn on the right above, with discriminant group of order $2$.
\end{example}

\begin{example} \label{ex.star}   
We show in this example that there are (many) intersection matrices $N$ with $\Phi_N$ killed by $2$ and of order $2^s$ with $s$ odd.
Since our interest is to provide evidence that there may exist wild ${\mathbb Z}/2{\mathbb Z}$-quotient singularities whose resolutions 
have discriminant groups of order $2^s$ with $s$ odd, we note that any such resolution must also have an intersection matrix $N$ whose fundamental cycle $Z$ satisfies
$|Z^2|\leq 2$ (\cite{Lorenzini 2013}, 2.4). This is a non-trivial restriction on the possible matrices $N$, and we exhibit below matrices that also satisfy this restriction.

Recall that a star-shaped graph with $n\geq 4$ vertices is called a {\it star}, or the  {\it complete bipartite graph} $K_{1,n-1}$, if 
it consists in a single node and $n-1$ terminal vertices attached to the node.
We write the intersection matrix $N$ of a star on $n$ vertices as $N=N(s_0\mid s_1/1,\dots, s_{n-1}/1) $, 
where $-s_0$ denotes the self-intersection of the node, and $-s_i$ denotes the self-intersection of the $i$-th terminal vertex when $i>0$. 
The Dynkin diagram $D_4$ is a star on $4$ vertices, and so are the two graphs in  Example \ref{ex.second}.


Consider any intersection matrix $N=N(s_0\mid s_1/1,\dots, s_{n-1}/1) $ such that one of the $s_j$ with $j\geq 1$ is even and at most one of the $s_j$ with $j\geq 1$ is divisible by $4$.
Assume in addition that $\Phi_N$ is killed by $2$, and that
the fundamental cycle $Z$ of $N$ satisfies $|Z^2|\leq 2$.
Define the matrix $N_i(s_0 \mid s_1/1,\dots, s_{n-1}/1,s_n/1)$, $i=1,2$, by
 $$s_n := i+ (\prod_{j=1}^{n-1} s_j)/|\Phi_N|.$$
{\it We claim that the two intersection matrices $N_1$ and $N_2$ have graphs that are stars on $n+1$ vertices with $|\text{\rm det}(N_i)|=i |\text{\rm det}(N)|$.
Moreover, both groups $\Phi_{N_i}$ are killed by $2$, and both fundamental cycles $Z_i$ of $N_i$ satisfy  $|Z_i^2|\leq 2$.}

\proof Let $\ell_{n-1}:=\lcm(s_1, \dots, s_{n-1})$. Then the order of the node in $\Phi_N$ is equal to $\ell_{n-1}(s_0 - \sum_{j=1}^{n-1} 1/s_j)$ (use \ref{determinant star-shaped} (ii)). This order equals $1$ since we assume that one of the $s_j$ is even (use \ref{determinant star-shaped} (v)).
It follows that
$|\Phi_N| = (\prod_{j=1}^{n-1}s_j)/\ell_{n-1}$ (use \ref{determinant star-shaped} (i)). In particular, $(\prod_{j=1}^{n-1} s_j)/|\Phi_N|= \ell_{n-1}$ is an integer.  
The equality $|\text{\rm det}(N_i)|=i |\text{\rm det}(N)|$
follows from an easy computation.

We find that $\lcm(s_1,\dots, s_{n-1},\ell_{n-1}+i) = \lcm(\ell_{n-1}, \ell_{n-1}+i)$, which equals $ \ell_{n-1} (\ell_{n-1}+1)$ when $i=1$, 
and $ \ell_{n-1} (\ell_{n-1}/2+1)$ when $i=2$.
Hence, the node is trivial in $\Phi_{N_i}$ since its order is
$$\lcm(s_1,\dots, s_{n-1},\ell_{n-1}+i)(s_0 - \sum_{j=1}^{n-1} 1/s_j -1/(\ell_{n-1}+i))=1.$$
Let $R \in {\mathbb Z}^{n+1}$ denote the transpose of vector $(\ell_{n-1}, \ell_{n-1}/s_1,\dots, \ell_{n-1}/s_{n-1}, 1)$. 
Then $N_iR = -ie_{n+1}$. Since all coefficients of $R$ are positive and $N_iR$ has non-positive coefficients, we find that $R$ is an upperbound for the fundamental cycle $Z_i$ of $N_i$. Then $|Z_i^2| \leq |R^2| \leq i$, as desired.

To show that $\Phi_{N_i}$ is killed by $2$,  it suffices to show that the classes of the standard vectors have order $1$ or $2$ in $\Phi_{N_i}$ 
for each terminal vertex of the graph. This is clear for a terminal vertex $v_j$ with $s_j$ odd or exactly divisible by $2$, since the column of $N_i$ corresponding to $v_j$ shows that the class of $s_jv_j$ is equal to the class of the node. 
We note now that the construction implies that there can be at most one terminal vertex $v_j$ with $s_j$ divisible by $4$. If the corresponding class in $\Phi_{N_i}$ has order divisible by $4$, we would find using the first column of the matrix $N_i$ that this unique class is equal to the sum of classes which all have order $1$ or $2$, a contradiction. This ends the proof of the claim. \qed

\medskip
The sequence $ \{ s_n \}_{n \geq 1}$ with $s_1=2$ and $s_n:= \lcm(s_1,\dots, s_{n-1}) +1$ is called {\it Sylvester's sequence} $\{2,3,7,43,\dots\}$ in the literature,
and produces the only intersection matrices $N(1 \mid s_1/1,\dots, s_{n-1}/1)$ with trivial group $\Phi_N$ in the above construction.  
An example of a star with intersection matrix $N$ such that $\Phi_N$ is killed by $2$
but $|Z^2| >2$ is given by $N=N(1\mid 2/1,3/1,10/1,16/1)$, with  group $\Phi_N= ({\mathbb Z}/2{\mathbb Z})^2$
and $|Z^2| =4$. 
\end{example}
 
\begin{example} \label{explicit8} 
Let $p=2$. Fix an integer $n \geq 1$. Consider the star graph with a central node of self-intersection $-(n+1)$ attached to $2n+1$ terminal vertices of self-intersection $-2$. Denote by $N_0$
its intersection matrix. Proposition \ref{determinant star-shaped} (iv) shows that $\Phi_{N_0} = ({\mathbb Z}/2{\mathbb Z})^{2n}$. We remark in passing that this matrix does occur as the intersection matrix attached to a quotient singularity (use the equation $z^2-xy(x^{2n-1}-y^{2n-1})$ and Theorem \ref{Brieskorn quotient sing} (ii)).

Starting with $N_0$, the construction in Example \ref{ex.star}  
produces  two intersection matrices, the matrix $N_1(n):=N(n+1\mid 2/1,\ldots,2/1,3/1)$ with group of order $2^{2n}$ and whose graph is represented on the left below, and
the matrix $N_2(n):=N(n+1\mid 2/1,\ldots,2/1,4/1)$ with group of order $2^{2n+1}$ and whose graph is represented below on the right. 
$$
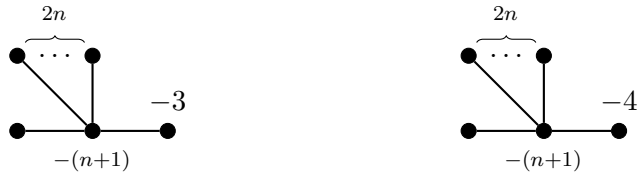
\begin{tikzpicture}
[node distance=1cm, font=\small] 
\tikzstyle{vertex}=[circle, draw, fill,  inner sep=0mm, minimum size=1.1ex]
\node[vertex]	(v1)  	at (0,0) 	[label=below:{}] 		{};
\node[vertex]	(v2)			[right of=v1, label=below:{$^{-(n+1)}$}]	{};
\node[vertex]	(v3)			[above of=v1, label=below:{}]	{};
\node[vertex]	(v4)			[above of=v2, label=right:{}]	{};
\node[vertex]	(v6)			[right of=v2, label=above:{$-3$}]	{};
\draw[thick] (v1)--(v2);
\draw[thick] (v2)--(v3);
\draw[thick] (v2)--(v4);
\draw[thick] (v2)--(v6);
\draw [decorate,decoration={brace, raise=5pt}] (v3)--(v4) node [black, midway,xshift=-0cm,yshift=0.5cm] {$^{2n}$};
\draw [decorate,decoration={}] (v3) (v4) node [black, midway,xshift=0.5cm,yshift=1.0cm] {$\hdots$};
\node[]	(dummy1)		[right of=v6, label=above:{}]	{};
\node[]	(dummy2)		[right of=dummy1, label=above:{}]	{};
\node[]	(dummy3)		[right of=dummy2, label=above:{}]	{};
\node[vertex]	(w1)  	  	[right of=dummy3, label=below:{}] 		{};
\node[vertex]	(w2)			[right of=w1, label=below:{$^{-(n+1)}$}]	{};
\node[vertex]	(w3)			[above of=w1, label=below:{}]	{};
\node[vertex]	(w4)			[above of=w2, label=right:{}]	{};
\node[vertex]	(w6)			[right of=w2, label=above:{${-4}$}]	{};
\draw[thick] (w1)--(w2);
\draw[thick] (w2)--(w3);
\draw[thick] (w2)--(w4);
\draw[thick] (w2)--(w6);
\draw [decorate,decoration={brace, raise=5pt}] (w3)--(w4) node [black, midway,xshift=-0cm,yshift=0.5cm] {$^{2n}$};
\draw [decorate,decoration={}] (w3) (w4) node [black, midway,xshift=6.5cm,yshift=1.0cm] {$\hdots$};
\end{tikzpicture}
$$
When $n=1$, the intersection matrices $N_1$ and $N_2$
are the matrices of the resolutions of the wild quotient singularities $\Spec \SB_4$ and $\Spec \SC_4$, respectively.
\if false
Recall that the defining equation for $\SB_4$ is 
$z^2 - x y^4 z - y^{9} + x^{3}$. 
Using Theorem \ref{thm.BrieskornResolution}, we find that the associated  Brieskorn singularity
given by $z^2 - y^{9} + x^{3}=0$ has the above graph on the left as resolution graph. The defining equation for $C_4$ is $z^2 - x y^4 z - y^{7} + yx^{3}$.
Using Theorem \ref{intersection graph}, we find that the associated  singularity
given by $z^2 - y^{7} + x^{3}y=0$ has the  graph on the right as resolution graph.
\fi

When  $n\geq 1$, consider the equation $
f:=z^p-(aby)^{p-1}z-a^pxy+b^py$ introduced in \eqref{E7modified}, and set 
$a:=x^n$ and $b:=y^{2n+1}$. Let $B:=k[[x,y,z]]/(f)$.
Proposition \ref{pro.wild} shows that this equation defines a wild ${\mathbb Z}/2{\mathbb Z}$-quotient singularity. We conjecture that  $\Spec B$ has a
resolution $X \to \Spec B$ with a dual graph equal to the dual graph of $N_2(n)$ represented on the right above. The conjecture thus provides examples of wild ${\mathbb Z}/2{\mathbb Z}$-quotient singularities  with discriminant group of order $2^{2n+1}$ for all $n \geq 1$. These quotient singularities are associated with actions that are ramified in codimension 1.
\end{example}

 \begin{example} \label{ex.C2} (Analogues of $E_7$.) Let $p$ be prime.
Computations suggest that the resolution of the wild $\ZZ/p\ZZ$-quotient singularity $\Spec \SC_2$ (see \eqref{eqC2}) has intersection matrix (notation as in \ref{notation.starshaped})
$$N=N(2 \mid \frac{p}{p-1}, \frac{p+1}{p}, \frac{p^2}{2p-1})$$
with group $\Phi_N={\mathbb Z}/p{\mathbb Z}$. When $p$ is odd, 
the intersection matrix $N$ has the following graph: 

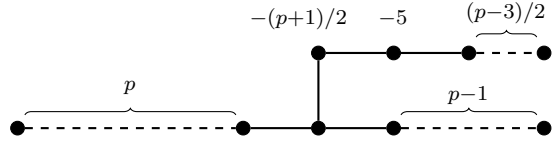
\if false
\begin{center}
{\centering
\begin{graph}(10,4) 
\autodistance{2}
\opaquetextfalse

\roundnode{DV}(1,2)
\edge{DV}{P1}
\roundnode{P1}(2,2)
\roundnode{P2}(3,2)
\edge{P1}{P2}[\graphlinedash{2}]
\bowtext{DV}{P2}{-.3}{$\underbrace{\phantom{aaaaaaaaaaa}}_{p}$}

\roundnode{CN}(4,2)
\edge{CN}{P2}
\edge{CN}{Ra1}
\edge{CN}{Rb1}

\roundnode{Ra1}(5,3)
\autonodetext{Ra1}[n]{$-e$}
\roundnode{Ra2}(6,3)
\autonodetext{Ra2}[n]{$-5$}
\roundnode{Ra3}(7,3)
\roundnode{Ra4}(8,3)

\edge{Ra1}{Ra2}
\edge{Ra2}{Ra3}
\edge{Ra3}{Ra4}[\graphlinedash{2}]
\bowtext{Ra3}{Ra4}{-.5}{$\underbrace{\phantom{aaaaaa}}_{(p-3)/2}$}

\roundnode{Rb1}(5,1)
\roundnode{Rb2}(6,1)
\edge{Rb1}{Rb2}[\graphlinedash{2}]
\bowtext{Rb1}{Rb2}{-.5}{$\underbrace{\phantom{aaaaaa}}_{p-1}$}

\end{graph}
}
\end{center}
\fi

\begin{equation*}
\begin{gathered}
\begin{tikzpicture}
[node distance=1cm, font=\small]
\tikzstyle{vertex}=[circle, draw, fill,  inner sep=0mm, minimum size=1.0ex]
\node[vertex]	(E1)  	at (0,0) 	[label=below:{}]               {};
\node[]	        (E11)  	         	[right of=E1,label=above:{}] {};
\node[]	        (E2)  	         	[right of=E11,label=above:{}]               {};
\node[vertex]	(E3)  	         	[right of=E2,label=below:{}]               {};

\node[vertex]	(D)			[right of=E3, label=right:{ }]    {};

\node[vertex]	(E4)			[above of=D, label=above:{$^{ -(p+1)/2  \quad \ }$}]    {};
\node[vertex]	(E5)			[right of=E4, label=above:{$^{-5}$}]    {};
\node[vertex]	(E9)			[right of=E5, label=below:{}]    {};
\node[vertex]	(E10)			[right of=E9, label=below:{}]    {};
\draw [decorate,decoration={brace, raise=6pt}] (E9)--(E10) node [black, midway,xshift=-0cm,yshift=0.5cm] {$^{(p-3)/2}$};
\draw [thick,dashed] (E9)--(E10);
\draw [thick] (E5)--(E9);

\node[vertex]	(E6)			[right of=D, label=below:{}]    {};
\node[]	        (E7)			[right  of=E6, label=above:{}]    {};
\node[vertex]	(E8)			[right  of=E7, label=below:{}]    {};
\draw [decorate,decoration={brace, raise=5pt}] (E6)--(E8) node [black, midway,xshift=-0cm,yshift=0.4cm] {$^{p-1}$};

\draw [decorate,decoration={brace, raise=5pt}] (E1)--(E3) node [black, midway,xshift=-0cm,yshift=0.5cm] {$^p$};
 
\draw [thick,dashed] (E1)--(E3);
\draw [thick,dashed] (E6)--(E8);

\draw [thick] (E3)--(D)--(E4)--(E5);
\draw [thick] (D)--(E6);

\end{tikzpicture}
\end{gathered}
\end{equation*}
The resolution of $\Spec \SB_2$ is discussed in Theorem \ref{antidiagonal E8}.
\end{example}

\begin{remark} \label{E7 homogeneous singularities} 
Consider the equation $
z^p-(aby)^{p-1}z-a^pxy+b^py=0$ introduced in \eqref{E7modified}, and set 
$a=y^n$ and $b=x^m$ 
for some integers $m,n \geq 1$. Proposition \ref{pro.wild} shows that this equation defines a wild ${\mathbb Z}/p{\mathbb Z}$-quotient singularity.
Computations with Magma \cite{Magma} suggest that for such $a$ and $b$, the resolution of the singularity at the origin of 
$z^p-(aby)^{p-1}z-a^pxy+b^py=0$ has the same intersection matrix as the resolution of the singularity of $z^p -a^pxy+b^py=0$.

When $a=y^n$ and $b=x^m$, this latter singularity has the form  
$z^p-xy(y^{pn}-x^{pm-1})=0$, and Theorem \ref{intersection graph} provides an explicit resolution for it. When $p=2$, we find that $g:=\gcd(pn,pm-1)$ 
is always odd, so the discriminant group of this resolution, which has order $2^{g+1}$ by \ref{special intersection graph}, is always of the form
 $|\Phi_N|=2^{s}$ with $s $ even. Thus the quotient singularity \eqref{E7modified} in this case is unlikely to provide examples of discriminant groups
 of order $|\Phi_N|=2^{s}$ with $s $ odd.

When $p=2$,  \eqref{E7modified} in the case $b=x$ and $a=y^n$ gives the equation of the singularity $D_{2(2n+1)}^{n}$ with resolution graph the Dynkin diagram 
$D_{2(2n+1)}$ (notation as in \cite{Artin 1977}, section 3).
\end{remark}

\section{\texorpdfstring{$D_{4}$}{D4} and \texorpdfstring{$A_{p-1}$}{Ap-1}}
\mylabel{The Ap-1 singularity}
 
 We compute in this section the resolution of the singularity of $\Spec B_{\mu}$
introduced in \ref{moderately ramified}, for any value of the parameter $\mu$ when $a=y$ and $b=x$. The ring $B_{\mu}$ is given in this case by 
$$B_\mu:=k[[x,y]][z]/(z^p-(\mu xy)^{p-1}z - x^{p+1}+y^{p+1}).
$$
Let $\YYY \ra \Spec(B_{\mu})$ be the blow-up of the ideal $\idealb=(x,y,z)$, as in \ref{exceptional divisor}.
We note in Theorem \ref{Ap-1} that $Z$ has $p+1$ singularities, each again $\ZZ/p\ZZ$-quotient singularities, with resolution graph $A_{p-1}$ and associated discriminant group $\ZZ/p\ZZ$.

 \begin{remark}
 When $k$ contains a third root of unity $\zeta$ with $\zeta^2+\zeta+1=0$, 
the change of variables $X:=x+\zeta y$ and $Y:=x+\zeta^2 y$ produces $x^3+y^3=-\zeta XY(X+\zeta Y)$. In particular, the singularity 
$z^q-(x^3+y^3)=0$ is always isomorphic over $k$ to the singularity $z^q-(x^2y-xy^2)=0$.
When in addition $p=2$, we find that 
$B_{\mu=0}$  is isomorphic over ${\mathbb F}_4$ to the singularity $D^0_4$, given by the equation $z^2  + x^{2}y+xy^{2}=0 $.
The dual graph of its resolution is the Dynkin diagram $D_4$. The Tjurina number of this singularity is equal to $8$.

The resolution of $\Spec B_{\mu=1}$ when $p=2$ is also known to have  dual graph $D_4$ over an algebraically closed field. Indeed, the equation when $\mu=1$ is stated to be equivalent to $D_4^1$ in \cite{Peskin 1980}, page 102, where $D_4^1$ is given by the equation $z^2 +xyz + x^{2}y+xy^{2}=0 $. 
This can be seen indirectly as follows. By Theorem \ref{a=y, b=x}, the resolution  of $\Spec B_{\mu=1}$ is of type $D_4$. According to  Artin's classification
\cite{Artin 1977}, there are only two possible isomorphism types of singularities with resolution $D_4$ when $p=2$, namely $D_4^0$ and $D_4^1$.
Since the Tjurina number of $B_{\mu=1}$
is equal to $6$, this ring must then be isomorphic  to the $D_4^1$ singularity over the algebraic closure of $k$.
The quotient singularity $\Spec B_{\mu=1}$ when $p>2$ can thus be considered as a generalization of $D_4^1$.
\end{remark}

\begin{theorem} \label{a=y, b=x}  $\Spec \R_{\mu}$ has a resolution of singularities 
with star-shaped dual graph $\Gamma_N$ having $p+1$ identical terminal chains with $p-1$ vertices as follows:
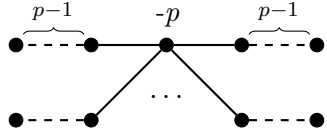
\begin{equation*}
\begin{gathered}
\begin{tikzpicture}
[node distance=1cm, font=\small]
\tikzstyle{vertex}=[circle, draw, fill,  inner sep=0mm, minimum size=1.0ex]
\node[vertex]	(C)  	at (0,0) 	[label=above:{-p}]               {};
\node[]	(dummyS)  	         	[below  of=C,label=above:{$\hdots$}]               {};

\node[vertex]	(E1)  	         	[left of=C,label=above:{ }]               {};
\node[vertex]	(D1)  	         	[left of=E1,label=above:{ }]               {};

\node[vertex]	(E2)  	         	[below of=E1,label=below:{}]               {};
\node[vertex]	(D2)  	         	[left of=E2,label=below:{}]               {};

\node[vertex]	(E4)  	         	[right of=C,label=above:{ }]               {};
\node[vertex]	(D4)  	         	[right of=E4,label=above:{ }]               {};

\node[vertex]	(E3)  	         	[below of=E4,label=below:{}]               {};
\node[vertex]	(D3)  	         	[right of=E3,label=below:{}]               {};

\draw [thick] (C)--(E1);
\draw [thick, dashed] (E1)--(D1);
\draw [thick] (C)--(E2);
\draw [thick, dashed] (E2)--(D2);
\draw [thick] (C)--(E3);
\draw [thick, dashed] (E3)--(D3);
\draw [thick] (C)--(E4);
\draw [thick, dashed] (E4)--(D4);
\draw [decorate,decoration={brace, mirror, raise=4pt}] (E1)--(D1) node [black, midway,xshift=0cm,yshift=0.4cm] {$^{p-1}$};
\draw [decorate,decoration={brace, mirror, raise=4pt}] (D4)--(E4) node [black, midway,xshift=0cm,yshift=0.4cm] {$^{p-1}$};
\end{tikzpicture}
\end{gathered}
\end{equation*}
\if false
\begin{equation*}
\begin{gathered}
\begin{tikzpicture}
[node distance=1cm, font=\small]
\tikzstyle{vertex}=[circle, draw, fill,  inner sep=0mm, minimum size=1.0ex]
\node[vertex]	(C)  	at (0,0) 	[label=above:{-p}]               {};
\node[]	(dummyW)  	         	[left   of=C,label=below:{}]               {};
\node[]	(dummyWW)  	         	[left  of=dummyW,label=below:{}]               {};
\node[]	(dummyE)  	         	[right  of=C,label=below:{}]               {};
\node[]	(dummyEE)  	         	[right  of=dummyE,label=below:{}]               {};
\node[]	(dummyS)  	         	[below  of=C,label=below:{$\hdots$}]               {};

\node[vertex]	(E1)  	         	[below of=dummyWW,label=below:{}]               {};
\node[vertex]	(D1)  	         	[below  of=E1,label=below:{}]               {};

\node[vertex]	(E2)  	         	[below of=dummyW,label=below:{}]               {};
\node[vertex]	(D2)  	         	[below of=E2,label=below:{}]               {};

\node[vertex]	(E3)  	         	[below of=dummyE,label=below:{}]               {};
\node[vertex]	(D3)  	         	[below of=E3,label=below:{}]               {};

\node[vertex]	(E4)  	         	[below of=dummyEE,label=below:{}]               {};
\node[vertex]	(D4)  	         	[below of=E4,label=below:{}]               {};

\draw [thick] (C)--(E1);
\draw [thick, dashed] (E1)--(D1);
\draw [thick] (C)--(E2);
\draw [thick, dashed] (E2)--(D2);
\draw [thick] (C)--(E3);
\draw [thick, dashed] (E3)--(D3);
\draw [thick] (C)--(E4);
\draw [thick, dashed] (E4)--(D4);
\draw [decorate,decoration={brace, mirror, raise=6pt}] (E1)--(D1) node [black, midway,xshift=-0.8cm,yshift=0.0cm] {$^{p-1}$};
\end{tikzpicture}
\end{gathered}
\end{equation*}
\fi
The associated discriminant group $\Phi_N$ has order $p^p$.
\end{theorem}

\proof
Let $\YYY \ra \Spec(B_{\mu})$ be the blow-up of the ideal $\ideala B=(a,b,z)=(x,y,z)$, as in Proposition \ref{exceptional divisor}.
Let as usual $E$ denote the exceptional divisor.
We find from Proposition \ref{exceptional divisor} that $E_{\red}$ is a smooth rational curve over $k$, and that $(E \cdot E_{\red})_Z=-1$.
In addition, $E=pE_{\red}$, and the $z$-chart is regular.

The blow-up $Z$ is covered by three affine charts, and we see that the $x$-chart is generated
by the expressions $x,y/x,z/x$ modulo the relation
$$
\left(\frac{z}{x}\right)^p - x\left(1+\mu^{p-1} x^{p-2}\left(\frac{y}{x}\right)^{p-1}\frac{z}{x}  - \left(\frac{y}{x}\right)^{p+1}\right) =0.
$$
Clearly, this chart is regular at the origin. Let $Y \to \YYY$ denote the normalization of $\YYY$.
Let $D$ denote as usual the pull-back of the exceptional divisor of $Z$. It follows from the regularity at the origin that
 the induced morphism $D_{\red} \to E_{\red}$ is an isomorphism.
Hence, we can conclude from Proposition \ref{correction degree} that
$(D_{\red} \cdot D_{\red})_Y= -1/p$. 

Using partial derivatives, one sees that
the singular locus on the $x$-chart is given by $x=z/x=0$ and $(y/x)^{p+1}=1$. Let $\zeta$ denote a primitive root of the equation $u^{p+1}=1$.
When rewriting the above equation defining the $x$-chart in terms of the expressions $x$, $y/x-\zeta^j$, and $z/x$,
we obtain a polynomial of the form $x(y/x-\zeta^j)+ O(3)$. Using the changes of variables discussed in the proof of \ref{detection rdp},
we find that the singularity is in fact a rational double point
of type $A_{p-1}$.

Let $X \to Y$ denote a resolution of the singularities of $Y$. 
Let $C$ denote the strict transform of $D_{\red}$ in $X$.
It follows from Proposition \ref{correction self-intersection} 
that $(C \cdot C)_{X}= -1/p- (p+1)\delta$, where $\delta $ is the correcting term 
associated with the rational double point $A_{p-1}$. As noted in \ref{chain}, $\delta=(p-1)/p$, and we find that $(C \cdot C)_{X}=-p$.
The associated discriminant group is computed with Proposition \ref{determinant star-shaped}.
\qed

\if false
Let $E\subset X$ be the exceptional divisor of the blow-up with $\maxid_A$. Since $X$ is regular, $E$ is a smooth projective line with
self-intersection number $(E\cdot E)_X=-1$.
The exceptional divisor of the blow-up of $\ideala A$  is the divisor 
$pE$, with self-intersection number $(pE)^2=-p^2$.
Now consider the exceptional divisor $D\subset Y$. According to Proposition \ref{exceptional divisor}, we have $D=pD_\red$,
with rational self-intersection number $(D_\red \cdot D_{\red})_Y=-1/p$. 
Using  \cite{Kleiman 1966}, Proposition 6 on page 299,  
we have $(pE)^2=p\cdot D^2$, which is another way to arrive at $D^2_\red=-1/p$.
Each of the  $p+1$ rational double points of type $A_{p-1}$ contributes a correction term
$\delta=(p-1)/p$. The self-intersection number of the strict transform of $D_\red$ on
the resolution of singularities $\tilde{Y} \ra Y $ is then $-s_0=-1/p -(p+1)(p-1)/p = -p$ (\ref{correction self-intersection}). We have thus completely determined
the intersection matrix and dual graph of the resolution $\tilde{Y} \ra Y $.
\fi

\medskip
Let $R:=k[[x,y]]$. As recalled in \ref{moderately ramified}, 
let 
$$A:=k[[u,v]] =R[u,v]/(u^p-(\mu y)^{p-1}u - x, v^p-(\mu x)^{p-1}v - y),$$
and let $\sigma$ be the automorphism defined by $\sigma(u)=u+\mu y$ and $\sigma(v)=v+\mu x$. 
Let $G:=\langle \sigma \rangle$. The element $z :=xu-yv$ is invariant, and we can identify the ring $B_{\mu}$ with $A^G$.

Let $Z'\ra\Spec(A)$ be the blow-up of the induced ideal $\ideala A$, and let $Y' \to Z'$ denote the normalization of $Z'$. We have  
the commutative diagram
$$
\begin{CD}
@. Y'	@>>>  Z'	@>>> 	\Spec(A)\\
@. @VVV	@VVV		@VVV\\
X 	@>>> Y	@>>> \YYY	@>>> 	\Spec(A^G).
\end{CD}
$$
Let $y_i$, $i=1,\dots, p+1,$ denote the  rational double points
in $Y$ of type $A_{p-1}$. We show below that these points are in fact ${\mathbb Z}/p{\mathbb Z}$-quotient singularities.
\begin{lemma} The scheme $Y'$ is regular, and the morphism
$Y'\ra\Spec(A)$ coincides with the blow-up of the maximal ideal $\maxid_A=(u,v)$.
\end{lemma}
\proof
Indeed, using the relations 
\begin{equation}
\label{implicit formula}
u^p-(\mu y)^{p-1}u=x\quadand v^p-(\mu x)^{p-1}v=y,
\end{equation}
we get $u^p,v^p\in\ideala A$. Since the finite ring extension $R\subset A$ is flat of degree $p^2$,
we must have $\ideala A=(u^p,v^p)$. More precisely, substituting the equations \eqref{implicit formula} into each others
one obtains
$$
x\cdot\text{unit} = u^p - \mu^{p-1}v^{p(p-1)}u\quadand y\cdot\text{unit} = v^p-\mu^{p-1}u^{p(p-1)}v,
$$
showing explicitly that $(x,y)A \subseteq (u^p,v^p)$. Since $z=xu-yv$, we have $(u^p,v^p)=\ideala A$. 

The blow-up $Z'$ of the ideal $(u^p,v^p)$ in $\Spec(A)$ is  covered by two charts. The $u^p$-chart has generators
$u,v$, and $v^p/u^p$, so $v/u$ satisfies an obvious integral equation, and  we also have $v=v/u\cdot u$.
It follows that on the normalization the chart becomes regular.
The situation on the $v^p$-chart is similar, and we see that the scheme $Y'$ is regular.
\qed

\begin{theorem}  \label{Ap-1}  The preimage of each $y_i$ under the map $Y' \to Y$ consists of a single regular point $x_i \in Y'$,  
and  $\O_{Y,y_i}=(\O_{Y',x_i})^G$. Thus $y_i$ is a $\ZZ/p\ZZ$-quotient singularity whose resolution has dual graph $A_{p-1}$
 and associated discriminant group $\ZZ/p\ZZ$.
The morphism $\Spec \O_{Y',x_i} \to \Spec (\O_{Y',x_i})^G$ is ramified in codimension $1$
and the punctured spectrum of the rational double point $y_i$
has trivial fundamental group.
\end{theorem}
\proof
The $G$-action on the ring $A$ induces a $G$-action on the normalized blow-up $Y'$, which on the field of fractions
of the $u$-chart is given by
$$
u\mapsto u+ \mu y\quadand v/u\mapsto (v+\mu x)/(u+\mu y).
$$
Since $Y$ is normal, the induced morphism $Y'\ra Y$ yields an identification $Y=Y'/G$.

Let $E'$ denote the exceptional divisor of the blow-up $Y' \to \Spec(A)$ of the maximal ideal. 
Then the natural map $E' \to D_{\red}$ induced by $Y' \to Y$ is purely inseparable of degree $p$, 
and, hence, the morphism $\Spec \O_{Y',x_i} \to \Spec (\O_{Y',x_i})^G$ is ramified at the codimension $1$ point corresponding to $E'$.
It follows from  \cite[Corollary 1.2]{Artin 1977} that
the punctured spectrum of the rational double point  
 $y_i$ 
 has trivial fundamental group. \qed

\if false
so the $G$-action on the regular local rings  $\O_{X,x_i}$ must be ramified in codimension one.
Indeed, since the preimage of $D\subset Y$ is $pE\subset X$ the $G$-action fixes $E$.
Furthermore, the action on the function field of $E$ is trivial, because $E\ra D$ has degree one.
\fi
\begin{remark}
The occurrence of the $A_{p-1}$-singularities $y_i$ on the quotient $Y=Y'/G$ is
caused by points $x_i \in Y'$ where the ideal of the fixed scheme $Y'^G\subset Y'$ is not a Cartier divisor.
Indeed, using Theorem 2 in \cite{K-L}, we find that when the action of $\sigma $ on the local ring $A=k[[u,v]]$ is such that the ideal $(\sigma(u)-u, \sigma(v)-v)$
of the fixed scheme  is principal, then the fixed ring $A^{\langle \sigma \rangle}$ is regular.
\end{remark}

\if false
\begin{remark} When $p=2$, the $A_{p-1}$-singularity has a smooth exceptional divisor consisting in a single irreducible component of self-intersection $-2$. 
When $p>2$, we do now know if there exists a quotient singularity $\Spec A^G$ whose resolution has an exceptional divisor consisting in a single smooth irreducible component of self-intersection $-p$.
\end{remark}
\fi

\if false 
For $p=2$, this becomes the rational double point of type $D_4^1$.
Using Artin's Algorithm, one quickly computes the fundamental cycle $Z$, and the 
Adjunction Formula yields the canonical cycle $K$ and 
the fundamental genus $h^1(\O_Z)$. This all suggests to regard the wild quotient singularity $A^G$
as an analogue of the rational double point of type $D_4^1$:
\fi

We leave the proof of the following proposition to the reader.
\begin{proposition} \label{pro.NumericallyGorensteinEquality}
The multiplicities in the fundamental cycle ${\mathbf Z}$ 
of the resolution of $\Spec \R_{\mu}$ are strictly decreasing along each terminal chain, as indicated below next to the corresponding vertex. 
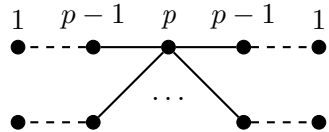
\begin{equation*}
\begin{gathered}
\begin{tikzpicture}
[node distance=1cm, font=\small]
\tikzstyle{vertex}=[circle, draw, fill,  inner sep=0mm, minimum size=1.0ex]
\node[vertex]	(C)  	at (0,0) 	[label=above:{p}]               {};
\node[]	(dummyS)  	         	[below  of=C,label=above:{$\hdots$}]               {};

\node[vertex]	(E1)  	         	[left of=C,label=above:{$p-1$}]               {};
\node[vertex]	(D1)  	         	[left of=E1,label=above:{$1$}]               {};

\node[vertex]	(E2)  	         	[below of=E1,label=below:{}]               {};
\node[vertex]	(D2)  	         	[left of=E2,label=below:{}]               {};

\node[vertex]	(E4)  	         	[right of=C,label=above:{$p-1$}]               {};
\node[vertex]	(D4)  	         	[right of=E4,label=above:{$1$}]               {};

\node[vertex]	(E3)  	         	[below of=E4,label=below:{}]               {};
\node[vertex]	(D3)  	         	[right of=E3,label=below:{}]               {};

\draw [thick] (C)--(E1);
\draw [thick, dashed] (E1)--(D1);
\draw [thick] (C)--(E2);
\draw [thick, dashed] (E2)--(D2);
\draw [thick] (C)--(E3);
\draw [thick, dashed] (E3)--(D3);
\draw [thick] (C)--(E4);
\draw [thick, dashed] (E4)--(D4);
\end{tikzpicture}
\end{gathered}
\end{equation*}
The fundamental genus
is $h^1(\O_{\mathbf Z})=(p-2)(p+1)/2$, and ${\mathbf Z}^2=-p$. Moreover, the canonical cycle is $K=-(p-2){\mathbf Z}$.
\end{proposition}

\section{Numerically Gorenstein intersection matrices}
\mylabel{numerically gorenstein}

All wild ${\mathbb Z}/p{\mathbb Z}$-quotient singularities resolved in this article are hypersurface singularities. 
We prove in this section that all wild ${\mathbb Z}/2{\mathbb Z}$-quotient singularities are hypersurface singularities. 
We then recall  
that the intersection matrix associated with a hypersurface singularity is always numerically Gorenstein. 
We show in Proposition \ref{thm.p=2NumGor}  that 
any intersection matrix $N$
whose discriminant group $\Phi_N$ is killed by $2$ is automatically numerically Gorenstein. 
We exhibit in \ref{ex.notnumericallyGorenstein} an example when $p>2$ of a wild ${\mathbb Z}/p{\mathbb Z}$-quotient singularity which is not numerically Gorenstein.

\begin{proposition}\label{pro.p=2completeintersection}
 Let $p=2$. Let $A=k[[u,v]]$, endowed with a non-trivial action of $G={\mathbb Z}/2{\mathbb Z}$. Then there exists a power series ring $R:=k[[x,y]]$ such that $A^G$ is $k$-isomorphic to $R[z]/(z^2+sz+t)$, with $s,t \in R$.
\end{proposition}
\begin{proof} Let $\sigma$ denote the generator of $G$. Proposition 2.9 in \cite{LS1} allows us, if necessary, to replace the system of parameters $(u,v) $ for $A$
with a new system of parameters (again denoted by $(u,v)$ below) with the following properties (use  \cite[Proposition 2.3]{LS1}): 
let $x:=u \sigma(u)$ and $y:=v\sigma(v)$. Let $R:=k[[x,y]]$ be the subring of $A$ generated by $k, x$, and $y$. 
Then $A$ is a free $R$-module of rank $4$. 

We have the inclusions $R \subset A^G \subset A$, and the fraction field of $A^G$ is then of degree $2$ over the fraction field of $R$. Since $R$ is regular and $A^G$ is Cohen-Macaulay  because it is normal of dimension $2$, we find that $A^G$ is a free $R$-module of rank $2$. Thus, $R$ is a direct summand of $A^G$, with quotient $A^G/R$ free of rank $1$. 
We can therefore find an element $z \in A^G$ which generates the quotient  $A^G/R$. 
It follows that the natural map $R[Z] \to A^G$ with $Z \mapsto z$  is surjective. Since $z \notin R$, it satisfies a quadratic equation $z^2+sz+t=0$, with $s,t \in R$ and $Z^2+sZ+t$ irreducible in $R[Z]$.  Since $R[Z]$ is a UFD, we find that 
$R[Z]/(Z^2+sZ+t) \to A^G$ is an isomorphism.
\end{proof}

\begin{emp} \label{NumericallyGor} Let $N=(c_{ij}) \in \Mat_n(\ZZ)$ be an intersection matrix. 
Let $H_0 \in {\mathbb Z}^n$ be the integer vector whose $i$-th coefficient is $h_i:=-c_{ii}-2$ for $i=1,\dots,n$.
Since $N$ is invertible, there exist a vector $K \in {\mathbb Q}^n$ such that $NK=H_0$. The vector $K$ is called the {\it canonical cycle} of $N$.
We say that $N$ is
{\it numerically Gorenstein} if  $K \in{\mathbb Z}^n$.
\end{emp} 
When $N$ is the intersection matrix associated with a collection of irreducible curves $C_i$, $i=1,\dots,n$ on a surface, each component
$C_i$ has an arithmetical genus $p_a(C_i)$. Our definition of numerically Gorenstein coincides with the usual one (see for instance \cite{PPS}, (2.5))
when all arithmetical genera are equal to $0$. 
\if false
Otherwise, the canonical cycle is defined to take into account the $p_a(C_i)$'s: it is the vector $K$ with $NK=H$, where $H$ is such that $h_i:=-c_{ii}+2p_a(C_i)-2$. It is shown in \cite{PPS} that $K$ has non-negative coefficients (denoted by $K \geq 0$) when $c_{ii} \neq -1$ for all $i=1,\dots,n$, and that if in addition the intersection matrix is not
the matrix attached to a Dynkin diagram $A_n, D_n, E_6, E_7$, or $E_8$, then the coefficients of $K$ are strictly positive (denoted by $K > 0$).  
\fi

\begin{lemma} \label{Gor} Let $k$ be a field of characteristic $p$. Let $\R$ denote a complete 
local ring of dimension $2$, isomorphic to $k[[x,y,z]]/(f)$ for some $f \in (x,y,z)$, and formally smooth outside its closed point.
Let $X \to \Spec \R$ be a resolution of the singularity, with associated intersection matrix $N$. Assume that all the irreducible components in the exceptional locus of the resolution are smooth rational curves. Then $N$ is numerically Gorenstein.

\end{lemma} 
\proof We first use \cite{Artin 1969}, 3.8, to find an algebraic scheme $S$ over $k$ and a point $s \in S$ such that the completion of ${\mathcal O}_{S,s}$ is isomorphic to $\R$.
The ring ${\mathcal O}_{S,s}$ is Gorenstein since its completion $\R$ is (\cite{Eis}, 21.18). Thus there exists an open set $U$ of $S$, containing $s$, and such that $U$ is everywhere Gorenstein (\cite{G-M}, 1.5). It follows that $U$ has a canonical sheaf that is trivial. Consider a resolution $\pi:V \to U$ 
of the singularity $s \in U$. Then the canonical sheaf $K_V$ on $V$ is supported on the exceptional divisor of $\pi$. The adjunction formula for each irreducible component $E_i$ shows that $(K_V \cdot E_i)+(E_i \cdot E_i) = 2p_a(E_i)-2$. Since $K_V$ is equal to a linear combination of the $E_i$, 
we find that the intersection matrix $N$ of the exceptional locus is numerically Gorenstein.
\qed
\medskip

Let $N =(c_{ij}) \in \Mat_n(\ZZ)$ be an intersection matrix with discriminant group $\Phi_N$. As usual, denote by $e_1,\dots, e_n$ the standard basis of 
${\mathbb Z}^n$, and let $p_i$ denote the order of the class of $e_i$ in $\Phi_N$. For each $i=1,\dots,n$, let $R_i \in {\mathbb Z}^n$ denote the unique positive vector such 
that $NR_i=-p_ie_i$. Let $(R_i)_j$ denote the $j$-th coefficient of $R_i$, and define 
$$g_i:=\sum_{j=1}^n (R_i)_j (|c_{jj}|-2)= (^t\! R_i) H_0.$$
If the matrix $N$ is such that $ c_{jj}\leq -2$ for all $j=1,\dots, n$, then $g_i\geq 0$. 
 
\begin{lemma} \label{lem.numGor} Let $N$ be an intersection matrix. 
Then $^tK=(-g_1/p_1,\dots, -g_n/p_n)$. In particular, 
the matrix $N$ is numerically Gorenstein if and only if $p_i$ divides $g_i$ for each $i=1,\dots, n$.
\end{lemma}
\proof  
By hypothesis, we have $NK=H_0$ for some vector $K \in {\mathbb Q}^n$. It follows that
$^t\! R_i NK= -p_iK_i = g_i$, and we find that $K_i=-g_i/p_i$.  
 \qed
 
 \begin{proposition} \label{thm.p=2NumGor} 
 Let $N =(c_{ij}) \in \Mat_n(\ZZ)$ be an intersection matrix with discriminant group $\Phi_N$ killed by $2$. 
 Then $N$ is numerically Gorenstein.
 \end{proposition}
 \proof Our hypothesis implies that  $p_i=1$ or $2$, for all $i=1,\dots, n$. We use the criterion given in \ref{lem.numGor}: to show that $N$ is numerically Gorenstein, it suffices to show, for each $i$, that the integer $g_i$ is even when $p_i=2$. 
Assume then that $p_i=2$. Then by construction, 
$$^t\! R_i N R_i= -p_i(R_i)_i.$$
We now compute explicitly the term $^t\! R_i N R_i$ and obtain
$$^t\! R_i N R_i= \sum_{j=1}^n c_{jj} (R_i)_j^2 + 2 \sum_{j< k} c_{jk} (R_i)_j (R_i)_k.$$
Since $p_i $ is even and $(R_i)_j^2 \equiv (R_i)_j \pmod{2}$, we find that $\sum_{j=1}^n c_{jj} (R_i)_j$ is even, and so is $g_i$, as desired. \qed

\if false
We claim that the matrix $N$ with graph $\Gamma_N$ can be extended to an arithmetical graph
 $M$ with graph $\Gamma_M$ as follows. 
 
 If the coefficient $(R_i)_i$ is even, then the graph $\Gamma_M$ is obtained by attaching  to the vertex $v_i$ of $\Gamma_N$, by a single edge, a unique vertex $v_0$. We set the coefficient $M_{00}$ to be $-(R_i)_i/2$. We let $S$ denote the integer vector with $S_j=(R_i)_j$ if $j>0$, 
 and we set $S_0= 2$. It follows that $MS=0$, and then by definition $(\Gamma_M, M,S)$ is an arithmetical graph.
 
 If the coefficient $(R_i)_i$ is odd, then the graph $\Gamma_M$ is obtained by attaching  to the vertex $v_i$ of $\Gamma_N$, by a single edge, a vertex $v_0$, and then attaching to $v_0$, again with a single edge,
  a vertex $v_{-1}$. We set the coefficient $M_{00}$ to be $-((R_i)_i+1)/2$, and we set the coefficient $M_{-1,-1}$ to be $-2$.
  We let $S$ denote the integer vector with $S_j=(R_i)_j$ if $j>0$, 
 and we set $S_0= 2$ and $S_{-1}=1$. It follows that $MS=0$.
 
  We then use the easy fact that for an arithmetical graph $(\Gamma_M, M,S)$ on $s$ vertices (where by definition $S \in {\mathbb Z}^s$ and  $MS=0$), the sum
  $$\sum_{j=1}^s S_j (|M_{jj}|-2)$$ is always an even integer (\cite{Lor1989}, page 490). 
  
 When the coefficient $(R_i)_i$ is even, we then find that 
 $$\sum_{j=1}^n (R_i)_j (|N_{jj}|-2) + S_0 (|M_{00}|-2) = \sum_{j=1}^n (R_i)_j (|N_{jj}|-2) + (R_i)_i -4$$ 
 is even. It follows that $g_j=\sum_{j=1}^n (R_i)_j (|N_{jj}|-2)$ is also even.
 
 When the coefficient $(R_i)_i$ is odd, we then find that 
 $$\sum_{j=1}^n (R_i)_j (|N_{jj}|-2) + S_0 (|M_{00}|-2) + S_{-1} (|M_{-1-1}|-2)= \sum_{j=1}^n (R_i)_j (|N_{jj}|-2) + ((R_i)_i+1) -8 $$ 
 is even. It follows that $g_j=\sum_{j=1}^n (R_i)_j (|N_{jj}|-2)$ is also even.
 \qed
\fi 
\begin{remark} 
Let $N =(c_{ij}) \in \Mat_n(\ZZ)$ be an intersection matrix associated with the resolution of a hypersurface singularity, all of whose exceptional components are smooth rational curves. Assume that $c_{ii}\leq -2$ for all $i=1,\dots,n$.   
Laufer in \cite{Lau}, 3.7, provides additional constraints on the canonical vector $K$ associated with such $N$, with an improvement
by M. Tomari stated in the Addendum on page 496. A further improvement was found by Yau in \cite{Yau}, Theorems B and C,
which show that for such $N$, 
$$g_i/p_i \geq  (|{\mathbf Z} \cdot {\mathbf Z}|-2)z_i,$$
where $^t{\mathbf Z}=(z_1,\dots,z_n)$ is the fundamental cycle of $N$. In other words, we have $-K \geq (|{\mathbf Z} \cdot {\mathbf Z}|-2){\mathbf Z}$. Note that the singularity in Proposition \ref{pro.NumericallyGorensteinEquality} satisfies $-K = (|{\mathbf Z} \cdot {\mathbf Z}|-2){\mathbf Z}$.

In the context of wild ${\mathbb Z}/2{\mathbb Z}$-quotient singularities treated in this article, 
the resolution of such a singularity  has intersection matrix $N$ with $\Phi_N$ killed by $2$ and with $|{\mathbf Z} \cdot {\mathbf Z}|\leq 2$. 
Proposition \ref{thm.p=2NumGor} shows that any such $N$ is always numerically Gorenstein, and since $|{\mathbf Z} \cdot {\mathbf Z}|\leq 2$ and ${\mathbf Z}>0$, Laufer's constraints are also automatically satisfied. 
\end{remark}
 
\begin{example}  \label{ex.notnumericallyGorenstein}
We exhibit below a wild ${\mathbb Z}/p{\mathbb Z}$-quotient singularity that is not numerically Gorenstein. 
For this, note first the following. Suppose that $N=(c_{ij})$ is an intersection matrix such that $\Phi_N$ is killed by $p$, and such that $c_{ii} =-2$
except for a unique vertex $C$ with $2< |(C \cdot C)|$ and $\gcd(p, |C\cdot C|-2)=1$.
{\it Then $N$ is numerically Gorenstein if and only if the class of $e_C$ is trivial in $\Phi_N$}. Indeed, if the class of $e_C$ is trivial in $\Phi_N$,
then there exists an integer vector $K_C$ with $NK_C=e_C$. It follows that $K:=(|C\cdot C|-2)K_C$ is such that $NK=H_0$.
If the class of $e_C$ is not trivial, then it must have order $p$, and so there exists an integer vector $R$ such that $NR=pe_C$. 
If there also exists an integer vector $K$ with $NK=H_0= (|C\cdot C|-2)e_C$, then
the class of $e_C$ in $\Phi_N$ has order dividing $|C\cdot C|-2$, which is a contradiction since $\gcd(p, |C\cdot C|-2)=1$. 

Let $p>2$ be prime and consider now the wild ${\mathbb Z}/p{\mathbb Z}$-quotient singularity
 in \cite{Lorenzini 2014}, 6.8, with resolution graph with $r_1(i)=1$. This resolution graph
 has a single vertex of self-intersection different from $-2$, namely
the terminal vertex $C$ with $r_1(i)=1$ and self-intersection $-p$, represented as the top center vertex in the graph below. 
The class of $e_C$ is not trivial in $\Phi_N$, since the vector $R$ whose multiplicities are indicated below is such that 
 $NR=-pe_C$
and    the greatest common divisor of the coefficients of $R$ is equal to $1$. 
Then the above argument 
shows that $N$ cannot be numerically Gorenstein. In fact, the canonical vector $K$ is $-(p-2)R/p$. The fundamental cycle is given in 
\cite{Lorenzini 2018}, 4.4, and it is shown in \cite{Lorenzini 2018}, 4.1, that this singularity is rational.
$$
\begin{tikzpicture}
[node distance=1cm, font=\small] 
\tikzstyle{vertex}=[circle, draw, fill, inner sep=0mm, minimum size=1.1ex]
\node[vertex]	(v1)  	at (0,0) 	[label=below:{$1$}] 		{};
\node[]	(dummy1)		[right of=v1, label=above:{}]	{};
\node[vertex]	(v2)			[right of=dummy1, label=below:{$p-1$}]	{};
\draw [decorate,decoration={brace, raise=6pt}] (v1)--(v2) node [black, midway,xshift=-0cm,yshift=0.5cm] {$^{p-1}$};
\node[vertex]	(v3)			[right of=v2, label=below:{$p$}]	{};
\node[vertex]	(v)			[above of=v3,  label=above:{ ${2}$}, label=left:{ ${-p}$}]	{};
\node[vertex]	(v4)			[right of=v3, label=below:{$p-1$}]	{};
\node[]	(dummy2)		[right of=v4, label=above:{}]	{};
\node[vertex]	(v5)			[right of=dummy2, label=below:{$1$}]	{};
\draw [decorate,decoration={brace, raise=6pt}] (v4)--(v5) node [black, midway,xshift=-0cm,yshift=0.5cm] {$^{p-1}$};
\draw [thick, dashed] (v1)--(v2);
\draw [thick] (v2)--(v3)--(v4);
\draw [thick] (v)--(v3);
\draw[thick, dashed] (v4)--(v5);
\end{tikzpicture}
$$

\end{example}

\if false
\begin{example} \label{ex.Phi=2}
\if false
Let $N \in \Mat_n(\ZZ)$ be an intersection matrix such that $|\Phi_N|=2$. Let $Z$ denote its fundamental cycle. 
We leave it to the reader to check the following facts:

(a) The intersection matrix $N$ is numerically Gorenstein if and only if 
the graph of $N$ has an even number  of vertices $C$ whose self-intersection $(C \cdot C)$ is odd and whose class is not trivial in the group $\Phi_N$. 

(b) If the vertex $C$ is a terminal vertex on a terminal chain of the graph, and if the class of $C$ is trivial in $\Phi_N$, then the class of every vertex on the terminal chain, including the node to which it is attached, is trivial in $\Phi_N$.

(c) Label by $C_1, C_2,\dots$ the vertices of the graph, and suppose that $C_1$ is a terminal vertex on a chain of the graph, and that $C_2$ is the unique vertex 
of the graph to which $C_1$ is attached. Suppose given $R \in {\mathbb Z}^n$ such that $NR=-e_1$ and such that the first coefficient of $R$ is $2$. 
Then the class of $e_1$ is trivial in $\Phi_N$, and $|(Z \cdot Z)| \leq 2$. Note that the equality $NR=-e_1$ proves that $N$ is negative definite as soon as $\det(N) \neq 0$, and that $|(Z \cdot Z)| \leq |(R \cdot R)| \leq 2$.

The above facts can be used to obtain many intersection matrices $N$  with the three properties in \ref{three properties} (b) and $|\Phi_N|=2$, and which are numerically Gorenstein. 
\fi
We present in this example an infinite family of intersection matrices $N_t$, $t \geq 2$,
with $|\Phi_N|=2$, whose associated graph is a tree, whose fundamental cycle $Z$ is such 
that $|Z^2|\leq 2$, and which have a distinguished vertex. When $t=2$, the matrix $N_1$ corresponds to the Dynkin diagram $E_7$. It is natural to wonder whether a matrix $N_t$ with $t>2$ can occur as the intersection matrix attached to the resolution of a wild 
${\mathbb Z}/2{\mathbb Z}$-quotient singularity.

For each integer $t \geq 2$, the matrix $N_t$ has one single node and  three terminal chains.  
Denote by  $A$, $B$, and $C$, the three terminal vertices. 
The central node $D$ has self-intersection $-2$. 
The terminal chain started by $A$ has $2^t-2$ vertices of self-intersection $-2$. The terminal chain started by $B$ has one vertex only, namely $B$, with self-intersection $-2$. The terminal chain started by $C$ has three vertices only, namely $C$, with self-intersection $-2$, linked to $C_1$, with self-intersection $-(2^{t-2}+1)$, itself linked to $C_2$, with self-intersection $-2$. The vertex $C_2$ is linked to the central node  $D$.

Consider the vector $R$ with coefficient $2^t$ on the node $D$, 
$2^{t-1}$ on $B$, $1$ on $C$, $2$ on $C_1$, and $2^{t-1}+1$ on $C_2$. Let $2$ be the coefficient of $A$, and on the chain from $A$ to the node, let the coefficients increase: $2,3,\dots, 2^{t}-1$.  It is easy to check that $NR=-e_A$, where $e_A$ is the standard basis vector corresponding to $A$. In particular, $|(Z \cdot Z)| \leq |(R \cdot R)| \leq 2$, where as usual, $Z$ denotes the fundamental cycle of $N$. Theorem 3.14 in \cite{Lorenzini 2013} shows then that $|\Phi_N|=2$. Since  $NR=-e_A$, the class of $A$ is trivial in $\Phi_N$. It follows that the classes of all vertices on the chain started by $A$ are trivial, including the class of the node. Thus, using \ref{rem.trivialdistinguished}, we find that all these vertices are distinguished. 
\if false
Note that both the classes of $B$ and $C_2$ are not trivial (and equal) in $\Phi_N$. 
Note also that the matrix $N$ has all its self-intersections equal to $-2$, except at the vertex $C_1$. 

Let us now build an intersection matrix $N'_t$ with the above four properties, and such that the graph of $N'_t$ has two nodes. We leave it to the reader to modify this construction to obtain graphs with many nodes. To build $N_t'$, we consider the graph of $N_t$, and glue to it at the vertex $C$ an arithmetical graph with trivial component group. There are many ways of doing so, and we exhibit here the most straightforward choice. The graph of $N'_t$ is the graph of $N_t$ to which 
we attach at $C$ a vertex $E_1$ of self-intersection $-1$, and to $E_1$, we attach two vertices $E_2$ and $E_3$, of self-intersection $-2$ and $-3$, respectively. 
Thus, $E_1$ is a second node on the graph of $N'_t$. In $N'_t$, the self-intersection of $C$ is set to be $-4$. All other self-intersections of $N'_t$ are as those in $N_t$. Consider the vector $R'$ such that $N_t'R'=-e_A$, with the coefficient of $E_1$ being $6$, and the coefficients of $E_2$ and $E_3$ being $3$ and $2$, respectively. The other coefficients are the same as in $R$. The computation of $\Phi_{N'_t}$ follows as in the star-shaped case, and similarly for the inequality $|(Z' \cdot Z')| \leq |(R' \cdot R')| \leq 2$. 
\fi
\end{example}
\fi 

\if false
Suppose given an intersection matrix $N$ whose group $\Phi_N$ is killed by $p$, whose associated graph is a tree, and  whose fundamental cycle $Z$ is such 
that $|Z^2|\leq p$.  It is natural to wonder, in view of the results \ref{emp.graphtree}, \ref{emp.discriminantgroup}, and \ref{emp.fundamentalcycle},
whether there exists a wild ${\mathbb Z}/p{\mathbb Z}$-quotient singularity $\Spec A^{{\mathbb Z}/p{\mathbb Z}}$ with $\Spec A \to \Spec A^{{\mathbb Z}/p{\mathbb Z}}$ ramified precisely at the maximal ideal, and whose resolution of singularities has an associated matrix equal to $N$, up to permutation.  When the quotient singularity $A^{{\mathbb Z}/p{\mathbb Z}}$ comes from a moderately ramified action as in \ref{initial blow-up}, we found in \ref{cor.distinguishedvertex} that the answer to this question can only be positive if in addition $N$ also has a distinguished vertex. 

It turns out that the four combinatorial conditions on $N$  mentioned above are not sufficient to insure a positive answer to this question  when $p=2$. Indeed, Artin \cite{Artin 1977} shows that the Dynkin diagram $E_7$, with group $\Phi_{E_7} ={\mathbb Z}/2{\mathbb Z}$ and $|Z^2|=2$, cannot be obtained 
as the resolution graph of a wild ${\mathbb Z}/2{\mathbb Z}$-quotient singularity ramified precisely at the origin, even though its intersection matrix satisfies all four of the combinatorial properties listed above. The case of $E_7$ is further discussed in next section in \ref{emp.E7}.

When $p=2$, we have not been able to exhibit any wild ${\mathbb Z}/2{\mathbb Z}$-quotient singularities ramified precisely at the origin and whose associated intersection matrices have determinant of order $2^s$ with $s$ odd. (See for instance \cite{Lor2010}, 4.3, for a related question on a specific example where $s=3$). It is thus natural to wonder whether there is a structural obstruction that would explain why such example cannot be exhibited.
Artin \cite{Artin 1977} shows that the Dynkin diagram $E_7$, with group $\Phi_{E_7} ={\mathbb Z}/2{\mathbb Z}$ and $|Z^2|=2$, cannot be obtained 
as the resolution graph of a wild ${\mathbb Z}/2{\mathbb Z}$-quotient singularity ramified precisely at the origin. He shows however that 
it does occur  

We discuss in this section an additional combinatorial property that the intersection matrix associated with the resolution of a moderately ramified ${\mathbb Z}/p{\mathbb Z}$-quotient singularity is likely to satisfy. 
Unfortunately, this additional property does not impose any new condition when $p=2$, as we show in \ref{thm.p=2NumGor}.
 \fi


\end{document}